\documentclass[12pt,draft,oneside]{article}
\usepackage[cp1251]{inputenc}
\usepackage[english,russian,ukrainian]{babel}
\usepackage{amsmath}
\usepackage{amsthm}
\usepackage{amssymb}
\usepackage{amsfonts}
\usepackage{latexsym}
\usepackage{calc}
\usepackage{indentfirst}
\newtheorem*{proposition*}{Proposition}
\newtheorem{theorem}{Theorem}
\newtheorem*{theorem*}{Theorem}
\newtheorem{remark}{Remark}

\newtheorem*{definition*}{Definition}
\newtheorem{lemma}{Lemma}
\newtheorem*{lemma*}{Lemma}
\newtheorem{corollary}{Corollary}
\newtheorem*{corollary*}{Corollary}
\newtheorem{example}{Example}
\newcommand{\sep}{/\kern-2pt/ }

\allowdisplaybreaks

\begin{document}
	
 \begin{center}
	\Large\bf SOME PROPERTIES OF ANALYTIC IN A BALL FUNCTIONS OF BOUNDED $\mathbf{L}$-INDEX IN JOINT VARIABLES
\end{center}

\begin{center}
	\large\bf \MakeUppercase{A.\ I.\ BANDURA, O.\ B.\ SKASKIV}
\end{center}

\selectlanguage{english}
\vspace{20pt plus 0.5pt} {\abstract{ \noindent A.\ I.\ Bandura,  O.\ B.\ Skaskiv, \ 
		\textit{Some properties of analytic  in a ball  functions of bounded $\mathbf{L}$-index in joint variables} \vspace{3pt} 
		
	A concept of boundedness of the $\mathbf{L}$-index in joint variables (see in Bandura A. I., Bordulyak M. T., Skaskiv 
	O. B. {\it Sufficient conditions of boundedness of L-index
		in joint variables,} Mat. Stud. \textbf{45} (2016), 12--26. dx.doi.org/10.15330/ms.45.1.12-26) is generalized
	for analytic in a ball function.
	There are proved criteria of boundedness of the $\mathbf{L}$-index in joint variables which describe local behavior of 
	partial derivatives and maximum modudus on a skeleton of a polydisc, properties of 
	power series expansion. Also we obtained analog of Hayman's Theorem.
	
	As a result, they are applied to study linear higher-order systems of partial differential equations 
	and to deduce sufficient conditions of boundedness of the $\mathbf{L}$-index in joint variables for their 
	analytic solutions and to estimate it growth. 
	
	We used an exhaustion of ball in $\mathbb{C}^n$ by polydiscs. Also growth estimates of analytic in ball functions of 
	 bounded $\mathbf{L}$-index in joint variables are obtained. 
	Note that this paper and paper "Analytic functions in a bidisc of bounded L-index in joint variables" (arXiv:1609.04190) do not overlap 
	because analytic in a ball and analytic in a bidisc functions are different classes of holomorphic functions. 
	Some results have similar formulations but a bidisc and a ball are particular geometric objects.
}}



\vskip10pt

\section{Introduction} 
 A concept of entire function of bounded index appeared in a paper of B. Lepson \cite{lepson}. 
  An entire function $f$ is said
to be of bounded index if there exists an integer $N > 0$ that
\begin{equation}\label{deflepson}
(\forall z\in\mathbb{C})(\forall n\in\{0,1,2,\ldots\})\colon\
\frac{|f^{(n)}(z)|}{n!}\leq\max\Big\{\frac{|f^{(j)}(z)|}{j!}\colon
0\leq j\leq N \Big\}.
\end{equation}
The least such integer $N$ is called the index of $f.$
 
 Note that the functions from this class have interesting properties. 
 The concept is convenient to study the properties of entire solutions of differential equations. 
   In particular, if an entire solution has bounded index then it immediately yields
  its growth estimates, an uniform in a some sense distribution of its zeros, a certain regular behavior of the 
solution, etc.
 
Afterwards, S. Shah \cite{shah} and W. Hayman \cite{Hayman}
independently   proved that every entire function of bounded index
is a function of exponential type. 
 Namely, its growth is at most the first order and normal type.
 
 To study more general entire functions, A.~D. Kuzyk and M.~M. Sheremeta \cite{vidlindex}
introduced a boundedness of the $l$-index, replacing $ \frac{ |f^{(p)}(z)|}{p!}$ on $ \frac{ 
|f^{(p)}(z)|}{p!l^{p}(|z|)}$ in 
\eqref{deflepson}, where $l: \mathbb{R}_{+}\to \mathbb{R}_{+}$  is
a continuous function.
  It allows to consider an arbitrary entire function $f$ with bounded multiplicity of zeros. 
 Because for the function $f$ there exists a positive continuous function $l(z)$ 
 such that $f$ is of bounded $l$-index \cite{bordproof}.
 Besides, there are papers where the definition of bounded $l$-index is generalizing for analytic 
function of one variable \cite{strosher,kusher}.
 
In a multidimensional case a situation is more difficult and interesting. 
Recently we with N. V. Petrechnko \cite{petrechko1}-\cite{petrechkopower} proposed approach to 
consider bounded $\mathbf{L}$-index in joint variables for analytic in a polydisc functions, 
where $\mathbf{L}(z)=(l_1(z),$ $\ldots,$ $l_{n}(z)),$ $l_j: \mathbb{C}^n \to \mathbb{R}_+$ is 
a positive continuous functions, $j\in\{1,\ldots,n\}.$
Although J.~Gopala Krishna and S.M.~Shah \cite{krishna} introduced an analytic in a domain (a 
nonempty connected open set) $\Omega\subset \mathbb{C}^n$ $(n\in\mathbb{N})$ function of bounded index for 
$\alpha=(\alpha_1,\ldots,\alpha_n)\in \mathbb{R}^n_{+}.$ 
But analytic in a domain function of bounded index
by Krishna and Shah is an entire function. It follows from necessary condition
of the $l$-index boundedness for analytic in the unit disc function (\cite[Th.3.3,p.71]{sher}): 
$\int_0^rl(t)dt\to\infty$ as $r\to 1$ (we take $l(t)\equiv \alpha_1$). 
Thus, there arises necessity to introduce and to investigate bounded $\mathbf{L}$-index in joint variables for analytic 
in polydisc domain functions. 
Above-mentioned paper \cite{petrechko1} is devoted analytic in a polydisc functions. 
Besides a polydisc, other example of polydisc domain in $\mathbb{C}^n$ is a ball. 
There are two known monographs \cite{kehe-zhu,rudin-ball} about spaces of holomorphic functions in the unit ball of 
$\mathbb{C}^n:$ Bergman spaces, Hardy spaces, Besov spaces, Lipschitz spaces, the Bloch space, etc. 
 It shows the relevance of research of properties of holomorphic function in the unit 
ball. 
In this paper we will introduce and study analytic in a ball functions of bounded $\mathbf{L}$-index in 
joint variables.
Of course, there are wide bibliography about entire functions of bounded $\mathbf{L}$-index in 
joint variables \cite{sufjointdir}-\cite{ball-exhaustion}, \cite{bagzmin}-\cite{Chakraborty2001}, \cite{nuraypattersonexponent2015}-\cite{nuraypattersonmultivalence2015}.

Note that there exists other approach to consider bounded index in $\mathbb{C}^n$ --- so-called 
functions of bounded $L$-index in direction (see \cite{BandSk}-\cite{unbound-each-direct}), where $L: 
\mathbb{C}^n\to \mathbb{R}_+$ is a positive 
continuous function.

\section{Main definitions and notations}

We need some standard notations. Denote $\mathbb{R}_+=[0,+\infty),$
$\mathbf{0}=(0,\ldots,0)\in\mathbb{R}^n_{+},$
$\mathbf{1}=(1,\ldots,1)\in\mathbb{R}^n_{+},$ 
$\mathbf{1}_j=(0,\ldots,0, \underbrace{1}_{j-\mbox{th place}}, 0,\ldots,0)\in\mathbb{R}^n_{+},$
 $R=(r_1,\ldots,r_n)\in\mathbb{R}^n_+,$ $z=(z_1,\ldots,z_n)\in\mathbb{C}^n,$
 $|z|=\sqrt{\sum_{j=1}^n|z_j|^2}.$
 For $A=(a_1,\ldots,a_n)\in\mathbb{R}^n,$ $B=(b_1,\ldots,b_n)\in\mathbb{R}^n$ we will use formal notations 
without violation of the existence of these expressions 
 $AB=(a_1b_1,\cdots,a_nb_n),$
$A/B=(a_1/b_1,\ldots,a_n/b_n),$ 
$A^B=a_1^{b_1}a_2^{b_2}\cdot \ldots \cdot a_n^{b_n},$ $\displaystyle\|A\|=a_1+\cdots+a_n,$\
 and the notation $A<B$ means that $a_j<b_j,$ $j\in\{1,\ldots,n\};$ the relation $A\leq B$ is defined similarly.
For $K=(k_1,\ldots,k_n)\in \mathbb{Z}^n_{+}$ denote
$K!=k_1!\cdot \ldots \cdot k_n!.$
Addition, scalar multiplication, and conjugation are defined on $\mathbb{C}^n$ componentwise. 
For $z\in\mathbb{C}^n$ and $w\in\mathbb{C}^n$ we define 
$$
\langle z,w\rangle =z_1\overline{w}_1+\cdots+z_n\overline{w}_n,
$$
where $w_k$ is the complex conjugate of $w_k.$
The polydisc $\{z\in\mathbb{C}^n: \ |z_j-z_j^0|<r_j, \ j=1,\ldots, n\}$ is denoted by $\mathbb{D}^n(z^0,R),$ its skeleton $\{z\in\mathbb{C}^n: \ |z_j-z_j^0|=r_j, \ j=1,\ldots, n\}$ is denoted by $\mathbb{T}^n(z^0,R),$ and the closed polydisc
$\{z\in\mathbb{C}^n: \ |z_j-z_j^0|\leq r_j, \ j=1,\ldots, n\}$ is denoted by $\mathbb{D}^n[z^0,R],$
 $\mathbb{D}^n=\mathbb{D}^n(\mathbf{0},\mathbf{1}),$
 $\mathbb{D}=\{z\in\mathbb{C}: \ |z|<1\}.$ 
The open ball $\{z\in\mathbb{C}^n: \ |z-z^0|<r\}$ is denoted by $\mathbb{B}^n(z^0,r),$ 
its boundary is a sphere $\mathbb{S}^n(z^0,r)=\{z\in\mathbb{C}^n: \ |z-z^0|=r\},$ 
the closed ball  $\{z\in\mathbb{C}^n: \ |z-z^0|\leq r\}$ is denoted by $\mathbb{B}^n[z^0,r],$ 
 $\mathbb{B}^n=\mathbb{B}^n(\mathbf{0},1),$
$\mathbb{D}=\mathbb{B}^1=\{z\in\mathbb{C}: \ |z|<1\}.$

For $K=(k_1,\ldots,k_n)\in \mathbb{Z}^n_{+}$ and the  partial derivatives of an analytic in $\mathbb{B}^n$ function $F(z)=F(z_1,\ldots,z_n)$ we  use the notation
$$
F^{(K)}(z)=\frac{\partial^{\|K\|} F}{\partial z^{K}}= \frac{\partial^{k_1+\cdots+k_n}f}{\partial z_1^{k_1}\ldots \partial z_n^{k_n}}.
$$
 Let $\mathbf{L}(z)=(l_1(z),\ldots, l_{n}(z)),$ where $l_j(z): \mathbb{B}^n\to \mathbb{R}_+$ is a continuous function such that
\begin{equation} \label{Lbeta-ball}
(\forall z\in \mathbb{B}^n)\colon\ l_j(z)>{\beta}/{(1-|z|)}, \ j\in\{1,\ldots,n\},
\end{equation}
where $\beta>\sqrt{n}$ is a some constant. 

S.N. Strochyk, M.M. Sheremeta, V.O. Kushnir \cite{strosher}--\cite{sher} imposed a similar condition for a function $l: \mathbb{D}\to \mathbb{R}_+$ and $l: G\to \mathbb{R}_+,$ where $G$ is arbitrary domain in $\mathbb{C}$.

\begin{remark}
Note that if $R\in \mathbb{R}^n_+,$ $|R|\leq \beta,$\  $z^0\in \mathbb{B}^n$ and $z \in \mathbb{D}^n[z^0,R/\mathbf{L}(z^0)]$ then $z\in \mathbb{B}^n.$  
 Indeed,  we have 
 $|z|\leq |z-z^0|+|z^0|\leq 
 \sqrt{\sum_{j=1}^n \frac{r_j^2}{l_j^2(z^0)}}+|z^0| < 
 \sqrt{\sum_{j=1}^n \frac{r_j^2}{\beta^2}(1-|z^0|)^2}+|z^0| = 
 \frac{(1-|z^0|)}{\beta}\sqrt{\sum_{j=1}^n r_j^2}+|z^0| 
 \leq  \frac{(1-|z^0|)}{\beta} \beta +|z^0| =1.$
\end{remark}

 An analytic function $F\colon\mathbb{B}^n\to\mathbb{C}$ is said to be of \textit{bounded $\mathbf{L}$-index (in joint variables),} if there exists $n_0\in \mathbb{Z}_+$ such that for all $z\in \mathbb{B}^n$ and for all $J\in\mathbb{Z}^n_+$
\begin{equation} \label{ineqoz2}
\frac{|F^{(J)}(z)|}{J!\mathbf{L}^{J}(z)}\leq\max
\left\{\frac{|F^{(K)}(z)|}{K! \mathbf{L}^{K}(z)}:\
K\in\mathbb{Z}^{n}_{+},\ \|K\|\leq n_0\right\}.
\end{equation}
 The least such integer $n_{0}$
is called the {\it $\mathbf{L}$-index in joint variables of the function $F$} and is denoted by $N(F,\mathbf{L},\mathbb{B}^n)$ (see \cite{sufjointdir}--\cite{prostir}).

By $Q(\mathbb{B}^n)$ we denote the class of functions $\mathbf{L}$, which satisfy \eqref{Lbeta-ball} and the following condition
\begin{equation} \label{clasqballn}
(\forall R\in \mathbb{R}^n_+,  |R|\leq \beta,\  j\in\{1,\ldots,n\})\colon \ 0<\lambda_{1,j}(R)\leq \lambda_{2,j}(R)<\infty,
\end{equation}
 where
\begin{gather} \label{lam1}
\lambda_{1,j}(R)=\inf\limits_{ z^0\in \mathbb{B}^n} \inf \left \{
\frac{l_j(z)}{l_j(z^0)}: z\in \mathbb{D}^n\left[z^0, {R}/{\mathbf{L}(z^0)}\right]\right\},\\
\lambda_{2,j}(R)=\sup\limits_{z^0\in \mathbb{B}^n} \sup \left \{
\frac{l_j(z)}{l_j(z^0)}: z\in \mathbb{D}^n\left[z^0, {R}/{\mathbf{L}(z^0)}\right]\right\}.\label{lam2}\\
\Lambda_1(R)=(\lambda_{1,1}(R),\ldots,\lambda_{1,n}(R)), \ \Lambda_2(R)=(\lambda_{2,1}(R),\ldots,\lambda_{2,n}(R)).
\end{gather}
 
 It is not difficult to verify that the class $Q(\mathbb{B}^n)$ can be defined as following:
 \begin{equation} \label{umfavorovball}
\text{ for every } j\!\in\!\{1,\ldots,n\} \  \sup_{z,w\in\mathbb{B}^n} \left\{ \frac{l_j(z)}{l_j(w)}\colon |z_k-w_k| \leq \frac{r_k}{\min\{l_k(z),l_k(w)\}},  k\in\{1,\ldots,n\}\right\}\!<\!\infty,
 \end{equation}
 i. e.  conditions \eqref{clasqballn} and \eqref{umfavorovball} are equivalent.
 
 \begin{example}\rm
The function $F(z)=\exp\{ \frac{1}{(1-z_1)(1-z_2)}\}$ has bounded $\mathbf{L}$-index in joint variables with
 $\mathbf{L}(z)=\big(\frac{1}{(1-|z_1|)^2(1-|z|)},\frac{1}{(1-|z|)(1-|z_2|)^2}\big)$ and 
$N(F,\mathbf{L},\mathbb{B}^n)=0.$
\end{example}
 \section{Local behavior of derivatives of function of bounded $\mathbf{L}$-index in joint variables}
 
 The following theorem is basic in theory of functions of bounded index. It was necessary to prove more efficient criteria of index boundedness 
 which describe a behavior of maximum modulus on a disc or a behavior of logarithmic derivative (see \cite{kusher,sher,BandSk,monograph}). 
\begin{theorem} \label{petr1-ball}
	Let $\mathbf{L} \in Q(\mathbb{B}^n)$. An analytic in $\mathbb{B}^n$ function $F$ has bounded $\mathbf{L}$-index in joint variables if and only if for each
	$R\in\mathbb{R}^n_+,$ $|R|\leq \beta,$ there exist $n_0\in \mathbb{Z}_+$, $p_0>0$ such that for every $z^0 \in\mathbb{B}^n$ there exists $K^0\in \mathbb{Z}_+^n$, $\|K^0\|\leq n_0$, and
	\begin{gather}
	\max\left\{\frac{|F^{(K)}(z)|}{K!\mathbf{L}^K(z)}\colon  \|K\| \leq n_0, \ z\in \mathbb{D}^n\left[z^0, {R}/{\mathbf{L}(z^0)}\right] \right\} \leq 
	p_0 \frac{|F^{(K^0)}(z^0)|}{K^0!\mathbf{L}^{K^0} (z^0)}.
	\label{net1}
	\end{gather}
\end{theorem}

\begin{proof}
	Let $F$ be of bounded $\mathbf{L}$-index in joint variables with $N=N(F,\mathbf{L},\mathbb{B}^n)<\infty.$
	For every 	$R\in\mathcal{R}^n_+,$ $|R|\leq \beta,$ we put
	$$q= q(R)= [2(N+1)\|R\|\prod_{j=1}^{n}(\lambda_{1,j}(R))^{-N}(\lambda_{2,j}(R))^{N+1}]+1$$
	where $[x]$ is the entire part of the real number $x,$ i.e. it is a floor function.
	For $p\in\{0,\ldots,q\}$ and $z^0\in\mathbb{B}^n$ we denote
	\begin{gather*}
	S_p(z^0,R)=\max\left\{\frac{|F^{(K)}(z)|}{K! \mathbf{L}^{K} (z)}: \|K\|\leq N, z\in \mathbb{D}^n \left[z^0,\frac{pR}{q\mathbf{L}(z^0)}\right] \right\},
	\\
	S^{*}_p(z^0,R)=\max\left\{\frac{|F^{(K)}(z)|}{K! \mathbf{L}^{K} (z^0)}: \|K\|\leq N, z\in \mathbb{D}^n \left[z^0,\frac{pR}{q\mathbf{L}(z^0)}\right] \right\}.
	\end{gather*}
	Using \eqref{lam1} and $\mathbb{D}^n\left[z^0,\frac{pR}{q\mathbf{L}(z^0)}\right] \subset \mathbb{D}^n \left[z^0,\frac{R}{\mathbf{L}(z^0)}\right] $, we have
	\begin{gather*}
	S_p(z^0,R)=\!\max\left\{\frac{|F^{(K)}(z)|}{K! \mathbf{L}^{K} (z)}\frac{\mathbf{L}^{K} (z^0)}{\mathbf{L}^{K} (z^0)}: \|K\|\leq \!N,
	z\in\! \mathbb{D}^n\! \left[z^0,\frac{pR}{q\mathbf{L}(z^0)}\!\right]\! \right\}\! \leq\! \nonumber \\ \leq S^{*}_p(z^0,R)\max\left\{ \prod_{j=1}^n \frac{l_j^{N} (z^0)}{l_j^{N} (z)}: z\in \mathbb{D}^n \left[z^0,\frac{pR}{q\mathbf{L}(z^0)}\right] \right\} \leq S^{*}_p(z^0,R)\prod_{j=1}^{n} (\lambda_{1,j}(R))^{-N}.
	\end{gather*}
	and, using \eqref{lam2}, we obtain
	\begin{gather}
	S^{*}_p(z^0,R)= \max\left\{\frac{|F^{(K)}(z)|}{K!\mathbf{L}^{K} (z)}\!\frac{\mathbf{L}^{K} (z)}{\mathbf{L}^{K} (z^0)}: \|K\|\leq N,
	z\in \mathbb{D}^n \left[z^0,\frac{pR}{q\mathbf{L}(z^0)}\right] \right\} \leq \nonumber \\
	\! \leq\! \max\left\{\frac{|F^{(K)}(z)|}{\!K! \mathbf{L}^{K} (z)}\!(\Lambda_{2}(R))^{K}: \|K\|\leq \!N,
	z\in \mathbb{D}^n \left[z^0,\frac{pR}{q\mathbf{L}(z^0)}\right] \right\} \! \leq\! S_p(z^0,R)\prod_{j=1}^{n} (\lambda_{2,j}(R))^N.
	\label{netss2}
	\end{gather}
	Let $K^{(p)}$ with $\| K^{(p)}\|\leq N$ and $z^{(p)} \in \mathbb{D}^n \left[z^0,\frac{pR}{q\mathbf{L}(z^0)}\right]$ be such that
	\begin{gather}
	S_p^*(z^0,R)=\frac{|F^{(K^{(p)})}(z^{(p)})|}{K^{(p)}!\mathbf{L}^{K^{(p)}} (z^{0})}
	\label{stars}
	\end{gather}
	Since by the maximum principle $z^{(p)}\in \mathbb{T}^n(z^0,\frac{pR}{q\mathbf{L}(z^0)}),$ we have $z^{(p)}\neq z^0.$ We choose \\
	$\widetilde z^{(p)}_{j}=z_j^0+\frac{p-1}{p}(z_j^{(p)}-z_j^0).$ 
	Then for every $j\in\{1,\ldots,n\}$ we have 
	\begin{gather}
	|\widetilde z^{(p)}_{j}-z_j^0|=\frac{p-1}{p}|z^{(p)}_{j}-z_j^0|=\frac{p-1}{p} \frac{pr_j}{ql_j(z^0)},
	\label{zet}\\
	|\widetilde z^{(p)}_{j}-z_j^{(p)}|=|z_j^{0}+\frac{p-1}{p}(z_j^{(p)}-z_j^0)-z_j^{(p)}|=
	\frac{1}{p}|z_j^0-z_j^{(p)}|= \frac{1}{p}\frac{pr_j}{ql_j(z^0)}=\frac{r_j}{ql_j(z^0)}. \label{zetwave}
	\end{gather}
	From \eqref{zet} we obtain
	$\widetilde z^{(p)} \in \mathbb{D}^n \left[z^0,\frac{(p-1)R}{q(R)\mathbf{L}(z^0)}\right]$ and 
	\begin{gather*}
	S^{*}_{p-1}(z^0,R)\geq
	\frac{|F^{(K^{(p)})}(\widetilde z^{(p)})|}{K^{(p)}!\mathbf{L}^{K^{(p)}} (z^0)}.
	\end{gather*}
	From \eqref{stars} it follows that
	\begin{gather}
	0\leq S^{*}_p(z^0,R)-S^{*}_{p-1}(z^0,R) \leq \frac{|F^{(K^{(p)})}(z^{(p)})|-|F^{(K^{(p)})}(\widetilde z^{(p)})|}{K^{(p)}!\mathbf{L}^{K^{(p)}} (z^0)}=\nonumber \\
	= \frac{1}{K^{(p)}!\mathbf{L}^{K^{(p)}} (z^0)}\int_{0}^{1}
	\frac{d}{dt}|F^{(K^{(p)})}(\widetilde z^{(p)}+t(z^{(p)}-\widetilde z^{(p)}))|dt \leq
	\nonumber \\
	\leq \frac{1}{K^{(p)}!\mathbf{L}^{K^{(p)}}(z^0)}\int_0^1 \sum_{j=1}^n |z_j^{(p)}-z_{*j}^{(p)}| 
	\left|\frac{\partial^{\|K^{(p)}\| + 1}F}{\partial z_1^{k_1^{(p)}}  \ldots  \partial z_j^{k_j^{(p)}+1} \ldots \partial z_n^{k_n^{(p)}}} (\widetilde z^{(p)}+t(z^{(p)}-\widetilde z^{(p)}))  \right|dt= \nonumber \\ =
	\frac{1}{K^{(p)}!\mathbf{L}^{K^{(p)}}(z^0)} \sum_{j=1}^n |z_j^{(p)}-z_{*j}^{(p)}|
	\left|\frac{\partial^{\|K^{(p)}\|+1}F}{\partial z_1^{k_1^{(p)}} \ldots \partial z_j^{k_j^{(m)}+1} \ldots \partial z_n^{k_n^{(p)}}} (\widetilde z^{(p)}+t^*(z^{(p)}-\widetilde z^{(p)}))  \right|,
	\label{big}
	\end{gather}
	where $0\leq t^*\leq 1, \widetilde z^{(p)}+t^*(z^{(p)}-\widetilde z^{(p)}) \in \mathbb{D}^n (z^0,\frac{pR}{q\mathbf{L}(z^0)})$.
	For $z\in \mathbb{D}^n (z^0,\frac{pR}{q\mathbf{L}(z^0)})$ and $J\in\mathbb{Z}^n_+$, $\|J\| \leq N+1$ we have
	\begin{gather*}
	\frac{|F^{(J)}(z)|\mathbf{L}^{J} (z)}{J!\mathbf{L}^{J} (z^0) \mathbf{L}^{J} (z)}
	\leq (\Lambda_{2}(R))^{J} \max\left\{\frac{|F^{(K)}(z)|}{K! \mathbf{L}^{K} (z)}: \|K\|\leq N \right\} 
	\leq \prod_{j=1}^{n} (\lambda_{2,j}(R))^{N+1}(\lambda_{1,j}(R))^{-N} \times \\ \times \max\left\{\frac{|F^{(K)}(z)|}{K!\mathbf{L}^{K} (z^0)}: \|K\|\leq N \right\} 
	\leq \prod_{j=1}^{n} (\lambda_{2,j}(R))^{N+1}(\lambda_{1,j}(R))^{-N} S^{*}_p(z^0,R).
	\end{gather*}
	From \eqref{big} and \eqref{zetwave} we obtain
	\begin{gather*}
	0\leq S^{*}_p(z^0,R)-S^{*}_{p-1}(z^0,R) \leq \\ \leq \prod_{j=1}^{n} (\lambda_{2,j}(R))^{N+1}(\lambda_{1,j}(R))^{-N} S^{*}_p(z^0,R)\sum_{j=1}^{n}(k_j^{(p)}+1)l_j(z^0)|z_j^{(p)}-\widetilde{z}_j^{(p)}|= \\
	\\ = \prod_{j=1}^{n} (\lambda_{2,j}(R))^{N+1}(\lambda_{1,j}(R))^{-N} \frac{S^{*}_p(z^0,R)}{q(R)}\sum_{j=1}^{n}(k_j^{(p)}+1)r_j \leq \\
	\leq \prod_{j=1}^{n} (\lambda_{2,j}(R))^{N+1}(\lambda_{1,j}(R))^{-N} \frac{S^{*}_p(z^0,R)}{q(R)} (N+1)\|R\| \leq \frac{1}{2}S^{*}_p(z^0,R).
	\end{gather*}
	This inequality implies
	$S^{*}_p(z^0,R) \leq 2S^{*}_{p-1}(z^0,R),$
	and in view of inequalities \eqref{net1ex} and \eqref{stars} we have
	\begin{gather*}
	S_p(z^0,R) \leq 2 \prod_{j=1}^{n} (\lambda_{1,j}(R))^{-N}S^{*}_{p-1}(z^0,R) \leq
	2 \prod_{j=1}^{n} (\lambda_{1,j}(R))^{-N}(\lambda_{2,j}(R))^{N}S_{p-1}(z^0,R)
	\end{gather*}
	Therefore,
	\begin{gather}
	\max\left\{\frac{|F^{(K)}(z)|}{K!\mathbf{L}^{K} (z)}: \|K\|\leq N, z\in \mathbb{D}^n \left[z^0,\frac{pR}{q\mathbf{L}(z^0)}\right] \right\}
	= S_q(z^0,R) \leq \nonumber \\ \leq 2 \prod_{j=1}^{n} (\lambda_{1,j}(R))^{-N}(\lambda_{2,j}(R))^{N}S_{q-1}(z^0,R) \leq \ldots 
	\leq (2 \prod_{j=1}^{n} (\lambda_{1,j}(R))^{-N}(\lambda_{2,j}(R))^{N})^q S_{0}(z^0,R)=\nonumber \\
	= (2 \prod_{j=1}^{n} (\lambda_{1,j}(R))^{-N}(\lambda_{2,j}(R))^{N})^q \max\left\{\frac{|F^{(K)}(z^0)|}{K!\mathbf{L}^{K} (z^0)}: \|K\|\leq N \right\}.
	\label{lastth1}
	\end{gather}
	
	From \eqref{lastth1} we obtain inequality \eqref{net1} with $p_0=(2 \prod_{j=1}^{n} (\lambda_{1,j}(R))^{-N}(\lambda_{2,j}(R))^{N})^q $ and some $K^0$ with $\|K^0\|\leq N$.
	The necessity of condition \eqref{net1} is proved.
	
	Now we prove the sufficiency. Suppose that for every $R\in\mathbb{R}^n_+,$ $|R|\leq \beta,$ there exist $ n_0 \in \mathbb{Z}_+,$ $p_0>1 $ such that for all $ z_0 \in \mathbb{B}^n $ and some
	$ K^0 \in \mathbb{Z}_+^n,$ $\|K^0\|\leq n_0,$ the inequality \eqref{net1} holds.
	
	We write Cauchy's formula as following
	$\forall z^0\in \mathbb{B}^n$ $\forall k\in \mathbb{Z}_+^n$ $\forall s \in \mathbb{Z}_+^n$
	$$
	\frac{F^{(K+S)}(z^0)}{S!}=\frac{1}{(2\pi i)^n} \int_{\mathbb{T}^n\left(z^0,\frac{R}{\mathbf{L}(z^0)}\right)} \frac{F^{(K)}(z)}{(z-z^0)^{S+\mathbf{1}}} dz.
	$$
	Therefore, applying \eqref{net1}, we have
	\begin{gather*}
	\frac{|F^{(K+S)}(z^0)|}{S!}\leq \frac{1}{(2\pi)^n} \int_{\mathbb{T}^n\left(z^0,\frac{R}{\mathbf{L}(z^0)}\right)} \frac{|F^{(K)}(z)|}{|z-z^0|^{S+\mathbf{1}}} |dz|
	\leq  \int_{\mathbb{T}^n\left(z^0,\frac{R}{\mathbf{L}(z^0)}\right)} |F^{(K)}(z)| \frac{\mathbf{L}^{S+\mathbf{1}}(z^0)}{(2\pi)^nR^{S+\mathbf{1}}} |dz| \leq \\
	\leq
	\int_{\mathbb{T}^n\left(z^0,\frac{R}{\mathbf{L}(z^0)}\right)} |F^{(K^0)}(z^0)| \frac{K!p_0 \prod_{j=1}^n \lambda^{n_0}_{2,j}(R) {\mathbf{L}^{S+K+\mathbf{1}}(z^0)}}
	{(2\pi)^nK^0!{R}^{S+\mathbf{1}} {\mathbf{L}^{K^0}(z^0)} } |dz|= \\
	=|F^{(K^0)}(z^0)| \frac{K!p_0\prod_{j=1}^n \lambda^{n_0}_{2,j}(R){\mathbf{L}^{S+K}(z^0)}}
	{K^0!{R}^{S} {\mathbf{L}^{K^0}(z^0)} }.
	\end{gather*}
	This implies
	\begin{gather}
	\! \frac{|F^{(K+S)}(z^0)|}{(K+S)!\mathbf{L}^{S+K}(z^0)} 
	\leq \frac{\prod_{j=1}^n \lambda^{n_0}_{2,j}(R)p_0K! S!}{(K+S)! R^{S}}
	\frac{|F^{(K^0)}(z^0)|}{K^0! \mathbf{L}^{K^0}(z^0)}. \label{dopner}
	\end{gather}
	Obviously, that
	$$\frac{K! S!}{(K+S)!}=\frac{s_1!}{(k_1+1)\cdot\ldots \cdot (k_1+s_1)}\cdots \frac{s_n!}{(k_n+1)\cdot\ldots \cdot (k_n+s_n)}\le1.$$
	We choose $r_j\in(1,\beta/\sqrt{n}],$ $j\in\{1,\ldots,n\}.$ Then $|R|=\sqrt{\sum_{j=1}^n r_j^2}\leq \beta.$
	Hence,
	$ \frac{p_0\prod_{j=1}^n \lambda^{n_0}_{2,j}(R)}{R^{S}} \rightarrow 0 \text{ as } \|S\|\to +\infty.
	$ Thus, there exists $s_0$ such that for all $S\in\mathbb{Z}^n_+$  with $\|S\|\geq s_0$ the inequality holds
	$$
	\frac{p_0K!S!\prod_{j=1}^n \lambda^{n_0}_{2,j}(R) }{(K+S)!R^{S}}\leq 1.
	$$
	Inequality \eqref{dopner} yields 
	$ \frac{|F^{(K+S)}(z^0)|}{(K+S)!\mathbf{L}^{K+S}(z^0)} \leq
	\frac{|F^{(K^0)}(z^0)|}{K^0!{\mathbf{L}}^{K^0}(z^0)}.$
	This means that for every $ j\in \mathbb{Z}_+^n$
	\begin{gather*}
	\frac{|F^{(J)}(z^0)|}{J! \mathbf{L}^{J}(z^0)} \leq
	\max\left\{ \frac{|F^{(K)}(z^0)|}{K! {\mathbf{L}^{K}(z^0)}}: \|K\|\le s_0+n_0\right\}
	\end{gather*}
	where $s_0$ and $n_0$ are independent of $z_0$. Therefore, the analytic in $\mathbb{B}^n$ function $F$ has bounded $\mathbf{L}$-index in joint variables with $N(F,\mathbf{L},\mathbb{B}^n)\le s_0+n_0.$
\end{proof}

\begin{theorem}
	\label{cor1}
	Let $\mathbf{L} \in Q(\mathbb{B}^n).$ In order that an analytic in $\mathbb{B}^n$ function $F$ be of bounded $\mathbf{L}$-index in joint variables it is necessary that for every $R\in\mathbb{R}^n_+,$ $|R|\leq \beta,$ $\exists n_0 \in \mathbb{Z}_+$ $\exists p\geq 1$ $\forall z^0 \in \mathbb{B}^n$ $\exists K^0 \in \mathbb{Z}_+^n $, $\|K^0\| \leq n_0,$ and
	\begin{gather}
	\max\left \{ |F^{(K^0)}(z)|:z \in \mathbb{D}^n\left[z^0, {R}/{\mathbf{L}(z^0)}\right] \right\} \leq p|F^{(K^0)}(z^0)|
	\label{conc1}
	\end{gather}
	and it is sufficient that for every $R\in\mathbb{R}^n_+,$ $|R|\leq \beta,$ $\exists n_0 \in \mathbb{Z}_+$ $\exists p\geq 1$ $\forall z^0 \in \mathbb{B}^n$  $\forall j\in\{1, \ldots, n\}$ $\exists K^0_j=(0,\ldots,0, \underbrace{k^0_j}_{j\text{-th place}},0,\ldots,0)$  such that $k_j^0 \leq n_0$  and
	\begin{gather}
	\!\max \left\{ |F^{(K^0_j)}(z)|:  z \in \mathbb{D}^n\left[z^0, {R}/{\mathbf{L}(z^0)}\right] \right\} \leq p|F^{(K_j^0)}(z^0)|.
	\label{conc3}
	\end{gather}
\end{theorem}
\begin{proof}
	Proof of Theorem \ref{petr1-ball} implies that the inequality \eqref{net1}
	is true for some $K^0.$
	Therefore, we have
	\begin{gather*}
	\frac{p_0}{K^0!} \frac{|F^{(K^0)}(z^0)|}{\mathbf{L}^{K^0} (z^0)}
	\geq
	\max\left\{ \frac{|F^{(K^0)}(z)|}{K^0!\mathbf{L}^{K^0} (z)}: z\in \mathbb{D}^n 
\left[z^0,{R}/{\mathbf{L}(z^0)}\right] \right\} =
	\\
	= \max\left\{ \frac{|F^{(K^0)}(z)|}{K^0!} \frac{\mathbf{L}^{K^0} (z^0)}{ \mathbf{L}^{K^0} (z^0) \mathbf{L}^{K^0} 
(z)}: z\in \mathbb{D}^n \left[z^0,{R}/{\mathbf{L}(z^0)}\right] \right\} \geq
	\\
	\geq
	\max\left\{ \frac{|F^{(K^0)}(z)|}{K^0!} \frac{\prod_{j=1}^n{(\lambda_{2,j}(R))}^{-n_0}}{\mathbf{L}^{K^0} (z^0)}: 
z\in \mathbb{D}^n \left[z^0,{R}/{\mathbf{L}(z^0)}\right] \right\}.
	\end{gather*}
	This inequality implies
	\begin{gather}
	\frac{p_0\prod_{j=1}^n(\lambda_{2,j}(R))^{n_0}}{K^0!} \frac{|F^{(K^0)}(z^0)|}{\mathbf{L}^{K^0} (z^0)} 
\nonumber\\
	\geq
	\max\left\{  \frac{|F^{(K^0)}(z)|}{K^0!\mathbf{L}^{K^0} (z^0)}: z\in \mathbb{D}^n 
\left[z^0,{R}/{\mathbf{L}(z^0)}\right] \right\}.
	\label{conc2}
	\end{gather}
	From \eqref{conc2} we obtain inequality \eqref{conc1} with
	$p=p_0\prod_{j=1}^n{(\lambda_{2,j}(R))}^{n_0}$.
	The necessity of condition \eqref{conc1} is proved.
	
	Now we prove the sufficiency of \eqref{conc3}. Suppose that for every $R\in\mathbb{R}^n_+,$ $|R|\leq\beta,$ 
$\exists n_0 \in \mathbb{Z}_+, p>1 $ such that $\forall z_0 \in \mathbb{B}^n $ and some
	$ K_J^0\in\mathbb{Z}^n_+$ with $k^0_j\leq n_0$ the inequality \eqref{conc3}  holds.
	
	We write Cauchy's formula as following
	$\forall z^0\in \mathbb{B}^n$  $\forall s \in \mathbb{Z}_+^n $
	$$
	\frac{F^{(K_J^0+S)}(z^0)}{S!}=\frac{1}{(2\pi i)^n} \int_{\mathbb{T}^n\left(z^0,{R}/{\mathbf{L}(z^0)}\right)} 
\frac{F^{(K_J^0)}(z)}{(z-z^0)^{S+\mathbf{1}}} dz.
	$$
	This yields
	\begin{gather*}
	\frac{|F^{(K_j^0+S)}(z^0)|}{S!}\leq \frac{1}{(2\pi)^n} \int_{\mathbb{T}^n\left(z^0,{R}/{\mathbf{L}(z^0)}\right)} 
\frac{|F^{(K_j^0)}(z)|}{|z-z^0|^{S+\mathbf{1}}} |dz|
	\leq \\
	\\ \leq \frac{1}{(2\pi)^n} \int_{\mathbb{T}^n\left(z^0,{R}/{\mathbf{L}(z^0)}\right)} \max\{|F^{(K_j^0)}(z)|: z 
\in \mathbb{D}^n\left[z^0, {R}/{\mathbf{L}(z^0)}\right]\} \frac{\mathbf{L}^{S+\mathbf{1}}(z^0)}{R^{S+\mathbf{1}}} |dz| = 
\\
	=
	\max\{|F^{(K_j^0)}(z)|: z \in \mathbb{D}^n\left[z^0, {R}/{\mathbf{L}(z^0)}\right] \} 
\frac{{\mathbf{L}^{S}(z^0)}}{R^{S}}.
	\end{gather*}
	Now we put $R=(\frac{\beta}{\sqrt{n}},\ldots,\frac{\beta}{\sqrt{n}})$ and use \eqref{conc3}
	\begin{gather}
	\frac{|F^{(K_j^0+S)}(z^0)|}{S!} \!\leq\!
	\frac{\mathbf{L}^{S}(z^0)}{{(\beta/\sqrt{n})}^{\|S\|}}
	\max\{|F^{(K_j^0)}(z)|:z\in \mathbb{D}^n\left[z^0, {R}/{\mathbf{L}(z^0)}\right] \}
	\!\leq\!
	\frac{p\mathbf{L}^{S}(z^0)}{(\beta/\sqrt{n})^{\|S\|}}
	|F^{(K_j^0)}(z^0)|. \label{eqa1}
	\end{gather}
	We choose $S\in\mathbb{Z}^n_+$ such that $\|S\|\geq s_0$, where $\frac{p}{(\beta/\sqrt{n})^{s_0}} \leq 1$.
	Therefore, \eqref{eqa1} implies that for all $j\in\{1,\ldots,n\}$ and $k_j^0\leq n_0$ 
	\begin{gather*}
	\frac{|F^{(K_j^0+S)}(z^0)|}
	{{\mathbf{L}^{K_j^0+S}(z^0)}(K_j^0+S)!} \leq
	\frac{p}{{(\beta/\sqrt{n})}^{\|S\|}} \frac{S!K_j^0!}{(S+K_j^0)!}
	\frac{|F^{(K_j^0)}(z^0)|}
	{{\mathbf{L}^{K_j^0}(z^0)}K_j^0!} \leq
	\frac{|F^{(K_j^0)}(z^0)|}{{\mathbf{L}^{K_j^0}(z^0)}K_j^0!}.
	\end{gather*}
	Consequently, $N(F,\mathbf{L},\mathbb{B}^n)\leq n_0+s_0.$
\end{proof}

\begin{remark}\rm
Inequality \eqref{conc1} is necessary and sufficient condition of boundedness of $l$-index for functions of one variable 
\cite{sher,kusher,sherkuz}. But it is unknown whether this condition is sufficient condition of boundedness of 
$\mathbf{L}$-index in joint variables. Our restrictions \eqref{conc3} are corresponding multidimensional sufficient 
conditions.
 \end{remark}

\begin{lemma} \label{indexgreatl}
	Let $\mathbf{L}_1, $ $\mathbf{L}_2\in Q(\mathbb{B}^n)$ and for every $z\in\mathbb{B}^n$ 
	$\mathbf{L}_1(z)\leq \mathbf{L}_2(z).$ If analytic in $\mathbb{B}^n$ function $F$ has bounded 
$\mathbf{L}_1$-index in joint variables then 
	$F$ is of bounded $\mathbf{L}_2$-index in joint variables and $N(F,\mathbf{L}_2,\mathbb{B}^n)\leq n 
N(F,\mathbf{L}_1,\mathbb{B}^n).$
\end{lemma}	
\begin{proof}
Let $N(F,\mathbf{L}_1,\mathbb{B}^n)=n_0.$
Using \eqref{ineqoz2} we deduce 
\begin{gather*}
\frac{|F^{(J)}(z)|}{J!\mathbf{L}^{J}_2(z)} = \frac{\mathbf{L}_1^J(z)}{\mathbf{L}_2^J(z)} 
\frac{|F^{(J)}(z)|}{J!\mathbf{L}^{J}_1(z)} 
\leq  \frac{\mathbf{L}_1^J(z)}{\mathbf{L}_2^J(z)}  \max
\left\{\frac{|F^{(K)}(z)|}{K! \mathbf{L}_1^{K}(z)}:\
K\in\mathbb{Z}^{n}_{+},\ \|K\|\leq n_0\right\} \leq \\
\leq \frac{\mathbf{L}_1^J(z)}{\mathbf{L}_2^J(z)} 
\max
\left\{
 \frac{\mathbf{L}_2^K(z)}{\mathbf{L}_1^K(z)} 
\frac{|F^{(K)}(z)|}{K! \mathbf{L}_2^{K}(z)}:\
K\in\mathbb{Z}^{n}_{+},\ \|K\|\leq n_0\right\}\leq \\
\leq \max_{\|K\|\leq n_0} \left(\frac{\mathbf{L}_1(z)}{\mathbf{L}_2(z)}\right)^{J-K} 
\max
\left\{\frac{|F^{(K)}(z)|}{K! \mathbf{L}_2^{K}(z)}:\
K\in\mathbb{Z}^{n}_{+},\ \|K\|\leq n_0\right\}.
\end{gather*} 
Since $\mathbf{L}_1(z)\leq \mathbf{L}_2(z)$ it means that for all $\|J\|\geq n n_0$
$$
\frac{|F^{(J)}(z)|}{J!\mathbf{L}^{J}_2(z)} \leq 
\max
\left\{\frac{|F^{(K)}(z)|}{K! \mathbf{L}_2^{K}(z)}:\
K\in\mathbb{Z}^{n}_{+},\ \|K\|\leq n_0\right\}.
$$
Thus, $F$ has bounded $\mathbf{L}_2$-index in joint variables and $N(F,\mathbf{L}_2,\mathbb{B}^n)\leq n 
N(F,\mathbf{L}_1,\mathbb{B}^n).$
\end{proof}

Denote $\widetilde{\mathbf{L}}(z)=(\widetilde{l}_1(z),\ldots,\widetilde{l}_n(z))$. 
The notation $\mathbf{L}\asymp \widetilde{\mathbf{L}}$ means that there exist 
$\varTheta_1=(\theta_{1,j},\ldots,\theta_{1,n})\in \mathbb{R}_+^n,$ $\varTheta_2=(\theta_{2,j},\ldots,\theta_{2,n})\in 
\mathbb{R}_+^n$ such that  
$\forall z \in \mathbb{B}^n$
$\theta_{1,j}\widetilde{l}_j(z) \leq l_j(z)\leq \theta_{2,j}\widetilde{l}_j(z)$
 for each $j\in\{1,\ldots,n\}.$

\begin{theorem} \label{petr2}
	Let $\mathbf{L} \in Q(\mathbb{B}^n),$ $\mathbf{L}\asymp \widetilde{\mathbf{L}},$ $\beta|\Theta_1|>\sqrt{n}.$ An 
analytic in $\mathbb{B}^n$  function $F$ has bounded $\widetilde{\mathbf{L}}$-index in joint variables if and only if it 
has bounded $\mathbf{L}$-index.
\end{theorem}
\begin{proof}
	It is easy to prove that if $\mathbf{L} \in Q(\mathbb{B}^n)$ and $\mathbf{L}\asymp \widetilde{\mathbf{L}}$ then 
$\widetilde{\mathbf{L}} \in Q(\mathbb{B}^n).$
	
	Let $N(F,\widetilde{\mathbf{L}},\mathbb{B}^n)=\widetilde{n}_0<+\infty$. Then by Theorem \ref{petr1-ball} for 
every $\widetilde{R}=(\widetilde{r}_1,\ldots,\widetilde{r}_n)\in\mathbb{R}^n_+,$ $|R|\leq \beta,$ there exists 
$\widetilde{p}\ge 1$ such that for each $z^0 \in \mathbb{B}^n$ and some $K^0$ with  $\|K^0\|\leq \widetilde{n}_0,$ the 
inequality \eqref{net1} holds with $\widetilde{\mathbf{L}}$ and $\widetilde{R}$ instead of $\mathbf{L}$ and $R$.
	Hence
	\begin{gather*}
	\frac{\widetilde{p}}{K^0!} \frac{|F^{(K^0)}(z^0)|}{\mathbf{L}^{K^0} (z^0)}=
	\frac{\widetilde{p}}{K^0!} \frac{\Theta_{2}^{K^0} |F^{(K^0)}(z^0)|}
	{\Theta_{2}^{K^0} \mathbf{L}^{K^0} (z^0)} \geq
	\frac{\widetilde{p}}{K^0!} \frac{|F^{(K^0)}(z^0)|}
	{\Theta_{2}^{K^0} \widetilde{\mathbf{L}}^{K^0} (z^0)} \geq
	\\ \geq
	\frac{1}{\Theta_2^{K^0} }\max \left\{ \frac{|F^{(K)}(z)|}
	{K!\widetilde{\mathbf{L}}^K(z)}: \|K\| \leq \widetilde{n}_0, z \in \mathbb{D}^n\left[z^0, 
{\widetilde{R}}/{\widetilde{\mathbf{L}}(z)}\right] \right\} \geq
	\\ \geq
	\frac{1}{\Theta_2^{K^0} }\max \left\{ \frac{\Theta_{1}^{K} |F^{(K)}(z)|}
	{K!\mathbf{L}^{K}(z)}: \|K\| \leq \widetilde{n}_0, z \in \mathbb{D}^n\left[z^0, 
{\Theta_1\widetilde{R}}/{\mathbf{L}(z)}\right] \right\} \geq \\
	\geq
	\frac{\min \limits_{0\! \leq\! \|K\| \!\leq\! n_0
		} \{\Theta_{1}^{K} \}}{\Theta_2^{K^0} }
	\max \left\{ \frac{|F^{(K)}(z)|}
	{K!\mathbf{L}^{K}(z)}: \|K\| \leq\! \widetilde{n}_0, z \!\in\! \mathbb{D}^n\!\left[z^0, 
{\Theta_1\widetilde{R}}/{\widetilde{\mathbf{L}}(z)}\right] \right\}.
	\end{gather*}
	In view of Theorem \ref{petr1-ball} we obtain that function $F$ has bounded $\mathbf{L}$-index in joint 
variables.
\end{proof}

\begin{theorem} \label{ball-cor1}
	Let $\mathbf{L} \in Q(\mathbb{B}^n)$. An analytic  in $\mathbb{B}^n$  function $F$ has bounded 
$\mathbf{L}$-index in joint variables if and only if there exist $R\in\mathbb{R}^n_+,$ with 
	$|R|\le \beta,$ $n_0 \in \mathbb{Z}_+,$ $p_0>1$ such that for each $z^0 \in 
\mathbb{B}^n$ and for some $K^0 \in \mathbb{Z}_+^n$ with  $\|K^0\| \leq n_0$ the inequality \eqref{net1} holds.
\end{theorem}
\begin{proof}
	The necessity of this theorem follows from the necessity of Theorem \ref{petr1-ball}.
	We prove the sufficiency.
	The proof of Theorem \ref{petr1-ball} with $R=(\frac{\beta}{\sqrt{n}},\ldots,\frac{\beta}{\sqrt{n}})$ implies 
that $N(F,\mathbf{L},\mathbb{B}^n)<+\infty.$
	
	Let $\mathbf{L}^{*}(z)=\frac{R_0\mathbf{L}(z)}{R},$ 
$R^0=(\frac{\beta}{\sqrt{n}},\ldots,\frac{\beta}{\sqrt{n}}).$
	In general case from validity of \eqref{net1} for $F,$  $\mathbf{L}$ and 
	$R=(r_1,\ldots,r_n)$ with 
	 $r_j<\frac{\beta}{\sqrt{n}},$ $j\in\{1,\ldots,n\}$ we obtain
	\begin{gather*}
		\max \left\{ \frac{|F^{(K)}(z)|}
	{K!(\mathbf{L}^*(z^0))^{K}}: \|K\| \leq n_0, z \in \mathbb{D}^n\left[z^0, {R_0}/{\mathbf{L}^*(z^0)}\right] 
\right\} = \\
=	\max \left\{ \frac{|F^{(K)}(z)|}
	{K!(R_0 \mathbf{L} (z)/R)^{K}}: \|K\| \leq n_0, z \in \mathbb{D}^n\left[z^0, {R_0}/(R_0 \mathbf{L} (z)/R)\right] 
\right\} \leq \\
	\leq
	\max \left\{ \frac{{n}^{\|K\|/2}|F^{(K)}(z)|}
	{K!\mathbf{L}^{K} (z)}: \|K\| \leq n_0, z\! \in\! \mathbb{D}^n\left[z^0, {R}/{\mathbf{L}(z^0)}\right] \right\} 
\leq
	\\
	\leq 
	\frac{p_0}{K^0!} \frac{{n}^{n_0/2}|F^{(K^0)}(z^0)|}{\mathbf{L}^{K^0} (z^0)}
	=
	\frac{n^{n_0/2}(\beta/\sqrt{n})^{\|K^0\|}p_0}{R^{K^0}K^0!} \frac{|F^{(K^0)}(z)|}{(R_0 \mathbf{L} 
(z)/R)^{K^0}}
	< \frac{p_0\beta^{n_0}}{\prod_{j=1}^nr_j^{n_0}} \frac{|F^{(K^0)}(z)|}{K^0!(\mathbf{L}^*(z))^{K^0}}.
	\end{gather*}
	i. e. \eqref{net1} holds for $F,$ $\mathbf{L}^*$ and 
$R_0=(\frac{\beta}{\sqrt{n}},\ldots,\frac{\beta}{\sqrt{n}}).$
	As above 
	we apply Theorem \ref{petr1-ball} to the function $F(z)$ and $\mathbf{L}^{*}(z)=R_0 \mathbf{L} (z)/R$. This 
implies that $F$ is of bounded $\mathbf{L}^{*}$-index in joint variables. Therefore, by Theorem \ref{petr2} the function 
$F$ has bounded $\mathbf{L}$-index in joint variables.
\end{proof}

\section{Local behaviour of maximum modulus of analytic in ball  function}\
For an analytic in $\mathbb{B}^n$ function $F$ we put
$$
M(r,z^0,F)=\max\{|F(z)|\colon  z\in \mathbb{T}^n(z^0,r)\},
$$
where $z^0\in\mathbb{B}^n,$ $r\in\mathbb{R}^n_+.$  
Then $M(R,z^0,F)=\max\{|F(z)|  \colon z\in \mathbb{D}^n[z^0,R]\},$ because the maximum modulus for an analytic function in a closed polydisc is  attained on its skeleton.

\begin{theorem}\sl  \label{bordte14-ball} 
	Let $\mathbf{L}\in Q(\mathbb{B}^n).$
	If an analytic in $\mathbb{B}^n$ function $F$ has bounded $\mathbf{L}$-index in joint variables 
	then for any $R',$ 
$R''\in\mathbb{R}^n_+$ $0<|R'|<|R''|<\beta,$ there exists a number $p_1=p_1(R',R'')\geq 1$ such that for every 
$z^0\in\mathbb{B}^n$
	\begin{equation}
	\label{bordriv112nec}
	M\left(\frac{R''}{\mathbf{L}(z^0)},z^0,F\right) \leq p_1  M\left(\frac{R'}{\mathbf{L}(z^0)},z^0,F\right).
	\end{equation}
\end{theorem}
\begin{proof}
	Let $N(F,\mathbf{L})=N<+\infty.$ Suppose that  inequality \eqref{bordriv112nec} does not hold i.e. there exist
	$R',$ $R'',$ $0<|R'|<|R''|<\beta,$ such that for each $p_*\geq 1$ and for some $z^0=z^0(p_*)$
	\begin{equation}
	\label{bordriv113nec}
	M\left(\frac{R''}{\mathbf{L}(z^0)},z^0,F\right) > p_*  M\left(\frac{R'}{\mathbf{L}(z^0)},z^0,F\right).
	\end{equation}
	By Theorem~\ref{ball-cor1}, there exists a number $p_0=p_0(r'')\geq 1$ such that for every $z^0\in\mathbb{B}^n$ 
and some $K^0\in\mathbb{Z}^n_{+},$ $\|K^0\| \leq N,$  one has
	\begin{equation}
	\label{bordriv114nec}
	M\left(\frac{R''}{\mathbf{L}(z^0)},z^0,F^{(K^0)}\right) \leq p_0 |F^{(K^0)}(z^0)|.
	\end{equation}
	We put 
	\begin{gather*}
	b_1=p_0\left(\prod_{j=2}^n \lambda_{2,j}^N(R'')\right)(N!)^{n-1} \left(\sum_{j=1}^{N} \frac{(N-j)!}{(r_1'')^j} \right)
	\left(\frac{r_1''r_2'' \ldots r_n''}{r_1'r_2'\ldots r_n'} \right)^N,\\
	b_2 = p_0 \left( \prod_{j=3}^n \lambda_{2,j}^N(R'') \right)(N!)^{n-2}  \left( \sum_{j=1}^{N} \frac{(N-j)!}{(r_2'')^j} \right)
	\left(\frac{r_2'' \ldots r_n''}{r_2'\ldots r_n'}  \right)^N \max\left\{1,\frac{1}{(r_1')^N}  \right\}, \\
	\ldots \\
	b_{n-1}=p_0\lambda_{2,n}^N(R')N! \left(\sum_{j=1}^N \frac{(N-j)!}{(r''_{n-1})^j}\right) \left(\frac{r''_{n-1}r''_n}{r'_{n-1}r'_n} \right)^N
	\max\left\{1, \frac{1}{(r_1'\ldots r'_{n-2})^N} \right\}, \\
	b_n=p_0 \left(\sum_{j=1}^N \frac{(N-j)!}{(r''_n)^j} \right) \left(\frac{r''_n}{r'_n}\right)^N \max\left\{1, \frac{1}{(r_1'\ldots r'_{n-1})^N} \right\}
	\end{gather*}
	and
	$$
	p_* = (N!)^n p_0 \left( \frac{r_1''r_2''\ldots r_n''}{r_1'r_2'\ldots r_n'}\right)^N+\sum_{k=1}^n b_k+1.
	$$
	
	Let $z^0=z^0(p_*)$ be a point for which  inequality \eqref{bordriv113nec} holds and $K^0$ be such  that \eqref{bordriv114nec} holds and
	\begin{gather*}
	M\left(\frac{R'}{\mathbf{L}(z^0)},z^0,F\right)=|F(z^*)|, \
	M\left(\frac{r''}{\mathbf{L}(z^0)},z^0,F^{(J)}\right)=|F^{(J)}(z^*_J)|
	\end{gather*}
	for every $J\in\mathbb{Z}^n_{+},$ $\|J\|\leq N.$
	We apply Cauchy's inequality
	\begin{equation}
	\label{bordriv115}
	|F^{(J)}(z^0)|\leq J! \left(\frac{\mathbf{L}(z^0)}{R'}\right)^J |F(z^*)|
	\end{equation}
	for estimate the difference
	\begin{gather}
	|F^{(J)}(z_{J,1}^*,z_{J,2}^*,\ldots,z_{J,n}^*)-F^{(J)}(z_1^0,z_{J,2}^*,\ldots,z_{J,n}^*)|= \nonumber\\ =
	\left|
	\int_{z_1^0}^{z^*_{J,1}} \frac{\partial^{\|J\|+1}F}{\partial z_1^{j_1+1}\partial z_2^{j_2}\ldots \partial z_n^{j_n}}(\xi,z_{J,2}^*,\ldots,z_{J,n}^*)d\xi
	\right|\leq \nonumber\\ \leq
	\left|\frac{\partial^{\|J\|+1}F}{\partial z_1^{j_1+1}\partial z_2^{j_2}\ldots \partial 
z_n^{j_n}}(z^*_{(j_1+1,j_2,\ldots,j_n)}) \right| \frac{r''_1}{l_1(z^0)}. \label{bordriv116}
	\end{gather}
	Taking into account $(z_1^0,z^*_{J,2},\ldots,z^*_{J,n})\in \mathbb{D}^n[z^0,\frac{R''}{\mathbf{L}(z^0)}],$ for all 
$k\in\{1, \ldots, n\}$ 
	$|z^*_{J,k}-z^0_k| = \frac{r''_k}{l_k(z^0)},$
	$l_k(z_1^0,z^*_{J,2},\ldots,z^*_{J,n})\leq$ $\lambda_{2,k}(R'')l_k(z^0)$  and \eqref{bordriv115}
	with $J=K^0,$ by Theorem \ref{petr1-ball}  we have 
	\begin{gather}
	|F^{(J)}(z_1^0,z^*_{J,2},\ldots,z^*_{J,n})| \leq
	\frac{J! l_1^{j_1}(z_1^0,z^*_{J,2},\ldots,z^*_{J,n})\prod_{k=2}^n l_k^{j_k}(z_1^0,z^*_{J,2},\ldots,z^*_{J,n})}{K^0!\mathbf{L}^{K^0}(z^0)}p_0 |F^{(K^0)}(z^0)| \leq \nonumber \\
	\leq \frac{J!\mathbf{L}^J(z^0)\prod_{k=2}^n \lambda^{j_k}_{2,k}(R'')}{K^0!\mathbf{L}^{K^0}(z^0)}p_0K^0! \left(\frac{\mathbf{L}(z^0)}{R'}\right)^{K^0}|F(z^*)|
	=\frac{p_0J!\mathbf{L}^J(z^0)\prod_{k=2}^n \lambda^{j_k}_{2,k}(R'')}{(R')^{K^0}}|F(z^*)|. \label{bordriv117}
	\end{gather}
	From inequalities \eqref{bordriv116} and \eqref{bordriv117} it follows that
	\begin{gather*}
	\left|\frac{\partial^{\|J\|+1}F}{\partial z_1^{j_1+1}\partial z_2^{j_2}\ldots \partial z_n^{j_n}}(z^*_{(j_1+1,j_2,\ldots,j_n)})\right|\geq
	\frac{l_1(z^0)}{r_1''} \left\{|F^{(J)}(z^*_j)|-|F^{(J)}(z_1^0,z^*_{J,2},\ldots,z^*_{J,n})| \right\} \geq \\
	\geq \frac{l_1(z_1^0)}{r_1''}|F^{(J)}(z^*_j)|- \frac{p_0J!\mathbf{L}^{(j_1+1,j_2,\ldots,j_n)}(z^0)\prod_{k=2}^n \lambda^{j_k}_{2,k} (R'')}{r_1''(R')^{K^0}}|F(z^*)|.
	\end{gather*}
	Then
	\begin{gather}
	|F^{(K^0)}(z^*_{K^0})| \geq  \frac{l_1(z^0)}{r_1''} \left|\frac{\partial^{\|K^0\|-1}f}{\partial z_1^{k_1^0-1}\partial z_2^{k_2^0} \ldots \partial z_n^{k_n^0}} (z^*_{(k_1^0-1,k_2^0,\ldots,k_n^0)})\right|-\nonumber\\
	-\frac{p_0(k_1^0-1)!k_2^0!\ldots k_n^0!\mathbf{L}^{K^0}(z^0)\prod_{i=2}^n \lambda_{2,i}^{k_i^0}(R'')}{r_1''(R')^{K^0}}|F(z^*)| \geq \nonumber\\
	\geq \frac{l_1^2(z^0)}{(r_1'')^2} \left| \frac{\partial^{\|K^0\|-2}f}{\partial z_1^{k_1^0-2}\partial z_2^{k_2^0} \ldots \partial z_n^{k_n^0}} (z^*_{(k_1^0-2,k_2^0,\ldots,k_n^0)})\right|-\nonumber\\
	-\frac{p_0(k_1^0-2)!k_2^0!\ldots k_n^0!\mathbf{L}^{K^0}(z^0)\prod_{i=2}^n \lambda_{2,i}^{k_i^0}(R'')}{(r_1'')^2(R')^{K^0}}|F(z^*)|-\nonumber\\
	- \frac{p_0(k_1^0-1)!k_2^0!\ldots k_n^0!\mathbf{L}^{K^0}(z^0)\prod_{i=2}^n \lambda_{2,i}^{k_i^0}(r_i'')}{r_1''(R')^{K^0}}|F(z^*)|\geq \nonumber\\
	\ldots \nonumber \\
	\geq \frac{l_1^{k_1^0}(z^0)}{(r_1'')^{k_1^0}} \left| \frac{\partial^{\|K^0\|-k_1^0}f}{\partial z_2^{k_2^0} \ldots \partial z_n^{k_n^0}} (z^*_{(0,k_2^0,\ldots,k_n^0)})\right|- \nonumber \\
	- \frac{p_0}{(R')^{K^0}} \mathbf{L}^{K^0}(z^0) \left(\prod_{i=2}^n \lambda_{2,i}^{k_i^0}(R'')\right) k_2^0!\ldots k_n^0! \sum_{j_1=1}^{k_1^0}
	\frac{(k_1^0-j_1)!}{(r_1'')^{j_1}}|F(z^*)| \geq \nonumber\\  \ldots \nonumber\\
	\geq \frac{l_1^{k_1^0}(z^0)}{(r_1'')^{k_1^0}} \frac{l_2^{k_2^0}(z^0)}{(r_2'')^{k_2^0}}
	\left| \frac{\partial^{\|K^0\|-k_1^0-k_2^0}f }{\partial z_3^{k_3^0}\ldots \partial z_n^{k_n^0}}(z^*_{(0,0,k_3^0,\ldots,k_n^0)}) \right|- \nonumber\\
	- \frac{l_1^{k_1^0}(z^0)p_0L^{(0,k_2^0,\ldots,k_n^0)}(z^0)}{(r_1'')^{k_1^0}(R')^{K^0}} \left(\prod_{i=3}^n \lambda_{2,i}^{k_i^0}(R'')\right)
	k_3^0!\ldots k_n^0! \sum_{i_2=1}^{k_2^0} \frac{(k_2^0-j_2)!}{(r_2'')^{j_2}}|F(z^*)|-\nonumber\\
	-\frac{p_0}{(R')^{K^0}}\mathbf{L}^{K^0}(z^0) \left(\prod_{i=2}^n \lambda_{2,i}^{k_i^0}(R'') \right) k_2^0! \ldots k_n^0! \sum_{j_1=1}^{k_1^0}
	\frac{(k_1^0-j_1)!}{(r_1'')^{j_1}}|F(z^*)| \geq  \nonumber\\
	\ldots \nonumber \\
	\geq \left(\frac{L(z^0)}{R''}\right)|F(z^*_{\mathbf{0}})|-|F(z^*)|\sum_{i=1}^b \tilde{b}_i, \label{bordriv118}
	\end{gather}
	where  in view of the inequalities $\lambda_{2,i}(R'')\geq 1$ and $R''\geq R'$ we {have}
	\begin{gather*}
	\tilde{b}_1=\frac{p_0}{(R')^{K^0}}\mathbf{L}^{K^0}(z^0) \left(\prod_{i=2}^n \lambda_{2,i}^{k_i^0}(R'')\right) k_2^0!\ldots k_n^0!
	\sum_{j_1=1}^{k_1^0} \frac{(k_1^0-j_1)!}{(r_1'')^{j_1}}= \\
	= \left(\frac{\mathbf{L}(z^0)}{R''} \right)^{K^0} \left( \frac{R''}{R'} \right)^{K^0}p_0 \left(\prod_{i=2}^n \lambda_{2,i}^{k_i^0}(R'') \right)
	k_2^0!\ldots k_n^0! \sum_{j_1=1}^{k_1^0} \frac{(k_1^0-j_1)!}{(r_1'')^{j_1}}
	\leq \left(\frac{\mathbf{L}(z^0)}{R''} \right)^{K^0} b_1,
	\end{gather*}
	\begin{gather*}
	\tilde{b}_2 =  \frac{p_0}{(R')^{K^0}}\mathbf{L}^{K^0}(z^0)  \left(  \prod_{i=3}^n  \lambda_{2,i}^{k_i^0}(R'') \right)  \frac{k_3^0!\ldots k_n^0!}{(r_1'')^{k_1^0}}   \sum_{j_2=1}^{k_2^0} \frac{(k_2^0 - j_2)!}{(r_2'')^{j_2}}
	\leq  \left( \frac{\mathbf{L}(z^0)}{R''}\right)^{K^0}b_2,  \\ \ldots
	\end{gather*}
	\begin{gather*}
	\tilde{b}_{n-1} = \frac{p_0}{(R')^{K^0}}\mathbf{L}^{K^0}(z^0) \lambda_{2,n}^{k_n^0}(R'') \frac{k_n^0!}{(r_1'')^{k_1^0} \ldots (r''_{n-2})^{k^0_{n-2}}} \times \\
	\times \sum_{j_{n-1}=1}^{k_{n-1}^0} \frac{(k^0_{n-1}-j_{n-1})!}{(r''_{n-1})^{j_{n-1}}} \leq \left(\frac{\mathbf{L}(z^0)}{R''}\right)^{K^0}b_{n-1},
	\\
	\tilde{b}_n  =  \frac{p_0}{(R')^{K^0}}\mathbf{L}^{K^0}(z^0)  \frac{1}{(r_1'')^{k_1^0}\ldots (r''_{n-1})^{k_{n-1}^0}}
	\sum_{j_n=1}^{k_n^0} \frac{(k_n^0-j_n)!}{(r_n'')^{j_n}}
	\leq  \left(\frac{\mathbf{L}(z^0)}{R''} \right)^{K^0}b_n.
	\end{gather*}
	Thus, \eqref{bordriv118} implies that
	$$
	|F^{(K^0)}(z^*_{K^0})|\geq \left(\frac{\mathbf{L}(z^0)}{R''} \right)^{K^0} |F(z^*)| \left\{ \frac{|F(z^*_{\mathbf{0}})|}{|F(z^*)|}-\sum_{j=1}^n b_j \right\}.
	$$
	But in view of \eqref{bordriv113nec} and a choice of $p_*$ we have
	$
	\frac{|F(z^*_{\mathbf{0}})|}{|F(z^*)|}\geq p_*> \sum_{j=1}^n b_j.
	$
	Thus, in view of \eqref{bordriv114nec} and \eqref{bordriv115} we obtain
	\begin{gather*}
	|F^{(K^0)}(z^*_{K^0})|\geq \left(\frac{\mathbf{L}(z^0)}{R''} \right)^{K^0}\!\! |F(z^*)| \left\{p_* -\sum_{j=1}^n b_j \right\}\geq \\
	\geq \left(\frac{\mathbf{L}(z^0)}{R''} \right)^{K^0}   \left\{ p_*-\sum_{j=1}^nb_j\right\}
	\frac{|F^{(K^0)}(z^0)|(R')^{K^0}}{K^0!\mathbf{L}^{K^0}(z^0)}
	\geq \left( \frac{r_1'\ldots r_n'}{r_1''\ldots r_n''}\right)^N\!\!\left\{p_*-\sum_{j=1}^nb_j \right\} \frac{|F^{(K^0)}(z^*_{K^0})|}{p_0(n!)^n}.
	\end{gather*}
	Hence, we have
	$
	p_*\leq p_0 \big( \frac{r_1'\ldots r_n'}{r_1''\ldots r_n''}\big)^N (N!)^n+\sum_{j=1}^n b_j,
	$
	but this contradicts the choice of $p_*.$
	\end{proof}

\begin{theorem} \label{bordte43}\sl
	Let $\mathbf{L}\in Q^n,$ $F: \mathbb{B}^n \to \mathbb{C}$ be analytic function. If 
 there exist $R',$ $R''\in\mathbb{R}^n_+,$ $\mathbf{0}<R'<\mathbf{1}<R'',$  and $p_1=p_1(R',R'')\geq 1$ such that for 
every 
$z^0\in\mathbb{C}^n$ inequality \eqref{bordriv112nec} holds 
then  $F$ is of bounded $\mathbf{L}$-index 
in joint variables.
\end{theorem}
\begin{proof}
Let $z^0\in\mathbb{B}^n$ be an arbitrary point. We expand a function $F$ in power 
series
	\begin{equation}
	\label{bordriv119}
	F(z)=\sum_{K\geq\mathbf{0}}b_K(z-z^0)^K= \sum_{k_1,\ldots,k_n\geq 0}b_{k_1,\ldots,k_n}(z_1-z_1^0)^{k_1} \ldots (z_n-z_n^0)^{k_n},
	\end{equation}
	where
	$b_{K}=b_{k_1,\ldots,k_n}= \frac{F^{(K)}(z^0)}{K!}.$
	
	Let $\mu(R,z^0,F)=\max\{|b_K|R^K\colon  \ K\geq \mathbf{0}\}$ be the maximal term of series \eqref{bordriv119} 
and $\nu(R)=\nu(R,z^0,F)=(\nu_1^0(R),\ldots,\nu_n^0(R))$ be a set of indices such that
	$$
	\mu(R,z^0,F)=|b_{\nu(R)}|R^{\nu(R)},\
	\|\nu(R)\|=\!\sum_{j=1}^n \nu_j(R)=\max\! \{\|K\|\colon  K\geq \mathbf{0},\  |b_K|R^K=\mu(R,z^0,F) \}.
	$$
	In view of inequality \eqref{bordriv115} we {obtain}
	for any $|R|< 1-|z^0|$ \ $\mu(R,z^0,F)\leq M(R,z^0,F).$
	Then for given $R'$ and $R''$ with $0<|R'|<1<|R''|<\beta$ we conclude 
	 \begin{gather*}
   M(R'R,z^0,F)\leq \sum_{k\geq\mathbf{0}}|b_k|(R'R)^k\leq
   \sum_{k\geq\mathbf{0}}\mu(R,z^0,F)(R')^k=\mu(R,z^0,F)\sum_{k\geq\mathbf{0}}(R')^k= \nonumber \\
   =\prod_{j=1}^{n}\frac{1}{1-r'_j}\mu(R,z^0,F).
   \end{gather*}
   Besides,
   \begin{gather*}
   \ln \mu(R,z^0,F)=\ln \{|b_{\nu(R)}|R^{\nu(R)} \}=
   \ln\left\{|b_{\nu(R)}|(RR'')^{\nu(R)}\frac{1}{(R'')^{\nu(R)}}\right\}= \nonumber \\
   =\ln\{|b_{\nu(R)}|(RR'')^{\nu(R)}\}+\ln\left\{\frac{1}{(R'')^{\nu(R)}}\right\}\leq
   \ln \mu(R''R,z^0,F)-\|\nu(R) \| \ln \min_{1\leq j\leq n} r_j''.
   \end{gather*}
   This implies that
   \begin{gather}
   \|\nu(R)\|\leq \frac{1}{\ln \min_{1\leq j \leq n} r_j''}( \ln \mu(R''R,z^0,F)- \ln \mu(R,z^0,F))\leq
   \nonumber \\
   \leq
   \frac{1}{\ln  \min_{1\leq j \leq n} r_j''}\left( \ln M(R''R,z^0,F)- \ln( \prod_{j=1}^n (1-r_j')  
M(R'R,z^0,F))\right)\leq
   \nonumber \\
   \leq
   \frac{1}{\ln  \min_{1\leq j \leq n} r_j''}\left( \ln M(R''R,z^0,F)- \ln M(R'R,z^0,F)\right)-
   \frac{\sum_{j=1}^{n}\ln(1-r'_j)}{ \min_{1\leq j \leq n} r_j''}= \nonumber \\
   =\frac{1}{ \min_{1\leq j \leq n} r_j''}
   \ln\frac{M(R''R,z^0,F)}{M(R'R,z^0,F)}- \frac{\sum_{j=1}^{n}\ln(1-r_j)}{ \min_{1\leq j \leq n} r_j''}.
   \label{th4suf4}
   \end{gather}
	
Put $R=\frac{\mathbf{1}}{\mathbf{L}(z^0)}.$ Now let $N(F,z^0,\mathbf{L})$ be a $\mathbf{L}$-index of the function $F$ in 
joint variables at point $z^0$ i. e. it is the least integer
   for which inequality \eqref{ineqoz2} holds at point $z^0.$ Clearly that
   \begin{equation} \label{eq1ver}
   N(F,z^0,\mathbf{L})\leq\nu\left(\frac{1}{\mathbf{L}(z^0)},z^0,F\right)=\nu(R,z^0,F).\end{equation}
   But
   \begin{equation} \label{eq2ver}
   M\left({R''}/{\mathbf{L}(z^0)},z^0,F\right)\leq p_1(R',R'')M\left({R'}/{\mathbf{L}(z^0)},z^0,F\right).
   \end{equation}
   Therefore, from \eqref{th4suf4}, \eqref{eq1ver}, \eqref{eq2ver} we obtain that $\forall z^0 \in\mathbb{B}^n$
   $$
   N(F,z^0,\mathbf{L})\leq \frac{-\sum_{j=1}^{2}\ln(1-r'_j)}{\ln \min \{r_1'',r_2''\}}+\frac{\ln p_1(R',R'')}{\ln \min 
\{r_1'',r_2''\}}.
   $$
   This means that $F$ has bounded $\mathbf{L}$-index in joint variables.

\end{proof}

\section{Boundedness of $\mathbf{L}$-index in every direction $\mathbf{1}_j$}
The boundedness of $l_j$-index of a function $F(z)$ in every variable $z_j,$ generally speaking,
does not imply the boundedness of $\mathbf{L}$-index in joint variables (see example in \cite{bagzmin}). But, if  $F$ has bounded 
$l_j$-index in every  direction $\mathbf{1}_j,$  $j\in\{1,\ldots, n\},$ then $F$ is a function of bounded 
$\mathbf{L}$-index in joint variables.

 For $\eta \in[0,\beta],$\ $z\in\mathbb{B}^n,$\
we define
$\lambda^{\mathbf{b}}_{1}(z,\eta,L)=\inf\left\{\frac{L(z+t\mathbf{b})} 
{L(z)}:
|t|\leq\frac{\eta}{L(z)}\right\},$
$\lambda^{\mathbf{b}}_{1}(\eta,L)=\inf\{\lambda^{\mathbf{b}}_{1}(z,\eta,L):
z\in\mathbb{B}^n\},$ 
$\lambda^{\mathbf{b}}_{2}(z,\eta,L)=\sup\left\{\frac{L(z+t\mathbf{b})}
{L(z)}:
|t|\leq\frac{\eta}{L(z)}\right\},$
$\lambda^{\mathbf{b}}_{2}(\eta,L)\!=\sup\{\lambda^{\mathbf{b}}_{2}(z,\eta,L):
z\in\mathbb{B}^n\}.$

By ${Q}_{\mathbf{b},\beta}(\mathbb{B}^n)$ we denote the class of all functions $L$ for which the following condition holds for any
$\eta\in[0,\beta]$ \ \ 
$0<\lambda^{\mathbf{b}}_{1}(\eta,L)\leq\lambda^{\mathbf{b}}_{2}(\eta,L)<+\infty.$

We need the following theorem. 
\begin{theorem}[\cite{banduraball}] \label{te3-ball}
Let $\beta>1,$ $L\in {Q}_{\mathbf{b},\beta}(\mathbb{B}^n)$. Analytic in $\mathbb{B}^n$ function
$F(z)$ is of bounded $L$-index in the  direction
$\mathbf{b}\in\mathbb{C}^{n}$ if and only if for any
 $r_{1}$ and any $r_{2}$ with $0<r_{1}<r_{2}\leq \beta,$
there exists number $P_{1}=P_{1}(r_{1},r_{2})\geq 1$ such that for each
$z^{0}\in \mathbb{B}^n$ 
\begin{gather}  \max\Big\{|F(z^{0}+t\mathbf{b})|:
|t|=\frac{r_{2}}{L(z^{0})}\Big\} \label{riv31} \leq P_{1}
\max\Big\{|F(z^{0}+t\mathbf{b})|:\
|t|=\frac{r_{1}}{L(z_{0})}\Big\}.
\end{gather}
\end{theorem}

\begin{theorem} \label{bordcor42}\sl
	Let $\mathbf{L}(z)=(l_1(z),\ldots,l_n(z))$, where $l_j\in Q_{\mathbf{1}_{j},\beta/\sqrt{n}}(\mathbb{B}^n)$ 
$(j\in\{1,\ldots,n\})$.	If an analytic in $\mathbb{B}^n$ function $F$  has bounded $l_j$-index in the direction 
$\mathbf{1}_j$ for every $j\in\{1,\ldots,n\},$ then
	$F$ is of bounded $\mathbf{L}$-index in joint variables.
\end{theorem}
\begin{proof}
	Let  $F$ be an analytic in $\mathbb{B}^n$ function of  bounded $l_j$-index in every direction $\mathbf{1}_j.$ Then by Theorem \ref{te3-ball}
	for every $j\in\{1, \ldots, n\}$ and arbitrary $0<r_j'<1<r_j''\leq \frac{\beta}{\sqrt{n}}$ there exists a number $p_j=p_j(r',r'')$ such that  for every $(z_1,\ldots,z_{j-1},z_j^0,z_{j+1},\ldots,z_n)\in \mathbb{B}^{n}$  inequality
	\begin{gather}
	\max \left\{ |F(z)|\colon  |z_j-z_j^0| = \frac{r_j''}{l_j(z_1,\ldots,z_{j-1},z_j^0,z_{j+1},\ldots,z_n)}\right\} \leq  p_j (r_j',r_j'' )
	\times \nonumber\\ \times
	\max \left\{ |F(z)|\colon  |z_j-z_j^0|=\frac{r_j'}{l_j(z_1,\ldots,z_{j-1},z_j^0,z_{j+1},\ldots,z_n)}  \right\} \label{preineq}
	\end{gather}
	holds. 

	Obviously, if for every $j\in\{1,\ldots,n\}$ \ $l_j\in Q_{\mathbf{1}_{j},\beta/\sqrt{n}}(\mathbb{B}^n)$  then $\mathbf{L}\in Q(\mathbb{B}^n).$
	Let $z^0$ be an arbitrary point in $\mathbb{B}^n,$ and a point $z^*\in \mathbb{T}^n(z^0,\frac{R''}{\mathbf{L}(z^0)})$ is such that
	$M(\frac{R''}{\mathbf{L}(z^0)},z^0,F)=|F(z^*)|.$
	We choose $R''$ and $R'$ such that $\mathbf{1}< R''<(\frac{\beta}{\sqrt{n}},\ldots,\frac{\beta}{\sqrt{n}})$  and $R'< \Lambda_1(R'').$
	Then inequality \eqref{preineq} implies that
	\begin{gather*}
	M\left(\frac{R''}{\mathbf{L}(z^0)},z^0,F\right)\leq   \max\left\{|F(z_1,z_2^*,z_3^*,\ldots,z_n^*)|\colon  \ |z_1-z_1^0|=\frac{r_1''}{l_1(z^0)}\right\} = \\
	= \max\left\{|F(z_1,z_2^*,\ldots,z_n^*)|\colon   |z_1-z_1^0|=\frac{r_1''}{l_1(z_1^0,z_2^*,\ldots,z_n^*)}\frac{l_1(z_1^0,z_2^*,\ldots,z_n^*)}{l_1(z^0)} \right\}\leq
	\\
	\leq \max\left\{|F(z_1,z_2^*,\ldots,z_n^*)|\colon  \ |z_1-z_1^0|=\frac{r_1''\lambda_{2,1}(R'')}{l_1(z_1^0,z_2^*,\ldots,z_n^*)}\right\}\leq
	\\
	\leq  p_1(r'_1,r_1''\lambda_{2,1}(R'')) \max\left\{|F(z_1,z_2^*,\ldots,z_n^*)|\colon  \ |z_1-z_1^0|=\frac{r_1'}{l_1(z_1^0,z_2^*,\ldots,z_n^*)}\right\} = \\
	=  p_1(r'_1,r_1''\lambda_{2,1}(R''))  \max \left\{ |F(z_1,z_2^*,\ldots,z_n^* )|\colon   |z_1 - z_1^0| = \frac{r_1'}{l_1(z^0)} \frac{l_1(z^0)}{ l_1 ( z_1^0,z_2^*,\ldots,z_n^*)} \right\} \leq  \\
	\leq p_1(r'_1,r_1''\lambda_{2,1}(R'')) \max\left\{|F(z_1,z_2^*,\ldots,z_n^*)|\colon  \ |z_1-z_1^0|=\frac{r_1'}{\lambda_{1,1}(R'')l_1(z^0)} \right\}=
	\\
	=p_1(r'_1,r_1''\lambda_{2,1}(R''))  |F(z_1^{**},z_2^*,\ldots,z_n^*)|
	\leq p_1(r'_1,r_1''\lambda_{2,1}(R''))  \times \\ \times
	\max\left\{|F(z_1^{**},z_2,z_3^*,\ldots,z_n^*)|\colon  \ |z_2-z_2^0|=\frac{r_2''}{l_2(z^0)} \right\}=    p_1(r'_1,r_1''\lambda_{2,1}(R''))   \times
	\\  \times     \max \left\{|F(z_1^{**},z_2,\ldots,z_n^*)|\colon  \ |z_2-z_2^0| = \frac{r_2''}{l_2(z_1^{**},z_2^0,\ldots,z_n^*)} \frac{l_2(z_1^{**},z_2^0,\ldots,z_n^*)}{l_2(z^0)}  \right\}  \leq
	\\  \leq           p_1(r'_1,r_1''\lambda_{2,1}(R''))
	\max\left\{|F(z_1^{**},z_2,\ldots,z_n^*)|\colon   |z_2-z_2^0|=\frac{r_2''\lambda_{2,2}(R'')}{l_2(z_1^{**},z_2^0,\ldots,z_n^*)} \right\}  \leq                \\ \leq  \prod_{j=1}^2   p_j(r'_j,r_j''\lambda_{2,j}(R''))
	\max \left\{ |F(z_1^{**},z_2,\ldots,z_n^*) |\colon   |z_2 - z_2^0| = \frac{r_2'}{l_2(z_1^{**},z_2^0,\ldots,z_n^*)}  \right\}  \leq
	\\   \leq   \prod_{j=1}^2   p_j(r'_j,r_j''\lambda_{2,j}(R''))
	\max \left\{ |F(z_1^{**},z_2,\ldots,z_n^*)|\colon   |z_2-z_2^0|=\frac{r_2'}{\lambda_{1,2}(R'')l_2(z^0)} \right\}  = \\
	=  \prod_{j=1}^2   p_j(r'_j,r_j''\lambda_{2,j}(R''))   |F(z_1^{**},z_2^{**},z_3^*,\ldots,z_n^*)| 
	\leq \ldots \leq \prod_{j=1}^n   p_j(r'_j,r_j''\lambda_{2,j}(R''))    \times   \\ \times
	\max\left\{|F(z_1,z_2,\ldots,z_n)|\colon  \ |z_j-z_j^0| = \frac{r_j'}{\lambda_{1,j}(R'')l_j(z^0)}, j\in\{1,\ldots,n\} \right\}=\\
	= \prod_{j=1}^n   p_j(r'_j,r_j''\lambda_{2,j}(R'')) M\left(\frac{R'}{\Lambda_1(R'')\mathbf{L}(z^0)},z^0,F\right).
	\end{gather*}
	Hence,
	by Theorem \ref{bordte43} $F$ is of bounded $\mathbf{L}$-index in joint variables.
\end{proof}


\section{Analogue of Theorem of Hayman for analytic in a ball function of bounded $\mathbf{L}$-index in joint variables}
Denote $\boldsymbol{\beta}=(\frac{\beta}{\sqrt{n}},\ldots,\frac{\beta}{\sqrt{n}}).$
\begin{theorem}
	\label{bordte15} Let $\mathbf{L}\in Q(\mathbb{B}^n).$  An analytic function $F$ in $\mathbb{B}^n$ has bounded $\mathbf{L}$-index in 
	joint variables  if and only if there exist $p\in\mathbb{Z}_+$ and $c\in\mathbb{R}_{+}$ 
	such that for each  $z\in\mathbb{B}^n$ 
	\begin{equation}
	\label{bordriv122}
	\max\left\{\frac{|F^{(J)}(z)|}{\mathbf{L}^J(z)}: \ \|J\|=p+1 \right\}\leq c\cdot   \max\left\{\frac{|F^{(K)}(z)|}{\mathbf{L}^K(z)}: \ \|K\|\leq p \right\}.
	\end{equation}
\end{theorem}
\begin{proof}
	
   Let $N=N(F,\mathbf{L},\mathbb{B}^n)<+\infty.$
	   The definition of the boundedness of $\mathbf{L}$-index in joint variables  yields the necessity with $p=N$ and $c=((N+1)!)^n.$

	   We prove the sufficiency. 
	   For $F\equiv 0$  theorem is obvious. Thus, we suppose that $F\not\equiv 0.$

	Assume that \eqref{bordriv122} holds, $z^0\in\mathbb{B}^n,$ $z\in \mathbb{D}^n[z^0,\frac{\boldsymbol{\beta}}{\mathbf{L}(z^0)}].$
	For all $J\in\mathbb{Z}^n_+,$ $\|J\|\leq p+1,$ one has 
	\begin{gather}
	\frac{|F^{(J)}(z)|}{\mathbf{L}^J(z^0)}\leq\Lambda_2^J(\boldsymbol{\beta}) \frac{|F^{(J)}(z)|}{\mathbf{L}^J(z)}\leq
	c\cdot\Lambda_2^J(\boldsymbol{\beta}) \max\left\{\frac{|F^{(K)}(z)|}{\mathbf{L}^K(z)}: \ \|K\|\leq p \right\} \leq \nonumber \\
	\leq c\cdot\Lambda_2^J(\boldsymbol{\beta}) \max\left\{\Lambda_1^{-K}(2) \frac{|F^{(K)}(z)|}{\mathbf{L}^K(z^0)}: \|K\|\leq p \right\} \leq BG(z), \label{bordriv123}
	\end{gather}
	where 
	$B=c\cdot\max\{\Lambda_2^{K}(\boldsymbol{\beta}): \ \|K\|=p+1\}\max\{\Lambda_1^{-K}(\boldsymbol{\beta}): \ \|K\|\leq p\},$
	and
	$ G(z)=\max\bigg\{\frac{|F^{(K)}(z)|}{\mathbf{L}^{K}(z^0)}: \ \|K\|\leq p\bigg\}.$
	
We choose $z^{(1)}\!=\!(z_1^{(1)},\ldots,z_n^{(1)})\!\in\! \mathbb{T}^n(z^0,\frac{\mathbf{1}}{2\beta\sqrt{n} \mathbf{L}(z^0)})$ and 
	$z^{(2)}\!=\!(z_1^{(2)},\ldots,z_n^{(2)})\in \mathbb{T}^n(z^0,\frac{\boldsymbol{\beta}}{\mathbf{L}(z^0)})$ such that 
	$F(z^{(1)})\neq 0$ and 
		\begin{equation} \label{choicez2}
	|F(z^{(2)})|=M\left( \frac{\boldsymbol{\beta}}{\mathbf{L}(z^0)},z^0,F\right) \neq 0.
	\end{equation}	
		
 These points exist, otherwise if $F(z)\equiv 0$ on skeleton
	$\mathbb{T}^n(z^0,\frac{\mathbf{1}}{2\beta\sqrt{n}\mathbf{L}(z^0)})$ or $\mathbb{T}^n(z^0,\frac{\boldsymbol{\beta}}{\mathbf{L}(z^0)})$ 
	  then 	by the uniqueness theorem $F\equiv 0$ in $\mathbb{B}^n.$
	We connect the points $z^{(1)}$ and $z^{(2)}$ with plane 
	$$
	\alpha=\left\{\begin{array}{r}
	z_2=k_2z_1+c_2,\\
	z_3=k_3z_1+c_3,\\
	\ldots \\
	z_n=k_nz_1+c_n,
	\end{array}
	\right.
	$$
	where 
	$k_i=\frac{z_i^{(2)}-z_i^{(1)}}{z_1^{(2)}-z_1^{(1)}},$ $c_i=\frac{z_i^{(1)}z_1^{(2)}-z_i^{(2)}z_1^{(1)}}{z_1^{(2)}-z_1^{(1)}},$ $i=2, \ldots, n.$
	It is easy to check   that $z^{(1)}\in\alpha$ and $z^{(2)}\in\alpha.$

  Let $\widetilde{G}(z_1)=G(z)|_{\alpha} $ be a restriction of the function $G$ onto $\alpha$.
  
  Fo every $K\in\mathbb{Z}^n_+$ the function $F^{(K)}(z)\big|_{\alpha}$ is analytic function of variable $z_1$ 
  and $\tilde{G}(z_1^{(1)})=G(z^{(1)})\big|_{\alpha}\neq 0$   	because $F(z^{(1)})\neq 0.$
  Hence, all zeros of the function  $F^{(K)}(z)\big|_{\alpha}$ are isolated as zeros of a function of one variable. Thus, zeros of the function   $\tilde{G}(z_1)$  are isolated too. Therefore, we can choose  piecewise analytic curve $\gamma$ onto $\alpha$ as following 
	$$z=z(t)=(z_1(t),k_2z_1(t)+c_2,\ldots,k_nz_1(t)+c_n), \   t\in[0,1],$$
		which connect the points $z^{(1)},$ $z^{(2)}$ and such that
 $G(z(t))\neq 0$  and 	$\int_0^1 |z_1'(t)|dt\leq \frac{2\beta}{\sqrt{n}l_1(z_1^0)}.$
	For a construction of the curve we connect $z^{(1)}_1$ and $z^{(2)}_1$ by a line $z^*_1(t)=(z_1^{(2)}-z_1^{(1)})t+z_1^{(1)},$ $t\in[0,1].$
The curve $\gamma$ can cross points $z_1$ at which the function $\widetilde{G}(z_1)=0.$ 
The number of such points $m=m(z^{(1)},z^{(2)})$ is finite. 
	Let $(z_{1,k}^{*})$ be a sequence of these points in ascending order of the value $|z_1^{(1)}-z_{1,k}^{*}|,$ $k\in \{1,2,...,m\}$.
We choose $r< \min\limits_{1\leq k\leq m-1} \{ |z^{*}_{1,k}-z^{*}_{1,k+1}|, |z^{*}_{1,1}-z^{(1)}_{1}|, |z^{*}_{1,m}-z^{(2)}_{1}|, 
\frac{2\beta^2-1}{2\pi\sqrt{n}\beta l_1(z^0)} \}.$ 
 Now  we construct circles with centers at the points $z_{1,k}^{*}$ and corresponding radii $r^{'}_k<\frac{r}{2^k}$ 
 such that $\widetilde{G}(z_1)\neq 0$ for all $z_1$ on the circles. It is possible, because $F\not\equiv  0$. 

Every such circle is divided onto two semicircles by the line $z_1^{*}(t)$.
The required piecewise-analytic curve consists with  arcs of the constructed semicircles  and segments of line $z_1^{*}(t)$, which
connect the arcs in series between themselves or with the points $z_1^{(1)}, z_1^{(2)}$. The length of $z_1(t)$ in $\mathbb{C}$ 
(but not $z(t)$ in $\mathbb{C}^n$!) 
is lesser than $\frac{\beta/\sqrt{n}}{l_1(z^0)}+\frac{1}{2\sqrt{n}\beta l_1(z^0)}+\pi r\leq \frac{2\beta}{\sqrt{n}l_1(z^0)}$.
	Then 
	\begin{gather*}
	\int_0^1|z_s'(t)|dt = |k_s|\int_0^1|z_1'(t)|dt\le \frac{|z_s^{(2)}-z_s^{(1)}|}{|z_1^{(2)}-z_1^{(1)}|} \frac{2\beta}{\sqrt{n}l_1(z^0)} \leq \\
	\leq \frac{2\beta^2+1}{2\sqrt{n}\beta l_s(z^0)} \frac{2\sqrt{n}\beta l_1(z^0)}{2\beta^2-1} \frac{2\beta}{\sqrt{n}l_1(z^0)} \leq
	 \frac{2\beta(2\beta^2+1)}{(2\beta^2-1)\sqrt{n}l_s(z^0)}, \ s\in\{2,\ldots,n\}.
	\end{gather*}
Hence, 
	\begin{equation}
	\label{bordriv124}
	\int_0^1 \sum_{s=1}^n l_s(z^0)|z_s'(t)|dt \leq  \frac{2\beta(2\beta^2+1)\sqrt{n}}{2\beta^2-1}=S.
	\end{equation}
   Since the function $z=z(t)$ is piece-wise analytic on $[0,1]$,
     then   for arbitrary $K\in\mathbb{Z}^n_+,$ $J\in\mathbb{Z}^n_+,$ $\|K\|\leq p,$ either
	\begin{equation} \label{identzero}
	\frac{|F^{(K)}(z(t))|}{\mathbf{L}^{K}(z^0)}\equiv \frac{|F{(J)}(z(t))|}{\mathbf{L}^{J}(z^0)},
	\end{equation}
    or the equality
	\begin{equation} \label{pointequal}
	\frac{|F^{(K)}(z(t))|}{\mathbf{L}^{K}(z^0)}= \frac{|F^{(J)}(z(t))|}{\mathbf{L}^{J}(z^0)}
	\end{equation}
	   holds only for a finite set of points $t_k\in[0;1]$.	
	
	 Then for function $G(z(t))$ as maximum of such expressions $\frac{|F^{(J)}(z(t))|}{\mathbf{L}^{J}(z^0)}$  by all $\|J\|\leq p$ two cases are possible:
		\begin{enumerate}
			\item In some interval of analyticity of the curve $\gamma$ the function $G(z(t))$  identically equals simultaneously to some derivatives, that is 
			  \eqref{identzero} holds. It means that $G(z(t))\equiv \frac{|F^{(J)}(z(t))|}{\mathbf{L}^{J}(z^0)}$
			for some $J,$ $\|J\|\leq p.$ Clearly, the function $F^{(J)}(z(t))$ is analytic. Then $|F^{(J)}(z(t))|$ is continuously differentiable function on the interval of analyticity except points where  this partial derivative equals zero $|F^{(j_1,j_2)}(z_1(t), z_2(t))|=0$. However, there are not the points, because in the opposite case $G(z(t))=0.$ But it contradicts the construction of the curve $\gamma$.
	\item 	In some interval of analyticity of the curve $\gamma$ the function $G(z(t))$ equals simultaneously to some derivatives at a finite number of points $t_k,$ that is \eqref{pointequal} holds. Then the points $t_k$ divide interval of analyticity onto a finite number of segments, in which of them $G(z(t))$ equals to one from the partial derivatives, i. e.  $G(z(t))\equiv \frac{|F^{(J)}(z(t))|}{\mathbf{L}^{J}(z^0)}$ for some $J,$ $\|J\|\leq p.$
	As above, in each from these segments the functions $|F^{(J)}(z(t))|,$ and $G(z(t))$ are continuously differentiable except the points $t_k$.
	\end{enumerate}
	
The inequality  $\frac{d}{dt}|f(t)|\leq |\frac{df(t)}{dt}|$  holds for complex-valued functions of real argument outside a countable set of points.
In view of this fact and \eqref{bordriv123}  we have 
	\begin{gather*}
	\frac{d}{dt}G(z(t))\leq \max\Big\{ \frac{1}{\mathbf{L}^J(z^0)}\Big|\frac{d}{dt} F^{(J)}(z(t))\Big|: \ \|J\|\leq p \Big\}\leq \\
	\leq \max\Big\{\sum_{s=1}^n \Big|\frac{\partial^{\|J\|+1}F}{\partial z_1^{j_1}\ldots \partial z_s^{j_s+1}\ldots \partial z_n^{j_n}}(z(t))\Big| \frac{|z_s'(t)|}{\mathbf{L}^j(z^0)}: \ \|J\|\leq p \Big\}\leq \\
	\leq \max\Big\{\sum_{s=1}^n \Big| \frac{\partial^{\|J\|+1}F}{\partial z_1^{j_1}\ldots \partial z_s^{j_s+1}\ldots \partial z_n^{j_n}}(z(t))\Big|
	\frac{l_s(z^0)|z_s'(t)|}{l_1^{j_1}(z^0)\ldots l_s^{j_1+1}(z^0)\ldots l_n^{j_n}(z^0)}: \\  \|J\|\leq p
	\Big\} \leq \Big( \sum_{s=1}^n l_s(z^0)|z_s'(t)|\Big) \max\Big\{ \frac{|F^{(j)}(z(t))|}{\mathbf{L}^J(z^0)}: \ \|J\|\leq p+1 \Big\} \leq \\ \leq
	\Big(\sum_{s=1}^n l_s(z^0)|z_s'(t)| \Big)BG(z(t)).
	\end{gather*}
	 Therefore, \eqref{bordriv124} yields 
	\begin{gather*}
	\Big|\ln \frac{G(z^{(2)})}{G(z^{(1)})} \Big|= \Big|\int_0^1 \frac{1}{G(z(t))} \frac{d}{dt} G(z(t)) dt \Big|
	\leq B \int_0^1 \sum_{s=1}^n l_s(z^0)|z_s'(t)| dt \leq S\cdot B.
	\end{gather*}
	   Using \eqref{choicez2}, we deduce
	$M(\frac{\boldsymbol{\beta}}{\mathbf{L}(z^0)},z^0,F)\leq G(z^{(2)})\leq G(z^{(1)})e^{SB}.$
	Since $z^{(1)}\in \mathbb{T}^n(z^0,\frac{\mathbf{1}}{2\beta \sqrt{n}\mathbf{L}(z^0)}),$  the Cauchy inequality holds
	$$
	\frac{|F^{(J)}(z^{(1)})|}{\mathbf{L}^J(z^0)}\leq J!(2\beta \sqrt{n})^{\|J\|}M\left(\frac{\mathbf{1}}{2\beta \sqrt{n}\mathbf{L}(z^0)},z^0,F\right).
	$$  for all $J\in\mathbb{Z}^n_+.$ 
	Therefore, for $\|J\|\leq p$ we obtain 
	$ G(z^{(1)})\leq (p!)^n (2\beta \sqrt{n})^p M\!\left(\!\frac{\mathbf{1}}{2\beta \sqrt{n}\mathbf{L}(z^0)},z^0,F\!\right),$
	$$M\left(\frac{\boldsymbol{\beta}}{\mathbf{L}(z^0)},z^0,F\right) \leq e^{SB} (p!)^n (2\beta \sqrt{n})^p M\left(\frac{\mathbf{1}}{2\beta \sqrt{n}\mathbf{L}(z^0)},z^0,F\right).
	$$
Hence, by Theorem  \ref{bordte43} the function $F$ has bounded $\mathbf{L}$-index in joint variables.
\end{proof}

\begin{theorem} \label{haymanjoint}
	Let $\mathbf{L} \in Q(\mathbb{B}^n)$. 
	 An analytic function $F$ in $\mathbb{B}^n$ has bounded $\mathbf{L}$-index in joint variables if and only if there exist $c\in (0;+\infty)$ and $N\in \mathbb{N}$ such that for each $z \in \mathbb{B}^n$ the inequality 
	\begin{gather}\label{th6main}
	\sum_{\|K\|=0}^{N}\frac{|F^{(K)}(z)|}{K!\mathbf{L}^{K}(z)}\geq c\sum_{\|K\|=N+1}^{\infty}\frac{|F^{(K)}(z)|}{K!\mathbf{L}^K(z)}.
	\end{gather}
\end{theorem}
\begin{proof}
	  Let $\frac{1}{\beta}<\theta_j<1,$ $j\in\{1,\ldots, n\},$ $\Theta=(\theta_1,\ldots,\theta_n).$ 
	   If the function $F$ has bounded $\mathbf{L}$-index in joint variables then by Theorem \ref{petr2} $F$ has bounded $\widetilde{\mathbf{L}}$-index in joint variables, where $\widetilde{\mathbf{L}}=(\widetilde{l}_1(z),\ldots,\widetilde{l}_n(z))$, $\widetilde{l}_j(z)=\theta_jl_j(z)$, $j\in\{1,\ldots,n\}$. 
	   Let $\widetilde{N}=N(F,\widetilde{L},\mathbb{B}^n).$
	  Therefore,
	 \begin{gather*}
	\max\left\{\frac{|F^{(K)}(z)|}{K!\mathbf{L}^{K}(z)}\colon \|K\|\leq \widetilde{N} \right\}
	=\max\left\{\frac{\Theta^K|F^{(K)}(z)|}
	{K!\widetilde{\mathbf{L}}^{K}(z)}\colon  \|K\|\leq \widetilde{N} \right\}\geq
	\nonumber \\ \geq
	\prod_{s=1}^n \theta_s^{\widetilde{N}}\max\left\{\frac{|F^{(K)}(z)|}
	{K!\widetilde{\mathbf{L}}^{K}(z)}\colon \|K\|\leq \widetilde{N} \right\}\geq
	\prod_{s=1}^n \theta_s^{\widetilde{N}}\frac{|F^{(J)}(z)|}
	{J!\widetilde{\mathbf{L}}^{J}(z)}
	\!= \prod_{s=1}^n \theta_s^{\widetilde{N}-j_s}\frac{|F^{(J)}(z)|}
	{J!\mathbf{L}^J(z)}
	\end{gather*}
	for all $J\geq \mathbf{0}$ and 
	\begin{gather*}
	\sum_{\|J\|=\widetilde{N}+1}^{\infty}\frac{|F^{(J)}(z)|}{J!\mathbf{L}^{j}(z)} \leq
	\max\left\{\frac{|F^{(K)}(z)|}{K!\mathbf{L}^K(z)}\colon  \|K\|\leq \widetilde{N} \right\} \sum_{\|J\|=\widetilde{N}+1}^{\infty}\theta_s^{j_s-\widetilde{N}} = \nonumber \\
	=\prod_{i=1}^n \frac{\theta_s}{1-\theta_s} \max\left\{\frac{|F^{(K)}(z)|}{K!\mathbf{L}^{K}(z)}\colon \|K\|\leq \widetilde{N} \right\}
	\leq \prod_{i=1}^n \frac{\theta_s}{1-\theta_s}
	\sum_{\|K\|=0}^{\widetilde{N}}\frac{|F^{(K)}(z)|}{K!\mathbf{L}^{K}(z)}.
	\end{gather*}
	  Hence, we obtain \eqref{th6main} with $N=\widetilde{N}$ and $c=\prod_{i=1}^n \frac{\theta_s}{1-\theta_s}.$
	   On the contrary, inequality \eqref{th6main} implies 
	\begin{gather*}
	\max\left\{\frac{|F^{(J)}(z)|}{J!\mathbf{L}^{J}(z)}\colon \|J\|=N+1 \right\} 
	\leq \sum_{\|K\|=N+1}^{\infty}\frac{|F^{(K)}(z)|}{K!\mathbf{L}^{K}(z)}
	\leq\frac{1}{c}\sum_{\|K\|=0}^{N}\frac{|F^{(K)}(z)|}{K!\mathbf{L}^{K}(z)}\leq \nonumber \\
	\leq \frac{1}{c}\sum_{i=0}^N C_{n+i-1}^i\max \left\{ \frac{|F^{(K)}(z)|}{K!\mathbf{L}^{K}(z)}\colon \|K\|\le N \right\}
	\end{gather*}
	and by Theorem  \ref{bordte15} $F$ is of bounded $\mathbf{L}$-index in joint variables.
\end{proof}

\section{Properties of power series of analytic  in a ball functions of bounded $\mathbf{L}$-index in joint variables}
Let $z^0\in \mathbb{B}^n$. 
 We develop an analytic in $\mathbb{B}^n$ function $F$ in the power series written in a diagonal form
\begin{gather}\label{th7main}
F(z)=\sum_{k=0}^{\infty}p_{k}((z-z^0))=
\sum_{k=0}^{\infty}\sum_{\|J\|=k}b_{J}(z-z^0)^{J},
\end{gather}
where $p_{k}$ are homogeneous polynomials of $k$-th degree, $b_{J}=\frac{F^{(J)}(z^0)}{J!}$.
The polynomial $p_{k_0}, k_0\in \mathbb{Z}_+,$ is called a dominating polynomial in the power series expansion \eqref{th7main} on $\mathbb{T}^n(z^0,R)$ if for every $z\in \mathbb{T}^n(z^0,R)$ the next inequality holds:
\begin{gather*}
|\sum_{k\neq k^0}p_{k}(z-z^0)|\leq \frac{1}{2}
\max\{|b_{J}|R^J \colon \|J\|=k^0 \}.
\end{gather*}
\begin{theorem} \label{bordmainpol}
	Let $\mathbf{L} \in Q(\mathbb{B}^n)$. 
	If an analytic function $F$ in $\mathbb{B}^n$ has bounded $\mathbf{L}$-index in joint variables then  there exists $p\in \mathbb{Z}_+$ that for all $d\in (0;\frac{\beta}{\sqrt{n}}]$ there exists $\eta(d)\in (0;d)$ such that for each $z^0 \in \mathbb{B}^n$ and some $r=r(d,z^0) \in(\eta (d),d),$ $k^0=k^0(d,z^0)\leq p$ the polynomial $p_{k^0}$ is the dominating polynomial in the series \eqref{th7main} on $\mathbb{T}^n(z^0,\frac{r\mathbf{1}}{\mathbf{L}(z^0)}).$
 \end{theorem}
\begin{proof}
		Let $F$ be an analytic function of bounded $\mathbf{L}$-index in joint variables with $N=N(F,\mathbf{L},\mathbb{B}^n)<+\infty$ and
	$n_0$ be the $\mathbf{L}$-index in joint variables at a point $z^0\in\mathbb{D}^2,$ 
	i.e. $n_0$ is the least number, for which inequality \eqref{ineqoz2} holds at the point $z^0.$
	Then for each $z^0 \in \mathbb{B}^n$ \ $n_0\leq N$.
	
	We put 
	$ a_J^*=\frac{|b_J|}{\mathbf{L}^J(z^0)}= \frac{|F^{(J)}(z^0)|}{J!\mathbf{L}^J(z^0)}, \  a_k=\max\{a^*_J: \ \|J\|=k \}, $
	$c=2\{(N+n+1)!(n+1)!+ (N+1)C_{n+N-1}^N\}.$
	Let $d\in (0;\frac{\beta}{\sqrt{n}}]$ be an arbitrary number.
	 We also denote $r_m=\frac{d}{(d+1)c^m},$ 
	$\mu_m=\max\{a_kr_m^k: \ k\in\mathbb{Z}_+\},$
	\ $s_m=\min\{k: \ a_kr_m^k=\mu_m\}$  for  $m\in \mathbb{Z}_{+}.$

	Since $z^0\in\mathbb{B}^n$ is a fixed point the inequality $a^*_K\leq \max\{a^*_J:\ \|J\|\leq n_0\}$ is valid for all $K\in\mathbb{Z}^n_+.$
	 Then $a_k\leq a_{n_0}$ for all $k\in \mathbb{Z}_+.$ Hence, for all 
	$k>n_0,$ in view of $r_0<1,$  we have $a_kr_0^k<a_{n_0}r_0^{n_0}.$ 
	This implies $s_0\leq n_0.$
	Since $cr_m=r_{m-1},$ we obtain that for each  $k>s_{m-1}$  ( $r_{m-1}<1$) 
	\begin{equation}
	a_{s_{m-1}}r_m^{s_{m-1}}=a_{s_{m-1}}r_{m-1}^{s_{m-1}}c^{-s_{m-1}}\geq a_k r^k_{m-1}c^{-s_{m-1}}=
	a_k r^k_mc^{k-s_{m-1}}\geq ca_k r^k_m. \label{bordriv127}
	\end{equation}
	It yields that $s_m\leq s_{m-1}$ for all $m\in\mathbb{N}.$ Thus, we can rewrite 
	$$\mu_0=\max\{a_kr_0^k: \ k\leq n_0\},  \ \mu_m=\max\{a_kr_m^k: \ k\leq s_{m-1}\}, \ m\in\mathbb{N}
	$$
	Let us to introduce additional notations for $m\in\mathbb{N}$
	\begin{gather*}
	\mu_0^*=\max\left\{a_kr_0^k: \ s_0\neq k\leq n_0\right\}, \
	s_0^*=\min\{k: \ k\neq s_0, a_kr_0^k=\mu_0^*\},\\
	\mu_m^*=\max\{a_kr_m^k: \ s_m\neq k\leq s_{m-1}\}, \
	s_m^*=\min\{k: \ k\neq s_m, a_kr_m^k=\mu^*_m\}.
	\end{gather*}
We will show that there exists $m_0 \in\mathbb{Z}_+$ such that 
	\begin{equation}
	\label{bordriv128}
	\frac{\mu^*_{m_0}}{\mu_{m_0}}\leq \frac{1}{c}.
	\end{equation}
		Suppose that for all $m\in \mathbb{Z}_+$ the next inequality holds
	\begin{equation}
	\label{bordriv129}
	\frac{\mu^*_{m}}{\mu_{m}}> \frac{1}{c}.
	\end{equation}
		If $s_m^{*}<s_m$ ($s_m^{*}\neq s_m$ in view of definition) then we have
	$$
	a_{s^*_m}r_{m+1}^{s^*_m}= \frac{a_{s^*_m}r_m^{s^*_m}}{c^{s^*_m}}=\frac{\mu^*_m}{c^{s^*_m}}>
	\frac{\mu_m}{c^{s^*_m+1}}=\frac{a_{s_m}r_m^{s_m}}{c^{s^*_m+1}}=\frac{a_{s_m}r_{m+1}^{s_m}}{c^{s^*_m+1-s_m}}\geq
	a_{s_m}r^{s_m}_{m+1}.
	$$
	Besides, for every $k>s^*_m,$ $k\neq s_m,$ (i. e., $k-1\geq s^*_m$) it can be deduced similarly that 
	$$
	a_{s^*_m}r_{m+1}^{s^*_m}= \frac{a_{s^*_m}r_m^{s^*_m}}{c^{s^*_m}}\geq  \frac{a_kr^k_m}{c^{s^*_m}}\geq  \frac{a_kr^k_m}{c^{k-1}}=
	c a_kr^k_{m+1}.
	$$
	Hence, $a_{s^*_m}r_{m+1}^{s^*_m} > a_k r_{m+1}^k $ for all $k>s^*_m.$ Then
	\begin{equation}
	\label{bordriv130}
	s_{m+1}\leq s^*_m\leq s_m-1.
	\end{equation}
	On the contrary, if $s_m<s^*_m\leq s_{m-1},$  then the equality  $s_{m+1}=s_m$ may holds.
	Indeed, by definition $s_{m+1}\leq s_m.$ It means that the specified equality is possible. 
	But if 	$s_{m+1}<s_m$ then $s_{m+1}\leq s_{m}-1$ (they are natural numbers!). 
	Hence, we obtain \eqref{bordriv130}.
	
	Thus, the inequalities $s^*_{m+1}\leq s_m$ and $s^*_m\neq s_{m+1}$ imply that
	$s^*_{m+1}<s_{m+1}.$
	As above instead \eqref{bordriv130} we have 
	$$s_{m+2}\leq s_{m+1}^*\leq s_{m+1}-1=s_m-1.$$
	
	Therefore, if for all $m\in\mathbb{Z}_+$ \ \eqref{bordriv129} holds, then for every $m\in\mathbb{Z}_+$ either 
	$s_{m+2}\leq s_{m+1}\leq s_m-1$  or  $s_{m+2}\leq s_{m}-1$ holds, that is  $s_{m+2}\leq s_{m}-1,$
	because $s_{m+2}\leq s_{m+1}.$
	 It follows that 
	$$
	s_m\!\leq\! s_{m-2}-1\leq \ldots \leq s_{m-2[m/2]}-[m/2]\leq s_0-[m/2] \leq
	n_0-[m/2]\leq N-[m/2].
	$$
  In other words, $s_m<0$ for $m>2N+1,$ which is impossible. 
  	Therefore, there exists $m_0\leq 2N+1$ such that \eqref{bordriv128} holds.
  We put $r=r_{m_0},$ $\eta(d)=\frac{d}{(d+1)c^{2(N+1)}},$ $p=N$ and $k_0=s_{m_0}.$ 
  Then for all $\|J\|\neq k_0=s_{m_0}$  in
	$\mathbb{T}^n(z^0,\frac{r\mathbf{1}}{\mathbf{L}(z^0)}),$ in view \eqref{bordriv127} and \eqref{bordriv128} we obtain 
	$$|b_J||z-z^0|^J=a_J^*r^{\|J\|}\leq a_{\|J\|}r^{\|J\|}\leq \frac{1}{c} a_{s_{m_0}}r_{m_0}^{s_{m_0}}=\frac{1}{c} a_{k_0}r^{k_0}.
	$$
	Thus, for  $z\in\mathbb{T}^n(z^0,\frac{r\mathbb{1}}{\mathbf{L}(z^0)})$ 
	\begin{gather}
	\left|\sum_{\|J\|\neq k_0}b_J(z-z^0)^J\right| \leq \sum_{\|J\|\neq k_0}a_j^*r^{\|J\|} \leq
	\sum_{\substack{k=0,\\ k\neq k_0}}^{\infty}a_kC^k_{n+k-1} r^k= \nonumber\\
	= \sum_{\substack{k=0,\\ k\neq s_{m_0}}}^{s_{m_0-1}}a_k C^k_{n+k-1} r^k+ \sum_{k=s_{m_0-1}+1}^{\infty} a_kC_{n+k-1}^k r^k.
	\label{bordriv131}
	\end{gather}
	We will estimate two sums in \eqref{bordriv131}.
	Taking into account \eqref{bordriv130}, it can established that
	\begin{equation}
	\label{bordriv132}
	\sum_{\substack{k=0,\\ k\neq s_{m_0}}}^{s_{m_0-1}}a_k C_{n+k-1}^k r^k \leq \frac{a_{k_0}r^{k_0}}{c} \sum_{k=0}^{N}C_{n+k-1}^k \leq
	\frac{a_{k_0}r^{k_0}}{c} (N+1)C_{n+N-1}^N.
	\end{equation}
	For all $k\geq s_{m_0-1}+1$ \ $a_kr^k_{m_0-1}\leq \mu_{m_0-1}$ holds. Then 
	$a_kr^k_{m_0}= \frac{a_kr^k_{m_0-1}}{c^k}\leq \frac{\mu_{m_0-1}}{c^k}. $
	In view of  \eqref{bordriv128} we deduce 
	\begin{gather}
	\sum_{k=s_{m_0-1}+1}^{\infty} a_k C_{n+k-1}^k r^k\leq \mu_{m_0-1} \sum_{k=s_{m_0-1}+1}^{\infty}C_{n+k-1}^k \frac{1}{c^k} \leq \nonumber\\
	\leq a_{s_{m_0-1}}r_{m_0}^{s_{m_0-1}}c^{s_{m_0-1}} \sum_{k=s_{m_0-1}+1}(k+1)(k+2)\ldots (k+n)\frac{1}{c^k} \leq \nonumber\\
	\!\leq\! \frac{a_{s_{m_0}}r^{s_{m_0}}}{c}c^{s_{m_0-1}}\!\bigg(\!\! \sum_{k=s_{m_0-1}\!+\!1}^{\infty} x^{k+n}\!\!\bigg)^{(n)}\Bigg|_{x\!=\!\frac{1}{c}}\!=\!
	\frac{a_{k_0}r^{k_0}}{c}c^{s_{m_0\!-\!1}} \!\bigg\{\frac{x^{s_{m_0-1}\!+\!n\!+\!1}}{1\!-\!x}\!\bigg\}^{(n)}\Bigg|_{\!x\!=\!\frac{1}{c}}\!=\! \nonumber\\
	= \frac{a_{k_0}r^{k_0}}{c}c^{s_{m_0-1}}\sum_{j=0}^n C_n^j(n-j)! (s_{m_0-1}+n+1)\ldots (s_{m_0-1}+n-j+2) \times \nonumber \\
	\times   \left. \frac{x^{s_{m_0-1}+1+n-j}}{(1-x)^{n-j+1}}\right|_{x=\frac{1}{c}} \!\leq\!
	\frac{a_{k_0}r^{k_0}}{c}c^{s_{m_0}-1}n!(N+n+1)! \sum_{j=0}^n \frac{(1/c)^{s_{m_0-1}+1+n-j}}{(1-1/c)^{n-j+1}}\!=\! \nonumber \\
	=n! (N+n+1)! \frac{a_{k_0}r^{k_0}}{c} \sum_{j=0}^n \frac{1}{(c-1)^{n-j+1}}\leq (n+1)!(N+n+1)!  \frac{a_{k_0}r^{k_0}}{c},
	\label{bordriv133}
	\end{gather}
	because $c\geq 2.$ Hence, from \eqref{bordriv131}-\eqref{bordriv133} it follows that
	$$
	\left|\sum_{\|J\|\neq k_0}b_J(z-z^0)^J \right| \leq \frac{((N+1)C_{n+N-1}^N+(n+1)!(N+n+1)!)a_{k_0}r^{k_0}}{c} \leq
	\frac{1}{2} a_{k_0}r^{k_0}.
	$$
	It means that the polynomial  $P_{k_0}$ is the dominating polynomial in the series \eqref{th7main} on skeleton $\mathbb{T}^n(z^0,\frac{r\mathbf{1}}{\mathbf{L}(z^0)}).$
	\end{proof}

	\begin{theorem} \label{petr8}
	Let  $\mathbf{L} \in Q(\mathbb{B}^n)$. If there exist $p\in \mathbb{Z}_+,$ $d\in (0;1],$  $\eta\in (0;d)$ such that for each $z^0 \in \mathbb{B}^n$ and some $R=(r_1,\ldots,r_n)$ with $r_j=r_j(d,z^0) \in(\eta,d),$ $j\in\{1,\ldots,n\},$ and certain $k^0=k^0(d,z^0)\leq p$ the polynomial $p_{k^0}$ is the dominating polynomial in the series \eqref{th7main} on $\mathbb{T}^2(z^0,{R}/{\mathbf{L}(z^0)})$
	then the analytic  in $\mathbb{B}^n$ function $F$ has bounded $\mathbf{L}$-index in joint variables.
\end{theorem}
\begin{proof}
Suppose that there exist $p\in\mathbb{Z}_+,$ $d\le 1$  and  $\eta\in(0;d)$  such that for each 
	$z^0\in\mathbb{B}^n$ and some 
	$R=(r_1,\ldots,r_n)$ with $r_j=r_j(d,z^0) \in(\eta,d),$ $j\in\{1,\ldots,n\},$ and 
	 $k_0=k_0(1,z^0)\leq p$ 
	the polynomial  $P_{k_0}$ is the dominating polynomial in the series
	\eqref{th7main} on $\mathbb{T}^n(z^0,\frac{R}{\mathbf{L}(z^0)}).$ 
		Let us to denote $r_0=max_{1\leq j\leq n} r_j.$
	Then 
	$$
	\left|\sum_{\|J\|\neq k_0}b_J(z-z^0)^J\right|= \left|f(z)- \sum_{\|J\|=k_0}b_J(z-z^0)^J\right|\leq \frac{a_{k_0}r_0^{k_0}}{2}.
	$$
		Using Cauchy's inequality we have
	$|b_J(z-z^0)^J|=a^*_jR^{J}\leq \frac{a_{k_0}r_0^{k_0}}{2}$
	for all $J\in\mathbb{Z}^n_{+},$ $\|J\|\neq k_0,$ that is for all $\|J\|=k\neq k_0$
	\begin{equation}
	\label{bordriv134}
	a_kR^J\leq \frac{a_{k_0}r_0^{k_0}}{2}.
	\end{equation}
	Suppose that $F$ is not a function of bounded $\mathbf{L}$-index in joint variables. 
Then in view of Theorem \ref{bordte15} for all $p_1\in\mathbb{Z}_+$ and $c\geq 1$ there exists $z^0\in\mathbb{B}^n$ such that the next inequality holds:
	$$
	\max\left\{\frac{|F^{(J)}(z^0)|}{\mathbf{L}^J(z^0)}: \ \|J\|=p_1+1 \right\}>c
	\max\left\{\frac{|F^{(K)}(z^0)|}{\mathbf{L}^K(z^0)}: \ \|K\|\leq p_1\right\}.
	$$
We put $p_1=p$ and $c=\left( \frac{(p+1)!}{\eta^{p+1}}\right)^n.$  Then for this $z^0(p_1,c)$ 
	$$
	\max\left\{\frac{|F^{(J)}(z^0)|}{J!\mathbf{L}^J(|z^0|)}: \ \|J\|=p+1 \right\}> \frac{1}{\eta^{p+1}}
	\max\{\frac{|F^{(K)}(z^0)|}{K!\mathbf{L}^K(|z^0|)}: \ \|K\|\leq p\},
	$$
	that is $a_{p+1}> \frac{a_{k_0}}{\eta^{p+1}}.$ Hence, 
	$a_{p+1}r_0^{p+1}> \frac{a_{k_0}r_0^{p+1}}{\eta^{p+1}}\geq a_{k_0}r^{k_0}.$
	The last inequality  contradicts \eqref{bordriv134}. Therefore, $F$ is of bounded $\mathbf{L}$-index in joint variables.
\end{proof}

%

\section{Properties of $Q(\mathbb{B}^n).$}
\begin{theorem} \label{boundlogderjoint} 
	Let $\mathbf{L}(z)=(l_1(z),\ldots, l_{n}(z)),$ $l_j: \mathbb{B}^n \to \mathbb{C}$ and $\frac{\partial l_j}{\partial z_m}$ be continuous functions in $\mathbb{B}^n,$  
 for all $j,$ $m\in\{1,2,\ldots,n\}.$ 
 	If for every $ j\in\{1,2,\ldots,n\}$ $|l_j(z)|$ satisfies \eqref{Lbeta-ball} 
 	and there exist  $P>0$ and $c>0$ such that
	for all $z\in\mathbb{B}^n$ and every $ j, m\in\{1,2,\ldots,n\}$
	\begin{equation} \label{condqnb}
	\frac{1}{c+|l_j(z)|}\left|\frac{\partial l_j(z)}{\partial z_m}\right|\leq P 
	\end{equation}
	then $\mathbf{L}^*\in Q(\mathbb{B}^n),$ where $\mathbf{L}^*(z)=(c+|l_1(z)|,\ldots, c+|l_n(z)|).$
\end{theorem}	
\begin{proof}
	Clearly, the function $\mathbf{L}^*(z)$ is positive and continuous.
	For given $z\in\mathbb{B}^n,$ $z^0\in\mathbb{B}^n$ we define an analytic curve $\varphi: [0,1]\to \mathbb{B}^n$
	$$\varphi_j(\tau)=z^0_j+\tau (z_j-z_j^0), \ j\in\{1,2,\ldots,n\},$$
	where $\tau\in [0,1].$ 
	It is known that for every continuously differentiable function  $g$ of real variable $\tau$  the inequality  $\frac{d}{dt}|g(\tau)|\leq |g'(\tau)|$  holds except the points  where $g(\tau)=0.$
	Using restrictions of this lemma, we establish the upper estimate of $\lambda_{2,j}(z_0,R):$
	\begin{gather*}
	\lambda_{2,j}(z_0,R)=\sup \left\{\frac{c+|l_j(z)|}{c+|l_j(z^0)|}\colon z\in \mathbb{D}^n\left[z^0,\frac{R}{\mathbf{L}^*(z^0)}\right]   \right\}= \\ 
	= \sup_{z\in \mathbb{D}^n\left[z^0,\frac{R}{\mathbf{L}^*(z^0)}\right]  } \left\{ \exp\left\{\ln(c+|l_j(z)|) - \ln(c+|l_j(z^0)|)\right\}
	\right\} = \\
	\!=\! \sup\! \left\{\exp\left\{ \int_0^{1} \frac{d(c+|l_j(\varphi(\tau))|)}{c+|l_j(\varphi(\tau))|} \right \}: z\in \mathbb{D}^n\left[z^0,\frac{R}{\mathbf{L}^*(z^0)}\right]  \! \right\} \!\leq\! \\
	\!\leq\! \sup_{ z\in \mathbb{D}^n\left[z^0,\frac{R}{\mathbf{L}^*(z^0)}\right] }\!\left\{\! \exp\!\left\{\!\int_0^{1}\! \sum_{m=1}^n \frac{|\varphi_m'(\tau)|}{c+|l_j(\varphi(\tau))|} 
	\left|\frac{\partial l_j(\varphi(\tau))}{\partial z_m}\right| d\tau \right\}  \!\right \} \!\leq\! \\
	\leq \sup_{  z\in \mathbb{D}^n\left[z^0,\frac{R}{\mathbf{L}^*(z^0)}\right] } \left\{ \exp\left\{ \int_0^1 \sum_{m=1}^n P|z_m-z_m^0| d\tau \right\}
	\right\} \leq \\ \leq 
	\sup_{z\in \mathbb{D}^n\left[z^0,\frac{R}{\mathbf{L}^*(z^0)}\right] } \left\{ \exp\left\{ \sum_{m=1}^n \frac{Pr_j}{c+|l_m(z^0)|} \right\}
	\right\}
	\leq \exp\left(\frac{P}{c}\sum_{m=1}^nr_j \right).
	\end{gather*}
	Hence,  for all $R\geq \mathbf{0}$ \ $\lambda_{2,j}(R)=\sup\limits_{z^0\in\mathbb{B}^n}  \lambda_{2,j}(z^0,\eta) \leq \exp\left(\frac{P}{c}\sum\limits_{m=1}^nr_j \right) <\infty.$
	Using $\frac{d}{dt}|g(t)|\geq -|g'(t)|$ it can be proved that for every $\eta\geq 0$ \  $\lambda_{1,j}(R)\geq \exp\left(-\frac{P}{c}\sum\limits_{m=1}^nr_j \right)>0.$ Therefore, $\mathbf{L}^*\in Q(\mathbb{B}^n).$
\end{proof}

%

At first we prove the following lemma. 
\begin{lemma}  \label{qngrowth}
	If $\mathbf{L}\in Q(\mathbb{B}^n)$ then  for every $ j\in\{1,\ldots,n\}$ and for every fixed $z^*\in\mathbb{B}^n$ $|z_j|l_j(z^*+z_j\mathbf{1}_j)\to \infty$ as $|z^*+z_j\mathbf{1}_j|\to 1-0.$ 
\end{lemma}
\begin{proof}
	In view of \eqref{Lbeta-ball} we have $l_j(z^*+z_j\mathbf{1}_j)\geq \frac{\beta}{1-|z^*+z_j\mathbf{1}_j|} \to+\infty$ 
	as $|z^*+z_j\mathbf{1}_j|\to 1-0.$ 
\end{proof}

\section{Estimates of growth of analytic in ball functions} 
Results in this section are similar to results obtained for entire functions in \cite{growth-joint-entire}.
Denote
$[0,2\pi]^n=\underbrace{[0,2\pi]\times \cdots \times[0,2\pi]}_{n-\text{th times}}.$
For $R=(r_1,\ldots,r_n)\in\mathbb{R}^n_{+},$ $\Theta=(\theta_1,\ldots,\theta_n)\in[0,2\pi]^n,$ $A=(a_1,\ldots,a_n)\in\mathbb{C}^n$ let us to write 
$Re^{i\Theta}=(r_1e^{i\theta_1},\ldots,r_ne^{i\theta_n}),$ 
$\mathop{arg} A=(\mathop{arg}a_1,\ldots,\mathop{arg}a_n).$

By $K(\mathbb{B}^n)$ we denote a class of positive continuous functions $\mathbf{L}(z)=(l_1(z),\ldots,l_n(z)),$ where $l_j:\mathbb{B}^n\to \mathbb{R}_+$  
satisfy \eqref{Lbeta-ball} 
and there exists $c\geq 1$  such that for
every $R\in\mathbb{R}^n_+$ with $|R|< 1$ and $j\in\{1,\ldots,n\}$  
$$
\max_{\Theta_1,\Theta_2\in[0,2\pi]^n} \frac{l_j(Re^{i\Theta_2})}{l_j(Re^{i\Theta_1})}\leq c.
$$
If $\mathbf{L}(z)=(l_1(|z_1|,\ldots,|z_n|),\ldots,l_n(|z_1|,\ldots,|z_n|))$ then $\mathbf{L}\in K(\mathbb{B}^n).$  
It is easy to prove that $\frac{|e^z| +1}{1-|z|}\in Q(\mathbb{D})\setminus K(\mathbb{D}),$ but $\frac{e^{e^{|z|}}}{1-|z|}\in K(\mathbb{D})\setminus Q(\mathbb{D}).$ 
Besides, if $\mathbf{L}_1, \mathbf{L}_2\in K(\mathbb{B}^n)$ then $\mathbf{L}_1+\mathbf{L}_2\in K(\mathbb{B}^n)$ and $\mathbf{L}_1\mathbf{L}_2\in K(\mathbb{B}^n).$
For simplicity, let us to write $M(F,R)= \max \{|F(z)|\colon z\in \mathbb{T}^n(\mathbf{0},R)\},$ where $|R|< 1.$ 
Denote $\boldsymbol{\beta}=(\frac{\beta}{c\sqrt{n}},\ldots,\frac{\beta}{c\sqrt{n}}).$
\begin{theorem} \label{thgrowthobig}
	Let $\mathbf{L}\in Q(\mathbb{B}^n)\cap K(\mathbb{B}^n),$ $\beta>c\sqrt{n}.$ If an analytic in $\mathbb{B}^n$ function $F$ has bounded $\mathbf{L}$-index in joint variables, then 
	\begin{equation}
	\label{growthobig}
	\ln M(F,R) \!=\! O\left(\min_{\sigma_n\in \mathcal{S}_n} \min_{\Theta\in[0,2\pi]^n} \sum_{j=1}^n \int_{0}^{r_j} l_j(R(j,\sigma_n,t)e^{i\Theta}) dt\right)
	\text{ as } \|R\|\to 1-0,
	\end{equation}
	where $\sigma_n$ is a permutation of $\{1,\ldots,n\},$ $\mathcal{S}_n$ is a set of all permutations of  $\{1,\ldots,n\},$ $R(j,\sigma_n,t)\!=(r'_1,\ldots,r'_n),$  $r'_k\!=\begin{cases}
	r^0_k, \text{ if } \sigma_n(k)<j,\\
	t, \text{ if } k=j,\\
	r_k, \text{ if } \sigma_n(k)>j,
	\end{cases}$
	$k\in \{1,\ldots,n\},$ 
	$R^0=(r_1^0,\ldots,r_n^0)$ is fixed radius. 	
\end{theorem}

\begin{proof}
	Let $R>\mathbf{0},$ $|R|<1,$ $\Theta\in[0,2\pi]^n$ and the point $z^*\in \mathbb{T}^n(\mathbf{0},R+\frac{\boldsymbol{\beta}}{\mathbf{L}(Re^{i\Theta})})$ be a such that 
	$$|F(z^*)|=\max\left\{|F(z)|: z\in \mathbb{T}^n\left(\mathbf{0},R+\frac{\boldsymbol{\beta}}{\mathbf{L}(Re^{i\Theta})}\right)\right\}.$$
	Denote $z^0=\frac{z^*R}{R+{\boldsymbol{\beta}}/\mathbf{L}(Re^{i\Theta})}.$ Then 
	\begin{gather*}
	|z^0_j-z^*_j|=
	\left|\frac{z_j^*r_j}{r_j+\frac{\beta}{c\sqrt{n} l_j(Re^{i\Theta})}} -z^*_j\right| = 
	\left|\frac{z^*_j{\beta}/({c\sqrt{n} l_j(Re^{i\Theta})})}{r_j+\frac{\beta}{c\sqrt{n} l_j(Re^{i\Theta})}} \right| = 
   \frac{\beta}{c\sqrt{n} l_j(Re^{i\Theta})} \text{ and } \\ 
	\!\mathbf{L}(z^0)\!=\!\mathbf{L} \!\left(\!\frac{z^*R}{R+\boldsymbol{\beta}/\mathbf{L}(Re^{i\Theta})}\!\right)\! =\! \mathbf{L} 
	\left(\frac{(R+\boldsymbol{\beta}/\mathbf{L}(Re^{i\Theta}))e^{i\mathop{arg}z^*}R}{R+\boldsymbol{\beta}/\mathbf{L}(Re^{i\Theta})}\right)\! =\! 
	\mathbf{L} (Re^{i\mathop{arg} z^*}).
	\end{gather*}
	Since $\mathbf{L}\in K(\mathbb{B}^n)$ we have that $c\mathbf{L}(z^0)= c\mathbf{L}(Re^{i\mathop{arg}z^*})\geq \mathbf{L}(Re^{i\Theta}) \geq \frac{1}{c}\mathbf{L}(z^0).$
	We consider two 
	skeletons $\mathbb{T}^n(z^0,\frac{\mathbf{1}}{\mathbf{L}(z^0)})$ and $\mathbb{T}^n(z^0,\frac{\boldsymbol{\beta}}{\mathbf{L}(z^0)}).$ 
	By Theorem \ref{bordte43} there exists $p_1=p_1(\frac{\mathbf{1}}{c},c\boldsymbol{\beta})\geq 1$ such that 
	\eqref{bordriv112nec} holds with $R'=\frac{\mathbf{1}}{c},$ $R''=c\boldsymbol{\beta},$ 
	i.e. 
	\begin{gather}
	\! \max\!\left\{|F(z)| \colon z\!\in\! \mathbb{T}^n\left(\mathbf{0},R+\frac{\boldsymbol{\beta}}{\mathbf{L}(Re^{i\Theta})}\!\right)\!\right\}=|F(z^*)|\leq \nonumber \\ \leq  
	\max\left\{|F(z)| \colon z\!\in\! \mathbb{T}^n\left(z^0,\frac{\boldsymbol{\beta}}{\mathbf{L}(Re^{i\Theta})}\right)\right\}  \leq \nonumber \\ 
	\leq       \max\left\{|F(z)| \colon z\in \mathbb{T}^n\left(z^0,\frac{c\boldsymbol{\beta}}{\mathbf{L}(z^0)}\right)\right\} \!\leq\! \nonumber \\ \leq 
	p_1   \max\left\{|F(z)| \colon z\in \mathbb{T}^n\left(z^0,\frac{\mathbf{1}}{c\mathbf{L}(z^0)}\right)\right\} \leq \nonumber \\ 
	\leq p_1\max\left\{|F(z)|: z\in \mathbb{T}^n\left(\mathbf{0},R+\frac{\mathbf{1}}{\mathbf{L}(Re^{i\Theta})}\right)\right\}
	\label{bordriv22}
	\end{gather}
	The function $\ln^+ \max \{|F(z)|\colon z\in \mathbb{T}^n(\mathbf{0},R)\}$ is a convex function of the variables  $\ln{r_1},$ $\ldots,$ $\ln{r_n}$ (see \cite{ronkin}, p.~138 in Russian edition or p.~84 in English translation). Hence, the function  admits a representation  
	\begin{gather}
	\ln^+ \max \{|F(z)|\colon z\in \mathbb{T}^n(\mathbf{0},R)\} -\nonumber \\ -
	\ln^+ \max \{|F(z)|\colon z\in \mathbb{T}^n(\mathbf{0},R+(r_j^0-r_j)\mathbf{1}_j)\}  = \nonumber\\ =
	\int_{r_j^0}^{r_j} \frac{A_j(r_1,\ldots,r_{j-1},t,r_{j+1},\ldots,r_n)}{t}dt \label{presentconvex}
	\end{gather}
	for arbitrary $0<r_j^0\leq r_j,$ where the functions $A_j(r_1,\ldots,r_{j-1},t,r_{j+1},$ $\ldots,r_n)$ are positive non-decreasing in variable $t,$ 
	$j\in\{1, \ldots,n\}.$
	
	Using \eqref{bordriv22} we deduce 
	\begin{gather}
	\ln p_1 \ge \ln\max\left\{|F(z)| \colon z\in \mathbb{T}^n\left(\mathbf{0},R+\frac{\boldsymbol{\beta}}{\mathbf{L}(Re^{i\Theta})}\right)\right\} -\nonumber 
	\\ -\ln \max\left\{|F(z)|: z\in \mathbb{T}^n\left(\mathbf{0},R+\frac{\mathbf{1}}{\mathbf{L}(Re^{i\Theta})}\right)\right\}=\nonumber\\
	=\sum_{j=1}^n \ln\max\left\{|F(z)| \colon z\in \mathbb{T}^n\left(\mathbf{0},R+\frac{\mathbf{1}+\sum_{k=j}^n (\frac{\beta}{c\sqrt{n}}-1)\mathbf{1}_k}{\mathbf{L}(Re^{i\Theta})}\right)\right\} -\nonumber\\ - 
	\ln\max\left\{|F(z)| \colon z\in \mathbb{T}^n\left(\mathbf{0},R+\frac{\mathbf{1}+\sum_{k=j+1}^n (\frac{\beta}{c\sqrt{n}}-1) \mathbf{1}_k}{\mathbf{L}(Re^{i\Theta})}\right)\right\}=\nonumber 
	\\ 
	= \sum_{j=1}^n \int^{r_j+\beta/(c\sqrt{n}l_j(Re^{i\Theta}))}_{r_j+1/l_j(Re^{i\Theta})} \frac{1}{t} 
	A_j\left(r_1+\frac{1}{l_1(Re^{i\Theta})},\ldots, r_{j-1}+\frac{1}{l_{j-1}(Re^{i\Theta})},t, \right. \nonumber\\  \left.
	r_{j+1}+\frac{\beta}{c\sqrt{n} l_{j+1}(Re^{i\Theta})},\ldots, 
	r_n+\frac{\beta}{c\sqrt{n}l_n(Re^{i\Theta})}\right)dt  
	\!\geq\! 
	\sum_{j=1}^n \ln\left(1+\frac{\frac{\beta}{c\sqrt{n}}-1}{r_jl_j(Re^{i\Theta})\!+\!1}\right)  \times \nonumber \\ \times 
	A_j\!\left(r_1\!+\!\frac{1}{l_1(Re^{i\Theta})},\ldots, r_{j-1}\!+\!\frac{1}{l_{j-1}(Re^{i\Theta})},r_j, 
	r_{j+1}+\frac{\beta}{c\sqrt{n}l_{j+1}(Re^{i\Theta})},\ldots, r_n+\frac{\beta}{c\sqrt{n}l_n(Re^{i\Theta})}\right) \label{estimaconvex}
	\end{gather}
	By Lemma \ref{qngrowth} the function  $r_jl_j(Re^{i\Theta})\to +\infty$ as $|R|\to 1-0.$ Hence, for $j\in\{1,\ldots, n\}$ and $r_i\geq r_i^0$  
	$$\ln\left(1+\frac{\frac{\beta}{c\sqrt{n}}-1}{r_jl_j(Re^{i\Theta})+1}\right) \sim 
\frac{\frac{\beta}{c\sqrt{n}}-1}{r_jl_j(Re^{i\Theta})+1} \geq 
\frac{\frac{\beta}{c\sqrt{n}}-1}{2r_jl_j(Re^{i\Theta})}.$$ 
	Thus, for every $j\in\{1,\ldots, n\}$ inequality \eqref{estimaconvex}  implies that 
	\begin{gather*}
	A_j\left(r_1+\frac{1}{l_1(Re^{i\Theta})},\ldots,r_{j-1}+\frac{1}{l_{j-1}(Re^{i\Theta})},r_j,r_{j+1}+\frac{\beta}{c\sqrt{n}l_{i+1}(Re^{i\Theta})},
	\ldots, \right. \\ \left. r_n+\frac{\beta}{c\sqrt{n}l_n(Re^{i\Theta})} \right) \leq \frac{2\ln p_1}{\frac{\beta}{c\sqrt{n}}-1}\ r_j l_j(Re^{i\Theta}).
	\end{gather*}
	Let $R^0=(r_1^0,\ldots,r_n^0),$ where every $r_j^0$ is above chosen.
	Applying \eqref{presentconvex} $n$-th times consequently  we obtain 
	\begin{gather*}
	\ln \max \{|F(z)|\colon z\in \mathbb{T}^n(\mathbf{0},R)\} = \\ =  \ln \max \{|F(z)|\colon z\in \mathbb{T}^n(\mathbf{0},R+(r_1^0-r_1)\mathbf{1}_1)\}  + 
	\int_{r_1^0}^{r_1} \frac{A_1(t,r_2,\ldots,r_n)}{t}dt= \\ =
	\ln \max \{|F(z)|\colon z\in \mathbb{T}^n(\mathbf{0},R+(r_1^0-r_1)\mathbf{1}_1+(r_2^0-r_2)\mathbf{1}_2)\}  + \\ +
	\int_{r_1^0}^{r_1} \frac{A_1(t,r_2,\ldots,r_n)}{t}dt+  \int_{r_2^0}^{r_2} \frac{A_2(r_1^0,t,r_3\ldots,r_n)}{t}dt= \\ 
	=  \ln \max \{|F(z)|\colon z\in \mathbb{T}^n(\mathbf{0},R^0)\} 
	+\sum_{j=1}^n \int_{r_j^0}^{r_j} \frac{A_j(r_1^0,\ldots, r_{j-1}^0,t,r_{j+1},\ldots,r_n)}{t} dt \leq \\       
	\leq \ln \max \{|F(z)|\colon z\in \mathbb{T}^n(\mathbf{0},R^0)\} + \\        
	+   \frac{2\ln p_1}{\frac{\beta}{c\sqrt{n}}-1}\sum_{j=1}^n \int_{r_j^0}^{r_j} l_j(r_1^0e^{i\theta_1},\ldots, r_{j-1}^0e^{i\theta_{j-1}},te^{i\theta_j},r_{j+1}e^{i\theta_{j+1}},\ldots,r_n  e^{i\theta_n}) dt \leq 
	\\
	\leq \ln \max \{|F(z)|\colon z\in \mathbb{T}^n(\mathbf{0},R^0)\} + \\        
	+   \frac{2\ln p_1}{\frac{\beta}{c\sqrt{n}}-1} \sum_{j=1}^n \int_{0}^{r_j} l_j(r_1^0e^{i\theta_1},\ldots, r_{j-1}^0e^{i\theta_{j-1}},te^{i\theta_j},r_{j+1}e^{i\theta_{j+1}},\ldots,r_n  e^{i\theta_n}) dt \leq 
	\\ \leq 
	(1+o(1)) \frac{2\ln p_1}{\frac{\beta}{c\sqrt{n}}-1} \sum_{j=1}^n \int_{0}^{r_j} l_j(r_1^0e^{i\theta_1},\ldots, r_{j-1}^0e^{i\theta_{j-1}},te^{i\theta_j},r_{j+1}e^{i\theta_{j+1}},\ldots,r_n  e^{i\theta_n}) dt.
	\end{gather*}
	The function $\ln \max \{|F(z)|\colon z\in \mathbb{T}^n(\mathbf{0},R)\}$ is independent of $\Theta.$  Thus,  the following estimate holds
	\begin{gather*}  
	\ln \max \{|F(z)|\colon z\in \mathbb{T}^n(\mathbf{0},R)\} = \\ \!=\!O\left(\!\min_{\Theta\in[0,2\pi]^n}\!\sum_{j=1}^n \!\int_{0}^{r_j} \! l_j(r_1^0e^{i\theta_1},\ldots, r_{j-1}^0e^{i\theta_{j-1}},te^{i\theta_j},r_{j+1}e^{i\theta_{j+1}},\ldots,r_n  e^{i\theta_n}) dt\!\right),\!
	\end{gather*}
	as  $|R| \to 1-0.$
	It is obviously that similar equality can be proved for arbitrary permutation $\sigma_n$ of the set $\{1,2,\ldots,n\}.$ 
	Thus, estimate \eqref{growthobig} holds.
	Theorem \ref{thgrowthobig} is proved.
\end{proof}

\begin{corollary}
	If 	$\mathbf{L}\in Q(\mathbb{B}^n)\cap K(\mathbb{B}^n),$ $\min\limits_{\Theta\in[0,2\pi]^n}l_j(Re^{i\Theta})$ is non-decreasing in each variable $r_k,$ $k\in\{1,\ldots,n\},$ 
	analytic in $\mathbb{B}^n$ function $F$ has bounded $\mathbf{L}$-index in joint variables then 
	\begin{equation*}
	\ln \max \{|F(z)|\colon z\in \mathbb{T}^n(\mathbf{0},R)\} = O\left(\min_{\Theta\in[0,2\pi]^n} \sum_{j=1}^n \int_{0}^{r_j} l_j(R^{(j)}e^{i\Theta}) dt\right)
	\end{equation*}
	as $|R|\to 1-0,$ 
	where $R^{(j)}=(r_1,\ldots,r_{j-1},t,r_{j+1},\ldots,r_n).$
\end{corollary}

Note that Theorem \ref{thgrowthobig} is new even for $n=1$ (see Theorem 3.3 in \cite{sher}) because we replace the condition $l=l(|z|)$ by the condition $l\in K (\mathbb{D}),$ i.e. 
there exists $c>0$  such that for
every $r\in(0,1)$ 
$\max\limits_{\theta_1,\theta_2\in[0,2\pi]} \frac{l(re^{i\theta_2})}{l(re^{i\theta_1})}\leq c.$ 
Particularly, the following proposition is valid. 
\begin{corollary} \label{growthcor2joint}
	If 	$l\in Q\cap K$ and an analytic in $\mathbb{D}$ function $f$ has bounded $l$-index then 
	\begin{equation*}
	\ln \max \{|f(z)|\colon |z|=r\} = O\left(\min_{\theta\in[0,2\pi]}  \int_{0}^{r} l(te^{i\Theta}) dt\right)
	\ \text{ as } r\to 1-0.
	\end{equation*}
\end{corollary}

W. K. Hayman, A. D. Kuzyk, M M. Sheremeta, V. O.  Kushnir and T. O. Banakh \cite{Hayman,vidlindex,bankush} improved an estimate \eqref{growthobig} 
by other conditions on the function $l$ for a case $n=1.$ M. T. Bordulyak and M. M. Sheremeta \cite{bagzmin} deduced similar results for entire functions of bounded $\mathbf{L}$-index in joint variables, if $l_j=l_j(|z_j|),$ $j\in\{1,\ldots,n\}.$
Using their method   we will generalize the estimate for $l_j:\mathbb{B}^n\to \mathbb{R}_+.$   

Let us to denote $a^+=\max\{a,0\},$  $u_j(t)=u_j(t,R,\Theta)=l_j(\frac{tR}{r^*}e^{i\Theta}),$   where $a\in\mathbb{R},$ $t\in\mathbb{R}_+,$ $j\in\{1,\ldots,n\},$ 
$r^*=\max_{1\leq j\leq n}r_j \neq 0$ and $\frac{t}{r^*}|R|<1.$
\begin{theorem} 
	\label{bordte22}
	Let $\mathbf{L}(Re^{i\Theta})$  be a positive continuously differentiable function in each variable $r_k,$ 
	 $k\in\{1,\ldots,n\},$ $|R|<1,$ $\Theta\in[0,2\pi]^n.$  If the function $\mathbf{L}$ satisfies \eqref{Lbeta-ball} and an analytic in $\mathbb{B}^n$ function $F$ has bounded $\mathbf{L}$-index $N=N(F,\mathbf{L})$ in joint variables then for every 	$\Theta\in[0,2\pi]^n$ and for every $R\in\mathbb{R}^n_+,$ $|R|<1,$   and $S\in\mathbb{Z}^n_+$
	\begin{gather}
	\ln\max\left\{\frac{|F^{(S)}(Re^{i\Theta})|}{S!\mathbf{L}^S(Re^{i\Theta})}:\ \|S\|\leq N \right\}\leq 
	\ln\max\left\{\frac{|F^{(S)}(\mathbf{0})|}{S!\mathbf{L}^S(\mathbf{0})}:\ \|S\|\leq N \right\}+ \nonumber \\ \!+\!
	\int_0^{r^*}\!\left(\!\max_{\|S\|\leq N}\!\left\{\!
	\sum_{j=1}^n \frac{r_j}{r^*}(k_j+1)l_j\!\left(\frac{\tau}{r^*} Re^{i\Theta}\!\right) 
	\!\right\} \!+\! \max_{\|S\|\leq N}\left\{ \sum_{j=1}^n \frac{k_j(-u_j'(\tau))^+}{l_j\left(\frac{\tau}{r^*} Re^{i\Theta}\right)}  \right\}\right)d\tau.\!
	\label{conclusiongrowth}
	\end{gather}   
	
	If, in addition,  there exists $C>0$ such that the function $\mathbf{L}$ satisfies inequalities
	\begin{gather} \label{suplinf}
	\sup\limits_{|R|<1} \max\limits_{t\in[0,r^*]}\max\limits_{\Theta\in[0,2\pi]^n} \max\limits_{1\leq j\leq n}\frac{(-(u_j(t,R,\Theta))'_t)^+}{\frac{r_j}{r^*} l_j^2(\frac{t}{r^*}Re^{i\Theta})} \leq C,
	\end{gather}
	then 
	\begin{equation}
	\label{nonzerogrowth}
	\varlimsup_{|R|\to 1-0}   \frac{\ln \max\{|F(z)\colon \ z\in 
\mathbb{T}^n(\mathbf{0},R)\}}{\max\limits_{\Theta\in[0,2\pi]^n} \int_0^{r^*}  \sum_{j=1}^n 
\frac{r_j}{r^*}l_j\left(\frac{\tau}{r^*}R e^{i\Theta}\right) d\tau}\leq (C+1) N+1.
	\end{equation}
	
	And if $r^*(-(u_j(t,R,\Theta))'_t)^+/(r_j l_j^2(\frac{t}{r^*}Re^{i\Theta}))\to 0$ 	uniformly for all $\Theta\in[0,2\pi]^n,$ $j\in\{1,\ldots,n\},$ $t\in[0,r^*]$ 
	 as $|R|\to 1-0$ 
	then 
	\begin{equation}
	\label{zerogrowth}
	\varlimsup_{|R|\to 1-0}   \frac{\ln \max\{|F(z)\colon \ z\in 
\mathbb{T}^n(\mathbf{0},R)\}}{\max\limits_{\Theta\in[0,2\pi]^n} \int_0^{r^*}  \sum_{j=1}^n 
\frac{r_j}{r^*}l_j\left(\frac{\tau}{r^*}R e^{i\Theta}\right) d\tau}\leq N+1.
	\end{equation}
\end{theorem}
\begin{proof}
	Let $R\in \mathbb{R}\setminus \{\square\},$ $\Theta\in[0,2\pi]^n.$ Denote $\alpha_j=\frac{r_j}{r^*},$ $j\in\{1,\ldots,n\}$ and $A=(\alpha_1,\ldots,\alpha_n).$  
	We consider a function 
	\begin{equation} \label{defingt}
	g(t)=\max\left\{\frac{|F^{(S)}(Ate^{i\Theta})|}{S!\mathbf{L}^S(Ate^{i\Theta})}:\ \|S\|\leq N \right\},
	\end{equation}
	where $At=(\alpha_1t,\ldots,\alpha_nt),$ $Ate^{i\Theta}= (\alpha_1te^{i\theta_1},\ldots,\alpha_nte^{i\theta_n}).$  
	
	Since the function $\frac{|F^{(S)}(Ate^{i\Theta})|}{K!\mathbf{L}^{K}(Ate^{i\Theta})}$ is  continuously differentiable by real $t\in[0,+\infty),$
	outside the zero set of function $|F^{(S)}(Ate^{i\Theta})|,$
	the function $g(t)$ is a continuously differentiable function on $[0,\frac{r^*}{|R|}),$ except, perhaps, for a countable set of points.
	
	Therefore, using    the inequality
	$\frac{d}{dr} |g(r)|\leq |g'(r)|$  which holds except for the points $r=t$ such that $g(t)=0,$  we deduce 
	\begin{gather}
	\frac{d}{dt} \left(\frac{|F^{(S)}(Ate^{i\Theta})|}{S!\mathbf{L}^S(Ate^{i\Theta})}\right)= \frac{1}{S!\mathbf{L}^S(Ate^{i\Theta})} 
	\frac{d}{dt}|F^{(S)}(Ate^{i\Theta})|+\nonumber \\ +|F^{(S)}(Ate^{i\Theta})| \frac{d}{dt}\frac{1}{S!\mathbf{L}^S(Ate^{i\Theta})} 
	\leq  \frac{1}{S!\mathbf{L}^S(Ate^{i\Theta})} \left|\sum_{j=1}^n F^{(S+\mathbf{1}_j)}(Ate^{i\Theta}) \alpha_j e^{i\theta_j}\right|- \nonumber \\ \!-\!
	\frac{|F^{(S)}(Ate^{i\Theta})|}{S!\mathbf{L}^S(Ate^{i\Theta})} 
	\sum_{j=1}^n  \frac{k_j u'_j(t)}{l_j(Ate^{i\Theta})} 
	\leq \sum_{j=1}^n  \frac{|F^{(S+\mathbf{1}_j)}(Ate^{i\Theta})|}{(S\!+\!\mathbf{1}_{j})!\mathbf{L}^{S\!+\!\mathbf{1}_j}(Ate^{i\Theta})} 
	\alpha_j(k_j\!+\!1)l_j(Ate^{i\Theta})\!+\! \nonumber \\ +
	\frac{|F^{(S)}(Ate^{i\Theta})|}{S!\mathbf{L}^S(Ate^{i\Theta})} \sum_{j=1}^n  \frac{k_j (-u'_j(t))^+}{l_j(Ate^{i\Theta})} \label{derivindex}
	\end{gather}
	For absolutely continuous functions
	$h_1,$ $h_2, $ $\ldots,$ $h_k$ and $h(x):=\max\{h_j(z): 1\leq j\leq k\},$
	\  $h'(x)\leq \max\{h'_j(x): 1\leq j\leq k \},$ $x\in[a,b]$  (see \cite[Lemma~4.1, p.~81]{sher}). The function $g$ is absolutely continuous, therefore, from \eqref{derivindex} it follows that
	\begin{gather*}
	g'(t) \leq \max \left\{\frac{d}{dt} \left(\frac{|F^{(S)}(Ate^{i\Theta})|}{S!\mathbf{L}^S(Ate^{i\Theta})}\right)\colon 
	\|S\|\leq N \right\} \leq \\ \leq 
	\max_{ \|S\|\leq N} \left\{ 
	\sum_{j=1}^n  \frac{\alpha_j(s_j+1)l_j(Ate^{i\Theta})|F^{(S+\mathbf{1}_j)}(Ate^{i\Theta})|}{(K+\mathbf{1}_{j})!\mathbf{L}^{K+\mathbf{1}_j}(Ate^{i\Theta})} 
	+ \right. \\ \left.+
	\frac{|F^{(S)}(Ate^{i\Theta})|}{S!\mathbf{L}^S(Ate^{i\Theta})} \sum_{j=1}^n  \frac{s_j (-u'_j(t))^+}{l_j(Ate^{i\Theta})}    \right\}\leq \\
	\leq g(t) \left( \max_{ \|S\|\leq N}\left\{ 
	\sum_{j=1}^n \alpha_j(s_j+1)l_j(Ate^{i\Theta})\right\}+ 
	\max_{ \|S\|\leq N}\left\{ \sum_{j=1}^n \frac{s_j(-u_j'(t))^+}{l_j(Ate^{i\Theta})}\right\}\right)=\\ =
	g(t)(\beta(t)+\gamma(t)),
	\end{gather*}
	where 
	$$ \beta(t)=\max_{\|S\|\leq N}\left\{
	\sum_{j=1}^n \alpha_j(s_j+1)l_j(Ate^{i\Theta}) \right\}, 
	\gamma(t)=\max_{\|S\|\leq N}\left\{ \sum_{j=1}^n \frac{s_j(-u_j'(t))^+}{l_j(Ate^{i\Theta})} \right\}.$$
	Thus, 
	$\frac{d}{dt} \ln g(t)\leq \beta(t)+\gamma(t)$
	and 
	\begin{equation} \label{deqr1}
	g(t)\leq g(0)\exp \int_0^t (\beta(\tau)+\gamma(\tau))d\tau, 
	\end{equation} because $g(0)\neq 0.$
	But $r^*A=R.$ Substituting $t=r^*$ in \eqref{deqr1} and taking  into account \eqref{defingt},
	we deduce 
	\begin{gather*}
	\ln\max\left\{\frac{|F^{(S)}(Re^{i\Theta})|}{S!\mathbf{L}^S(Re^{i\Theta})}:\ \|S\|\leq N \right\}\leq 
	\ln\max\left\{\frac{|F^{(S)}(\mathbf{0})|}{S!\mathbf{L}^S(\mathbf{0})}:\ \|S\|\leq N \right\}+ \\ \!+\!
	\int_0^{r^*} \left(\max_{\|S\|\leq N}\left\{
	\sum_{j=1}^n \alpha_j(s_j+1)l_j(A\tau e^{i\Theta}) 
	\right\} \!+\! \max_{\|S\|\leq N}\left\{ \sum_{j=1}^n \frac{s_j(-u_j'(\tau))^+}{l_j(A\tau e^{i\Theta})} \right\}\right)d\tau,\!
	\end{gather*}   
	i.e.  \eqref{conclusiongrowth} is proved. 
	It is easy to see that if $\mathbf{L}$ satisfies \eqref{Lbeta-ball} then
	\begin{equation}
	 \label{intliminf}
\max\limits_{\Theta\in[0,2\pi]^n} \int_0^{r^*}  \sum_{j=1}^n \frac{r_j}{r^*}l_j\left(\frac{\tau}{r^*}R e^{i\Theta}\right) d\tau\to +\infty 
\text{ as } |R|\to 1-0.
	\end{equation}

	Denote $\widetilde{\beta}(t)= \sum_{j=1}^n \alpha_jl_j(Ate^{i\Theta}).$   
	If, in addition, \eqref{suplinf}  holds  
	then  for some $S^*,$ $\|S^*\|\leq N$ and $\widetilde{S},$ $\|\widetilde{S}\|\leq N,$ 
	\begin{gather*}
	\frac{\gamma(t)}{\widetilde{\beta}(t)}= \frac{\sum_{j=1}^n \frac{s^*_j(-u_j'(t))^+}{l_j(Ate^{i\Theta})}}{\sum_{j=1}^n \alpha_jl_j(Ate^{i\Theta})}\leq 
	\sum_{j=1}^n s^*_j\frac{(-u_j'(t))^+}{\alpha_jl^2_j(Ate^{i\Theta})}\leq 
	\sum_{j=1}^n  s^*_j \cdot C \leq NC \text{ and } \\
	\frac{\beta(t)}{\widetilde{\beta}(t)}=\frac{ \sum_{j=1}^n \alpha_j(\tilde{s}_j+1)l_j(Ate^{i\Theta})}{\sum_{j=1}^n \alpha_jl_j(Ate^{i\Theta})} =
	1+\frac{ \sum_{j=1}^n \alpha_j\tilde{s}_jl_j(Ate^{i\Theta})}{\sum_{j=1}^n \alpha_jl_j(Ate^{i\Theta})} \leq \\ 
	\leq 1+
	\sum_{j=1}^n \tilde{s}_j\leq 1+ N.
	\end{gather*}
	
	But 
	$|F(Ate^{i\Theta})|\leq g(t)\leq g(0) \exp \int_0^t (\beta(\tau)+\gamma(\tau))d\tau$ 
	and $r^*A=R.$ Putting $t=r^*$ and taking into account \eqref{intliminf} 
 we obtain 
	\begin{gather*}
	\ln \max\{|F(z)\colon \ z\in \mathbb{T}^n(\mathbf{0},R)\}=\ln \max_{\Theta\in[0,2\pi]^n}|F(Re^{i\Theta})| \leq \ln \max_{\Theta\in[0,2\pi]^n} g(r^*) \leq 
	\\ \leq 
	\ln g(0)+
	\max_{\Theta\in[0,2\pi]^n} \int_0^{r^*} (\beta(\tau)+\gamma(\tau))d\tau   \leq\\ 
	\leq \ln g(0)+(NC+N+1)    \max_{\Theta\in[0,2\pi]^n} \int_0^{r^*} \widetilde{\beta}(\tau) d\tau = \\ 
	= \ln g(0)+(NC+N+1) \max_{\Theta\in[0,2\pi]^n}  \int_0^{r^*}  \sum_{j=1}^n \alpha_jl_j(A\tau e^{i\Theta}) d\tau =\\ 
	=\ln g(0)+(NC+N+1) \max_{\Theta\in[0,2\pi]^n} \int_0^{r^*}  \sum_{j=1}^n \frac{r_j}{r^*}l_j\left(\frac{\tau}{r^*}R e^{i\Theta}\right) d\tau.
	\end{gather*}
	Thus, we conclude that \eqref{nonzerogrowth} holds.
	Estimate \eqref{zerogrowth} can be deduced by analogy. 
	Theorem \ref{bordte22} is proved.
\end{proof}
We will write $u(r,\theta)=l(re^{i\theta}).$  Theorem \ref{bordte22} implies the following proposition for $n=1.$
\begin{corollary} \label{growthcor3joint}
	Let $l(re^{i\theta})$  be a positive continuously differentiable function in variable $r\in[0,1)$ for every $\theta\in[0,2\pi].$  If an analytic in $\mathbb{D}$ function $f$ has bounded $l$-index $N=N(f,l)$ 
	and  there exists $C>0$ such that 
	$\varlimsup\limits_{r\to 1-0}\max\limits_{\theta\in[0,2\pi]} \frac{(-u'_r(r,\theta))^+}{l^2(re^{i\Theta})} = C$   then 
	\begin{equation*}
	\varlimsup_{r\to 1-0}   \frac{\ln \max\{|f(z)\colon  |z|=r\}}{\max\limits_{\theta\in[0,2\pi]} \int_0^{r}   
		l\left(\tau e^{i\theta}\right) d\tau}\leq (C+1)N+1.
	\end{equation*}
\end{corollary}

\begin{remark}
	Our result is sharper than known result of Sheremeta which is obtained in a case $n=1,$ $C\neq 0$ 
	and $l=l(|z|).$
	Indeed, corresponding theorem \cite[p.~83]{sher} claims that 
	\begin{equation*}
	\varlimsup_{r\to 1-0}   \frac{\ln \max\{|f(z)\colon  |z|=r\}}{\int_0^{r}   
		l(\tau) d\tau}\leq (C+1)(N+1).
	\end{equation*}	
	Obviously, that $NC+N+1 < (C+1)(N+1)$ for $C\neq 0$ and $N\neq 0.$
\end{remark}

Estimate \eqref{zerogrowth} is sharp. It is easy to check for these functions $F(z_1,z_2)=\exp(z_1z_2),$  $l_1(z_1,z_2)=|z_2|+1,$  $l_2(z_1,z_2)=|z_1|+1.$ 
Then $N(F,\mathbf{L})=0$ and $\ln \max\{|F(z)|\colon z\in T^2(\mathbf{0},R)\}=r_1r_2.$

  \section{Bounded $\mathbf{L}$-index in joint variables in a bounded domain}
By $\overline{G}$ we denote the closure of a domain $G.$
\begin{theorem} \label{boundedomjoin-ball}
	Let $F(z)$ be an analytic in $\mathbb{B}^n$ function, $G$ be a bounded domain in $\mathbb{B}^n,$ 
 $d=\inf_{z\in\overline{G}} (1-|z|)>0$ and $\beta>\sqrt{n}.$
	If for every $j\in\{1,\ldots,n\}$ $l_j\colon \mathbb{B}^n\to\mathbb{R}_{+}$ is a continuous function  
	satisfying $l_j(z)\ge \frac{\beta}{d}$ for all $z\in \mathbb{B}^n$  then  there exists   $m\in\mathbb{Z}_{+}$ 
such that for all 
	$z\in \overline{G}$ and $J=(j_{1},j_{2},\ldots , j_{n})\in\mathbb{Z}^{n}_{+}$
	\begin{equation} \label{obmindex}
	\frac{|F^{(J)}(z)|}{J!\mathbf{L}^{J}(z)}\leq\max
	\left\{\frac{|F^{(K)}(z)|}{K! \mathbf{L}^{K}(z)}:\
	K\in\mathbb{Z}^{n}_{+},\ \|K\|\leq m\right\},
	\end{equation}
	where $\mathbf{L}(z)=(l_1(z),\ldots,l_n(z)).$
\end{theorem}
\begin{proof}
	If $F(z)\equiv 0$ then \eqref{obmindex} is obvious.
	Let $F(z)\not \equiv 0.$ 
	For every fixed $z^0\in \overline{G}$  \ 
	$\frac{|F^{(J)}(z^0)|}{J!\mathbf{L}^{J}(z^0)}$ is the modulus of a coefficient of power series expansion of function $F(z),$ $z\in \mathbb{T}^n(z^0,\frac{R_0}{\mathbf{L}(z^0)}),$ where $|R_0|=\sqrt{n}.$
	Since $F(z)$ is analytic, for every $z^0\in \overline{G}$ \ $\frac{|F^{(J)}(z^0)|}{J!\mathbf{L}^{J}(z^0)}\to 0$ as $\|J\|\to \infty,$ 
	i. e. there exists $m_0=m(z^0),$ for which inequality \eqref{obmindex} holds.
	
	Assume on the contrary, that the set of $m_0$ is not uniformly bounded in $z^0:$  
	$\sup\limits_{z^0\in \overline{G}} m_0=+\infty.$ 
	Hence, for every $m\in\mathbb{Z}_+$ there exist $z_m\in \overline{G}$ and $J^{m}\in\mathbb{Z}^n_+$
	\begin{equation} \label{suprot}
	\frac{|F^{(J^m)}(z^m)|}{J^m!\mathbf{L}^{J^m}(z_m)}> \max
	\left\{\frac{|F^{(K)}(z^m)|}{K! \mathbf{L}^{K}(z^m)}:\
	K\in\mathbb{Z}^{n}_{+},\ \|K\|\leq m\right\}.
	\end{equation}
	Since $z^m\in \overline{G},$ there exists subsequence $z'^m\to z'\in \overline{G}$ as $m\to +\infty.$ 
	By Cauchy's integral formula for any $J\in\mathbb{Z}^n_+$  $$\frac{F^{(J)}(z^0)}{J!}=\frac{1}{(2\pi i)^n}\int_{z\in \mathbb{T}^n(z^0,R)} \frac{F(z)}{(z-z^0)^{J+\mathbf{1}}} dz.$$ 
	We rewrite \eqref{suprot} in the form  
	\begin{gather}
	\max\left\{\frac{|F^{(K)}(z^m)|}{K! \mathbf{L}^{K}(z^m)}:\
	K\in\mathbb{Z}^{n}_{+},\ \|K\|\leq m\right\} \leq \nonumber \\ \leq\frac{1}{(2\pi)^n \mathbf{L}^{J^m}(z^m)} \int_{z\in \mathbb{T}^n(z^0,\frac{R}{\mathbf{L}(z^m)})} \frac{|F(z)|}{|z-z^m|^{J^m+\mathbf{1}}} |dz| \leq
	\frac{1}{R^{J^m}} \max\{|F(z)|\colon  z\in G_R\}, \label{suprot2}
	\end{gather}
	where $G_R=\bigcup_{z^*\in \overline{G}} {D}^n[z^*,\frac{R}{\mathbf{L}(z^*)}],$ $|R|\leq \beta.$
	We choose $R$ such that $r_j>1$ i. e. $|R|>\sqrt{n}.$ Taking the limit in \eqref{suprot2} as $m\to \infty$ we deduce
	$$
	\forall K\in\mathbb{Z}^n_+ \ \ \  \frac{|F^{(K)}(z')|}{K!\mathbf{L}^{K}(z')} \leq 
	\lim_{m\to \infty} \frac{1}{R^{J^m}} \max\{|F(z)|\colon  z\in G_R\} = 0.
	$$  as $m\to +\infty.$
	Thus, all partial derivatives of the function $F$ at point $z'$ equals $0.$ By uniqueness theorem $F(z)\equiv 0.$ 
	It is impossible.
\end{proof}

\begin{remark}
	A similar proposition for analytic in $\mathbb{B}^n$ functions of bounded $L$-index in a direction  $\mathbf{b} \in \mathbb{C}^n\setminus \mathbf{0}$ is valid by additional assumption 
	$\forall z\in\overline{G}$ $F(z+t\mathbf{b})\not \equiv 0,$ where $t\in\mathbb{C}$ (see \cite{bandurasum,monograph}).
\end{remark}

\section{Exhaustion of $\mathbb{B}^n$ by balls}
Addition, scalar multiplication, and conjugation are defined on $\mathbb{C}^n$ componentwise. 
For $z\in\mathbb{C}^n$ and $w\in\mathbb{C}^n$ we define 
$$
\langle z,w\rangle =z_1\overline{w}_1+\cdots+z_n\overline{w}_n,
$$
where $w_k$ is the complex conjugate of $w_k.$

Denote 
$$\ell(z)=\min_{1\le j\le n} l_j(z), \ \mathcal{L}(z)=\max_{1\le j\le n} l_j(z).$$
Obviously, that $\ell(z) \leq  \mathcal{L}(z).$


The open ball $\{z\in\mathbb{C}^n: \ |z-z^0|<r\}$ is denoted by $\mathbb{B}^n(z^0,r),$ 
its boundary is a sphere $\mathbb{S}^n(z^0,r)=\{z\in\mathbb{C}^n: \ |z-z^0|=r\},$ 
the closed ball  $\{z\in\mathbb{C}^n: \ |z-z^0|\leq r\}$ is denoted by $\mathbb{B}^n[z^0,r],$ 
$\mathbb{B}^n=\mathbb{B}^n(\mathbf{0},1),$
$\mathbb{D}=\mathbb{B}^1=\{z\in\mathbb{C}: \ |z|<1\}.$

By $Q'(\mathbb{B}^n)$ we denote the class of functions $\mathbf{L}$, which satisfy the condition
\begin{equation} \label{clasqballnex}
(\forall r\in[0,\beta],\ j\in\{1,\ldots,n\})\colon \ 0<\lambda_{1,j}(r)\leq \lambda_{2,j}(r)<\infty,
\end{equation}
where
\begin{gather} \label{lam1ex}
\lambda_{1,j}(r)=\inf\limits_{ z^0\in \mathbb{B}^n} \inf \left \{
\frac{l_j(z)}{l_j(z^0)}: z\in \mathbb{B}^n\left[z^0, {r}/{\ell(z^0)}\right]\right\},\\
\lambda_{2,j}(r)=\sup\limits_{z^0\in \mathbb{B}^n} \sup \left \{
\frac{l_j(z)}{l_j(z^0)}: z\in \mathbb{B}^n\left[z^0, {r}/{\ell(z^0)}\right]\right\}.\label{lam2ex}\\
\Lambda_1(r)=(\lambda_{1,1}(r),\ldots,\lambda_{1,n}(r)), \ \Lambda_2(r)=(\lambda_{2,1}(r),\ldots,\lambda_{2,n}(r)).
\end{gather}
These denotations of $\lambda_{1,j}(r),$ $\lambda_{2,j}(r),$ $\Lambda_1(r),$ $\Lambda_2(r)$ are only valid in this section.
In other sections their meanings are defined in \eqref{lam1}-\eqref{lam2}.

 The following theorem is basic in theory of functions of bounded index. It was necessary to prove more efficient criteria of index boundedness 
which describe a behavior of maximum modulus on a disc or a behavior of logarithmic derivative (see \cite{kusher,sher,BandSk,monograph}). 
\begin{theorem} \label{petr1ex-ball}
	Let $\mathbf{L} \in Q'(\mathbb{B}^n)$. 
	In order that an analytic in $\mathbb{B}^n$ function $F$ be of bounded $\mathbf{L}$-index in joint variables it is necessary 
	that for each
	$r\in(0,\beta]$ there exist $n_0\in \mathbb{Z}_+$, $p_0>0$ such that for every $z^0 \in\mathbb{B}^n$ there exists $K^0\in \mathbb{Z}_+^n$, $\|K^0\|\leq n_0$, such that 
	\begin{gather}
	\max\left\{ \frac{|F^{(K)}(z)|}{K!\mathbf{L}^K(z)}\colon  \|K\| \leq n_0, \ z\in \mathbb{B}^n\left[z^0, {r}/{\mathcal{L}(z^0)}\right] \right\} \leq 
	p_0 \frac{|F^{(K^0)}(z^0)|}{K^0!\mathbf{L}^{K^0} (z^0)}
	\label{net1ex}
	\end{gather}
	and it is sufficient  for each
	$r\in(0,\beta]$ there exist $n_0\in \mathbb{Z}_+$, $p_0>0$ such that for every $z^0 \in\mathbb{B}^n$ there exists $K^0\in \mathbb{Z}_+^n$, $\|K^0\|\leq n_0$, such that 
	\begin{gather}
	\max\left\{ \frac{|F^{(K)}(z)|}{K!\mathbf{L}^K(z)}\colon  \|K\| \leq n_0, \ z\in \mathbb{B}^n\left[z^0, {r}/{\ell(z^0)|}\right] \right\} \leq 
	p_0 \frac{|F^{(K^0)}(z^0)|}{K^0!\mathbf{L}^{K^0} (z^0)}.
	\label{net1ex-ball}
	\end{gather}
\end{theorem}

\begin{proof}
	Let $F$ be of bounded $\mathbf{L}$-index in joint variables with $N=N(F,\mathbf{L},\mathbb{B}^n)<\infty.$
	For every $r\in(0,\beta]$  we put
	$$q= q(r)= [ 2(N+1) r\sqrt{n} \prod_{j=1}^{n}(\lambda_{1,j}(r))^{-N}(\lambda_{2,j}(r))^{N+1}]+1$$
	where $[x]$ is the analytic in $\mathbb{B}^n$ part of the real number $x,$ i.e. it is a floor function.
	For $p\in\{0,\ldots,q\}$ and $z^0\in\mathbb{B}^n$ we denote
	\begin{gather*}
	S_p(z^0,r)=\max\left\{\frac{|F^{(K)}(z)|}{K! \mathbf{L}^{K} (z)}: \|K\|\leq N, z\in \mathbb{B}^n \left[z^0,\frac{pr}{q\mathcal{L}(z^0)}\right] \right\},
	\\
	S^{*}_p(z^0,r)=\max\left\{\frac{|F^{(K)}(z)|}{K! \mathbf{L}^{K} (z^0)}: \|K\|\leq N, z\in \mathbb{B}^n \left[z^0,\frac{pr}{q\mathcal{L}(z^0)}\right] \right\}.
	\end{gather*}
	Using \eqref{lam1} and $\mathbb{B}^n\left[z^0,\frac{pr}{q\ell(z^0)}\right] \subset \mathbb{B}^n \left[z^0,\frac{r}{\mathcal{L}(z^0)}\right] $, we have
	\begin{gather*}
	S_p(z^0,r)=\!\max\left\{\frac{|F^{(K)}(z)|}{K! \mathbf{L}^{K} (z)}\frac{\mathbf{L}^{K} (z^0)}{\mathbf{L}^{K} (z^0)}: \|K\|\leq \!N,
	z\in\! \mathbb{B}^n\! \left[z^0,\frac{pr}{q\mathcal{L}(z^0)}\!\right]\! \right\}\! \leq\! \nonumber \\ \leq S^{*}_p(z^0,r)\max\left\{ \prod_{j=1}^n \frac{l_j^{N} (z^0)}{l_j^{N} (z)}: z\in \mathbb{B}^n \left[z^0,\frac{pr}{q \mathcal{L}(z^0)}\right] \right\} \leq S^{*}_p(z^0,r)\prod_{j=1}^{n} (\lambda_{1,j}(r))^{-N}.
	\end{gather*}
	and, using \eqref{lam2}, we obtain
	\begin{gather}
	S^{*}_p(z^0,r)= \max\left\{\frac{|F^{(K)}(z)|}{K!\mathbf{L}^{K} (z)}\!\frac{\mathbf{L}^{K} (z)}{\mathbf{L}^{K} (z^0)}: \|K\|\leq N,
	z\in \mathbb{B}^n \left[z^0,\frac{pr}{q \mathcal{L}(z^0)}\right] \right\} \leq \nonumber \\
	\! \leq\! \max\left\{\frac{|F^{(K)}(z)|}{\!K! \mathbf{L}^{K} (z)}\!(\Lambda_{2}(r))^{K}: \|K\|\leq \!N,
	z\in \mathbb{B}^n \left[z^0,\frac{pr}{q \mathcal{L}(z^0)}\right] \right\} \! \leq\! S_p(z^0,r)\prod_{j=1}^{n} (\lambda_{2,j}(r))^N.
	\label{netss2ex}
	\end{gather}
	Let $K^{(p)}$ with $\| K^{(p)}\|\leq N$ and $z^{(p)} \in \mathbb{B}^n \left[z^0,\frac{pr}{q\mathcal{L}(z^0)}\right]$ be such that
	\begin{gather}
	S_p^*(z^0,r)=\frac{|F^{(K^{(p)})}(z^{(p)})|}{K^{(p)}!\mathbf{L}^{K^{(p)}} (z^{0})}
	\label{starsex}
	\end{gather}
	Since by the maximum principle $z^{(p)}\in \mathbb{S}_n(z^0,\frac{pr}{q \mathcal{L}(z^0)}),$ we have $z^{(p)}\neq z^0.$ We choose 
	$\widetilde z^{(p)}_{j}=z_j^0+\frac{p-1}{p}(z_j^{(p)}-z_j^0),$  $j\in\{1,\ldots,n\}$
	Then 
	we have 
	\begin{gather}
	|\widetilde z^{(p)}-z^0|=\frac{p-1}{p}|z^{(p)}-z^0|=\frac{p-1}{p} \frac{pr}{q\mathcal{L}(z^0)},
	\label{zetex}\\
	|\widetilde z^{(p)}-z^{(p)}|=|z^{0}+\frac{p-1}{p}(z^{(p)}-z^0)-z^{(p)}|=
	\frac{1}{p}|z^0-z^{(p)}|= \frac{1}{p}\frac{pr}{q\ell(z^0)}=\frac{r}{q \mathcal{L}(z^0)}. \label{zetexwave}
	\end{gather}
	From \eqref{zetex} we obtain
	$\widetilde z^{(p)} \in \mathbb{B}^n \left[z^0,\frac{(p-1)r}{q \mathcal{L}(z^0)}\right]$ and 
	\begin{gather*}
	S^{*}_{p-1}(z^0,r)\geq
	\frac{|F^{(K^{(p)})}(\widetilde z^{(p)})|}{K^{(p)}!\mathbf{L}^{K^{(p)}} (z^0)}.
	\end{gather*}
	From \eqref{starsex} it follows that
	\begin{gather}
	0\leq S^{*}_p(z^0,r)-S^{*}_{p-1}(z^0,r) \leq \frac{|F^{(K^{(p)})}(z^{(p)})|-|F^{(K^{(p)})}(\widetilde z^{(p)})|}{K^{(p)}!\mathbf{L}^{K^{(p)}} (z^0)}=\nonumber \\
	= \frac{1}{K^{(p)}!\mathbf{L}^{K^{(p)}} (z^0)}\int_{0}^{1}
	\frac{d}{dt}|F^{(K^{(p)})}(\widetilde z^{(p)}+t(z^{(p)}-\widetilde z^{(p)}))|dt \leq
	\nonumber \\
	\leq \frac{1}{K^{(p)}!\mathbf{L}^{K^{(p)}}(z^0)}\int_0^1 \sum_{j=1}^n |z_j^{(p)}-\widetilde z_{j}^{(p)}| 
	\left|\frac{\partial^{\|K^{(p)}\| + 1}F}{\partial z_1^{k_1^{(p)}}  \ldots  \partial z_j^{k_j^{(p)}+1} \ldots \partial z_n^{k_n^{(p)}}} (\widetilde z^{(p)}+t(z^{(p)}-\widetilde z^{(p)}))  \right|dt= \nonumber \\ =
	\frac{1}{K^{(p)}!\mathbf{L}^{K^{(p)}}(z^0)} \sum_{j=1}^n |z_j^{(p)}-\widetilde z_{j}^{(p)}|
	\left|\frac{\partial^{\|K^{(p)}\|+1}F}{\partial z_1^{k_1^{(p)}} \ldots \partial z_j^{k_j^{(p)}+1} \ldots \partial z_n^{k_n^{(p)}}} (\widetilde z^{(p)}+t^*(z^{(p)}-\widetilde z^{(p)}))  \right|,
	\label{bigex}
	\end{gather}
	where $0\leq t^*\leq 1, \widetilde z^{(p)}+t^*(z^{(p)}-\widetilde z^{(p)}) \in \mathbb{B}^n (z^0,\frac{pr}{q \mathcal{L}(z^0)})$.
	For $z\in \mathbb{B}^n (z^0,\frac{pr}{q \mathcal{L}(z^0)})$ and $J\in\mathbb{Z}^n_+$, $\|J\| \leq N+1$ we have
	\begin{gather*}
	\frac{|F^{(J)}(z)|\mathbf{L}^{J} (z)}{J!\mathbf{L}^{J} (z^0) \mathbf{L}^{J} (z)}
	\leq (\Lambda_{2}(r))^{J} \max\left\{\frac{|F^{(K)}(z)|}{K! \mathbf{L}^{K} (z)}: \|K\|\leq N \right\} 
	\leq \prod_{j=1}^{n} (\lambda_{2,j}(r))^{N+1}(\lambda_{1,j}(r))^{-N} \times \\ \times \max\left\{\frac{|F^{(K)}(z)|}{K!\mathbf{L}^{K} (z^0)}: \|K\|\leq N \right\} 
	\leq \prod_{j=1}^{n} (\lambda_{2,j}(r))^{N+1}(\lambda_{1,j}(r))^{-N} S^{*}_p(z^0,r).
	\end{gather*}
	From \eqref{bigex} and \eqref{zetexwave} we obtain
	\begin{gather*}
	0\leq S^{*}_p(z^0,r)-S^{*}_{p-1}(z^0,r) \leq \\ \leq \prod_{j=1}^{n} (\lambda_{2,j}(r))^{N+1}(\lambda_{1,j}(r))^{-N} S^{*}_p(z^0,r)\sum_{j=1}^{n}(k_j^{(p)}+1)l_j(z^0)|z_j^{(p)}-\widetilde{z}_j^{(p)}|= \\
	\\ = \prod_{j=1}^{n} (\lambda_{2,j}(r))^{N+1}(\lambda_{1,j}(r))^{-N} 
	S^{*}_p(z^0,r)  (N+1) \sum_{j=1}^{n}   l_j(z^0)|z_j^{(p)}-\widetilde{z}_j^{(p)}|
	\leq \\
	\leq \prod_{j=1}^{n} (\lambda_{2,j}(r))^{N+1}(\lambda_{1,j}(r))^{-N}  (N+1) S^{*}_p(z^0,R) \sqrt{n} \mathcal{L}(z^0) |z^{(p)}-\widetilde{z}^{(p)}| = \\ 
	= \prod_{j=1}^{n} (\lambda_{2,j}(r))^{N+1}(\lambda_{1,j}(r))^{-N} \sqrt{n} \frac{(N+1)r}{q(r)}  S^{*}_p(z^0,R) \leq \frac{1}{2}S^{*}_p(z^0,R).
	\end{gather*}
	This inequality implies
	$S^{*}_p(z^0,r) \leq 2S^{*}_{p-1}(z^0,r),$
	and in view of inequalities \eqref{netss2ex} and \eqref{starsex} we have
	\begin{gather*}
	S_p(z^0,r) \leq 2 \prod_{j=1}^{n} (\lambda_{1,j}(r))^{-N}S^{*}_{p-1}(z^0,r) \leq
	2 \prod_{j=1}^{n} (\lambda_{1,j}(r))^{-N}(\lambda_{2,j}(r))^{N}S_{p-1}(z^0,r)
	\end{gather*}
	Therefore,
	\begin{gather}
	\max\left\{\frac{|F^{(K)}(z)|}{K!\mathbf{L}^{K} (z)}: \|K\|\leq N, z\in \mathbb{B}^n \left[z^0,\frac{pr}{q\mathcal{L}(z^0)}\right] \right\}
	= S_q(z^0,r) \leq \nonumber \\ \leq 2 \prod_{j=1}^{n} (\lambda_{1,j}(r))^{-N}(\lambda_{2,j}(r))^{N}S_{q-1}(z^0,r) \leq \ldots 
	\leq (2 \prod_{j=1}^{n} (\lambda_{1,j}(r))^{-N}(\lambda_{2,j}(r))^{N})^q S_{0}(z^0,r)=\nonumber \\
	= (2 \prod_{j=1}^{n} (\lambda_{1,j}(r))^{-N}(\lambda_{2,j}(r))^{N})^q \max\left\{\frac{|F^{(K)}(z^0)|}{K!\mathbf{L}^{K} (z^0)}: \|K\|\leq N \right\}.
	\label{lastth1ex}
	\end{gather}
	
	From \eqref{lastth1ex} we obtain inequality \eqref{net1ex} with $p_0=(2 \prod_{j=1}^{n} (\lambda_{1,j}(r))^{-N}(\lambda_{2,j}(r))^{N})^q $ and some $K^0$ with $\|K^0\|\leq N$.
	The necessity of condition \eqref{net1ex} is proved.
	
	Now we prove the sufficiency. Suppose that for every $r\in(0,\beta]$ there exist $n_0\in \mathbb{Z}_+,$ $p_0>1 $ such that for all $ z_0 \in \mathbb{B}^n $ and some
	$ K^0 \in \mathbb{Z}_+^n,$ $\|K^0\|\leq n_0,$ the inequality \eqref{net1ex-ball} holds.
	
	We write Cauchy's formula for a ball (see \cite[p.~109]{kehe-zhu} or \cite[p.~349]{rudin-ball})  as following 
	$\forall z^0\in \mathbb{B}^n$ $\forall K\in \mathbb{Z}_+^n$ $\forall S \in \mathbb{Z}_+^n$ $\forall z\in \mathbb{B}^n(z^0,r/\ell(z^0))$ 
	$$
	F^{(K+S)}(z)= \frac{(n+\|S\|-1)!}{(n-1)!} \int_{\mathbb{S}^n(z^0,r/\ell(z^0))} \frac{|\xi-z^0|(\overline{\xi-z^0})^SF^{(K)}(\xi)}{(|\xi-z^0|^2-\langle z-z^0,\xi-z^0\rangle)^{n+\|S\|}} d\sigma(\xi),
	$$
	where $d\sigma(\xi)$ is the normalized surface measure on $\mathbb{S}_n,$  so
	that $\sigma (\mathbb{S}_n(\mathbf{0},1))= 1.$ 
	Put $z=z^0:$
	\begin{equation} \label{cauchy-ball}
	F^{(K+S)}(z^0) = \frac{(n+\|S\|-1)!}{(n-1)!} \int_{\mathbb{S}^n(z^0,r/\ell(z^0))} \frac{(\overline{\xi-z^0})^SF^{(K)}(\xi)}{|\xi-z^0|^{2(n+\|S\|)-1}} d\sigma(\xi)
	\end{equation}
	Therefore, applying \eqref{net1ex-ball}, we have
	\begin{gather}
	|F^{(K+S)}(z^0)|\leq  \frac{(n+\|S\|-1)!}{(n-1)!}  \int_{\mathbb{S}^n(z^0,r/\ell(z^0))} \frac{|(\xi-z^0)^S| |F^{(K)}(\xi)|}{|\xi-z^0|^{2(n+\|S\|)-1}} d\sigma(\xi)\leq
	\nonumber \\
	\leq \left(\frac{\ell(z^0)}{r}\right)^{2(n+\|S\|)-1}\frac{(n+\|S\|-1)!}{(n-1)!} \int_{\mathbb{S}^n(z^0,r/\ell(z^0))} \frac{|(\xi-z^0)^S| |F^{(K)}(\xi)|K!\mathbf{L}^K(\xi)}{K!\mathbf{L}^K(\xi)}d\sigma(\xi)\leq \nonumber\\
	\leq p_0 \left(\frac{\ell(z^0)}{r}\right)^{2(n+\|S\|)-1}\frac{(n+\|S\|-1)!}{(n-1)!} \int_{\mathbb{S}^n(z^0,r/\ell(z^0))} 
	\frac{|(\xi-z^0)^S| |F^{(K^0)}(z^0)|K!\mathbf{L}^K(z)}{K^0!\mathbf{L}^{K^0}(z^0)}d\sigma(\xi)\leq \nonumber\\
	\leq p_0
	\left(\frac{\ell(z^0)}{r}\right)^{2(n+\|S\|)-1}\frac{(n+\|S\|-1)!}{(n-1)!}  
	\frac{ |F^{(K^0)}(z^0)|K!\prod_{j=1}^n \lambda^{n_0}_{2,j}(r)  \mathbf{L}^K(z^0)}{K^0!\mathbf{L}^{K^0}(z^0)}
	\times \nonumber\\ \times 
	\int_{\mathbb{S}^n(z^0,r/\ell(z^0))} |(\xi-z^0)^S| d\sigma(\xi)\leq\nonumber\\ 
	\leq p_0
	\left(\frac{\ell(z^0)}{r}\right)^{\|S\|}\frac{(n+\|S\|-1)!}{(n-1)!}  
	\frac{ |F^{(K^0)}(z^0)|K!\prod_{j=1}^n \lambda^{n_0}_{2,j}(r)  \mathbf{L}^K(z^0)}{K^0!\mathbf{L}^{K^0}(z^0)} \times 
	\nonumber
	\\ \times 
	\int_{\mathbb{S}^n(z^0,r/\ell(z^0))} 
	\frac{|(\xi-z^0)^S|}{(r/\ell(z^0))^{\|S\|}} d\sigma\left(\frac{\xi-z^0}{r/\ell(z^0)}\right)\leq \nonumber\\ 
	\leq p_0
	\left(\frac{\ell(z^0)}{r}\right)^{\|S\|}\frac{(n+\|S\|-1)!}{(n-1)!}  
	\frac{ |F^{(K^0)}(z^0)|K!\prod_{j=1}^n \lambda^{n_0}_{2,j}(r)  \mathbf{L}^K(z^0)}{K^0!\mathbf{L}^{K^0}(z^0)} 
	\int_{\mathbb{S}^n(\mathbf{0},1)} |\xi^{S}| d\sigma(\xi) = \nonumber\\
	= p_0
	\left(\frac{\ell(z^0)}{r}\right)^{\|S\|}\frac{(n+\|S\|-1)!}{(n-1)!} 
	\frac{ |F^{(K^0)}(z^0)|K!\prod_{j=1}^n \lambda^{n_0}_{2,j}(r)  \mathbf{L}^K(z^0)}{K^0!\mathbf{L}^{K^0}(z^0)} 
	\frac{\Gamma(n)  \prod_{j=1}^n\Gamma(s_j/2+1)}{\Gamma(n+\|S\|/2)}  \label{gamma-cauchy-ball}
	\end{gather}
	This implies 
	\begin{gather}
	\frac{|F^{(K+S)}(z^0)|}{(K+S)!\mathbf{L}^{K+S}(z^0)}\leq \nonumber \\ \leq \frac{ |F^{(K^0)}(z^0)|}{K^0!\mathbf{L}^{K^0}(z^0)} 
	p_0 	\left(\frac{\ell(z^0)}{r}\right)^{\|S\|}
	\frac{K!\prod_{j=1}^n \lambda^{n_0}_{2,j}(r) (n+\|S\|-1)!  \prod_{j=1}^n\Gamma(s_j/2+1)}{
		(K+S)! 
		\Gamma(n+\|S\|/2)  \mathbf{L}^S(z^0)} \leq \nonumber \\
	\leq 
	\frac{ |F^{(K^0)}(z^0)|}{K^0!\mathbf{L}^{K^0}(z^0)} 
	p_0 	\frac{K!\prod_{j=1}^n \lambda^{n_0}_{2,j}(r) (n+\|S\|-1)!  \prod_{j=1}^n\Gamma(s_j/2+1)}{
		(K+S)! 	\Gamma(n+\|S\|/2)  r^{\|S\|}} \label{asymest}
	\end{gather}
	We choose $r> 1.$
	Since $\|K\|\le n_0$ the quantity 	$p_0K!\prod_{j=1}^n \lambda^{n_0}_{2,j}(R)$ does not depend of $S.$ Then there exists 
	$n_1$ such that 
	\begin{equation} \label{nongam1}
	\frac{p_0K!\prod_{j=1}^n \lambda^{n_0}_{2,j}(r)}{r^{\|S\|}}\leq 1 \text{ for all }\|S\|\ge n_1.
	\end{equation}
	The asymptotic behavior of $\frac{(n+\|S\|-1)!  \prod_{j=1}^n\Gamma(s_j/2+1)}{
		(K+S)! 	\Gamma(n+\|S\|/2)  r^{\|S\|}}$ is more difficult as $\|S\|\to +\infty.$
Using the Stirling formula $\Gamma(m+1)=\sqrt{2\pi m} \left(\frac{m}{e}\right)^m(1+\frac{\theta}{12m}),$ 
where $\theta=\theta(m)\in[0,1],$
we obtain 
\begin{gather}
\frac{(n+\|S\|-1)!  \prod_{j=1}^n\Gamma(s_j/2+1)}{(K+S)!\Gamma(n+\|S\|/2)  r^{\|S\|}}\leq 
\frac{(n+\|S\|-1)!  \prod_{j=1}^n\Gamma(s_j/2+1)}{S!\Gamma(n+\|S\|/2)  r^{\|S\|}}  \nonumber\\ 
= \frac{\sqrt{2\pi (n+\|S\|-1)} (\frac{n+\|S\|-1}{e})^{n+\|S\|-1} \prod_{j=1}^n \sqrt{2\pi s_j/2} (\frac{s_j}{2e})^{s_j/2}}
{\prod_{j=1}^n \sqrt{2\pi s_j} (\frac{s_j}{e})^{s_j} \sqrt{2\pi(n+\|S\|/2-1)} (\frac{n+\|S\|/2-1}{e})^{n+\|S\|/2-1} r^{\|S\|}} \times \nonumber\\
\times \frac{(1+\frac{\theta(n+\|S\|-1)}{12(n+\|S\|-1)}) \prod_{j=1}^n(1+\frac{\theta(s_j/2)}{12s_j/2})}
{(1+\frac{\theta(n+\|S\|/2)}{12(n+\|S\|/2)})\prod_{j=1}^n (1+\frac{\theta(s_j)}{12s_j})}.\nonumber
\end{gather}
Denoting 
$$\Theta(S) = 
\frac{(1+\frac{\theta(n+\|S\|-1)}{12(n+\|S\|-1)}) \prod_{j=1}^n(1+\frac{\theta(s_j/2)}{12s_j/2})}
{(1+\frac{\theta(n+\|S\|/2)}{12(n+\|S\|/2)})\prod_{j=1}^n (1+\frac{\theta(s_j)}{12s_j})}$$
and simplifying the previous inequality we deduce 
\begin{gather}
\frac{(n+\|S\|-1)!  \prod_{j=1}^n\Gamma(s_j/2+1)}{(K+S)!\Gamma(n+\|S\|/2)  r^{\|S\|}}\leq \nonumber\\
\leq \Theta(S)
\frac{2^{(1-n)/2}e^{-\|S\|/2}}{r^{\|S\|}}
\left(\frac{n-1+\|S\|}{n-1+\|S\|/2} \right)^{n-1+\|S\|/2} \cdot (n-1+\|S\|)^{\|S\|/2} \prod_{j=1}^n (\frac{e}{2s_j})^{s_j/2} \leq \nonumber \\
\leq  \Theta(S) \frac{2^{(n-1+\|S\|)/2}e^{-\|S\|/2}}{r^{\|S\|}}  (n-1+\|S\|)^{\|S\|/2} \prod_{j=1}^n (\frac{e}{2s_j})^{s_j/2}=\nonumber \\ 
= \Theta(S)\frac{2^{(n-1)/2}}{r^{\|S\|}} \left(1+\frac{n-1}{\|S\|}\right)^{\frac{\|S\|}{n-1}\cdot \frac{n-1}{2}}\cdot \|S\|^{\|S\|/2}
\prod_{j=1}^n \frac{1}{s_j^{s_j/2}} \le \nonumber \\ 
\le \Theta(S)(2e)^{(n-1)/2} \left(\frac{1}{r} \prod_{j=1}^n \left(\frac{\|S\|}{s_j}\right)^{\frac{s_j}{2\|S\|}}\right)^{\|S\|} \text{ as }
s_j\to \infty. \label{simplest}
\end{gather}

Denote $x_j=\frac{\|S\|}{s_j}\in(1,+\infty),$ $x=(x_1,\ldots,x_n).$ 
Obviously, $\Theta(S)\to 1$ as $s_j\to \infty,$ $j\in\{1,\ldots,n\}.$
Then \eqref{simplest} implies a constrained optimization problem 
	$$H(x):= \prod_{j=1}^n x_j^{1/(2x_j)}\to \max$$
	\begin{equation} \label{lagamma}
	\text{ subject to } \sum_{j=1}^n \frac{1}{x_j}=1, \ x_j\in(1,+\infty).
	\end{equation}
	If this problem has a solution, then $H(x)$ is not greater than some $H^*$ and we choose $r>H^*$ in \eqref{simplest}.
	
	Let us to introduce a Lagrange multiplier $\lambda$ and to study the Lagrange function $\mathcal{L}(x,\lambda)$ defined by
	$$
	\mathcal{L}(x,\lambda)= \prod_{j=1}^n x_j^{1/(2x_j)}+\lambda (\sum_{j=1}^n \frac{1}{x_j}-1).
	$$
	A necessary condition for optimality in constrained problems yield that
	$$
	\frac{\partial \mathcal{L}}{\partial x_j}=\frac{1-\ln{x_j}}{2x_j^2} \prod_{\substack{k=1}}^n x_k^{1/(2x_k)} 
	+\lambda (-\frac{1}{x_j^2})=0
	$$
	or 
	$$
	\frac{1-\ln x_j}{2}=\lambda/\prod_{\substack{k=1}}^n x_k^{1/(2x_k)}
	$$
	Hence, $x_j=\exp\left(1-2\lambda/\prod_{\substack{k=1}}^n x_k^{1/(2x_k)}\right),$ i.e. 
	$x_1=x_2=\ldots=x_n.$
	Constraint \eqref{lagamma} implies that 
	$$\sum_{j=1}^n \frac{1}{x_j}= \frac{n}{x_1}=1$$
	or $x_j=n$ for every $j\in\{1,\ldots,n\}.$ 
	Then $H(x)\leq \prod_{j=1}^n n^{1/(2n)}=\sqrt{n}.$
	
	We choose $r\ge \sqrt{n}. $ For this $r$ we have $\frac{1}{r} \prod_{j=1}^n \left(\frac{\|S\|}{s_j}\right)^{\frac{s_j}{2\|S\|}}\le 1.$
	In view of	\eqref{simplest} it means that there exist $n_2$ such that 
	\begin{equation} \label{gam1}
	\frac{(n+\|S\|-1)!  \prod_{j=1}^n\Gamma(s_j/2+1)}{(K+S)!\Gamma(n+\|S\|/2)  r^{\|S\|}} \leq 1
	\end{equation}
	for all $\|S\|\geq n_2.$
	
	The asymptotic behavior of right part \eqref{asymest} in other cases $S$ can be investigated similarly.
	Taking into account \eqref{asymest}, \eqref{nongam1} and \eqref{gam1} we have that for all $\|S\|\ge n_1+n_2$ 
	\begin{equation*}
	\frac{|F^{(K+S)}(z^0)|}{(K+S)!\mathbf{L}^{S+K}(z^0)}\leq \frac{ |F^{(K^0)}(z^0)|}{K^0!\mathbf{L}^{K^0}(z^0)}.
	\end{equation*}
	
	This means that for every $J\in \mathbb{Z}_+^n$
	\begin{gather*}
	\frac{|F^{(J)}(z^0)|}{J! \mathbf{L}^{J}(z^0)} \leq
	\max\left\{ \frac{|F^{(K)}(z^0)|}{K! {\mathbf{L}^{K}(z^0)}}: \|K\|\le n_0+n_1+n_2\right\}
	\end{gather*}
	where $n_0,$  $n_1,$ $n_2$ are independent of $z_0$. Therefore, the function $F$ has bounded $\mathbf{L}$-index in joint variables with $N(F,\mathbf{L},\mathbb{B}^n)\le n_0+n_1+n_2.$
\end{proof}

If we impose additional constraint by the function $\mathbf{L}$  then Theorem \ref{petr1ex-ball} implies the following criterion
\begin{theorem}
	Let $\mathbf{L} \in Q'(\mathbb{B}^n)$ be such that $\sup_{z\in\mathbb{B}^n} \frac{\mathcal{L}(z)}{\ell(z)}=C<
	\infty.$
	An analytic in $\mathbb{B}^n$ function $F$ has bounded $\mathbf{L}$-index in joint variables if and only if 
	for each
	$r\in(0,\beta]$ there exist $n_0\in \mathbb{Z}_+$, $p_0>0$ such that for every $z^0 \in\mathbb{B}^n$ there exists $K^0\in \mathbb{Z}_+^n$, $\|K^0\|\leq n_0$, such that inequality \eqref{net1ex-ball} holds.
\end{theorem}
\begin{proof}
	Sufficiency is proved in Theorem \ref{petr1ex-ball}. 
	As for necessity we choose $q= q(R)= [ 2(N+1) C r \prod_{j=1}^{n}(\lambda_{1,j}(r))^{-N}(\lambda_{2,j}(r))^{N+1}]+1$ 
	and replace $\mathcal{L}(z^0)$ by $\ell(z^0)$ in the proof of Theorem~\ref{petr1ex-ball}. No other changes. 
\end{proof}	

\begin{theorem}
	\label{ball-cor1ex}
	Let $\mathbf{L} \in Q'(\mathbb{B}^n).$ In order that an analytic in $\mathbb{B}^n$ function $F$ be of bounded $\mathbf{L}$-index in joint variables it is necessary that for every $r\in(0,\beta]$  $\exists n_0 \in \mathbb{Z}_+$ $\exists p\geq 1$ $\forall z^0 \in \mathbb{B}^n$ $\exists K^0 \in \mathbb{Z}_+^n $, $\|K^0\| \leq n_0,$ and
	\begin{gather}
	\max\left \{ |F^{(K^0)}(z)|:z \in \mathbb{B}^n\left[z^0, {r}/{\mathcal{L}(z^0)}\right] \right\} \leq p|F^{(K^0)}(z^0)|
	\label{conc1ex}
	\end{gather}
	and it is sufficient that for every $r\in(0,\beta]$ $\exists n_0 \in \mathbb{Z}_+$ $\exists p\geq 1$ $\forall z^0 \in \mathbb{B}^n$  $\forall j\in\{1, \ldots, n\}$ $\exists K^0_j=(0,\ldots,0, \underbrace{k^0_j}_{j\text{-th place}},0,\ldots,0)$  such that $k_j^0 \leq n_0$  and
	\begin{gather}
	\!\max \left\{ |F^{(K^0_j)}(z)|:  z \in \mathbb{B}^n\left[z^0, {r}/{\ell(z^0)}\right] \right\} \leq p|F^{(K_j^0)}(z^0)| \ \forall j\in\{1, \ldots, n\},\!
	\label{conc3ex}
	\end{gather}
\end{theorem}
\begin{proof}
	Proof of Theorem \ref{petr1ex-ball} implies that the inequality \eqref{net1ex}
	is true for some $K^0.$
	Therefore, we have
	\begin{gather*}
	\frac{p_0}{K^0!} \frac{|F^{(K^0)}(z^0)|}{\mathbf{L}^{K^0} (z^0)}
	\geq
	\max\left\{  \frac{|F^{(K^0)}(z)|}{K^0!\mathbf{L}^{K^0} (z)}: z\in \mathbb{B}^n \left[z^0,{r}/{\mathcal{L}(z^0)}\right] \right\} =
	\\
	= \max\left\{ \frac{|F^{(K^0)}(z)|}{K^0!} \frac{\mathbf{L}^{K^0} (z^0)}{ \mathbf{L}^{K^0} (z^0) \mathbf{L}^{K^0} (z)}: z\in \mathbb{B}^n \left[z^0,{r}/{\mathcal{L}(z^0)}\right] \right\} \geq
	\\
	\geq
	\max\left\{ \frac{|F^{(K^0)}(z)|}{K^0!} \frac{\prod_{j=1}^n{(\lambda_{2,j}(r))}^{-n_0}}{\mathbf{L}^{K^0} (z^0)}: z\in \mathbb{B}^n \left[z^0,{r}/{\mathcal{L}(z^0)}\right] \right\}.
	\end{gather*}
	This inequality implies
	\begin{gather}
	\frac{p_0\prod_{j=1}^n(\lambda_{2,j}(r))^{n_0}}{K^0!} \frac{|F^{(K^0)}(z^0)|}{\mathbf{L}^{K^0} (z^0)} 
	\geq
	\max\left\{ \frac{|F^{(K^0)}(z)|}{K^0!\mathbf{L}^{K^0} (z^0)}: z\in \mathbb{B}^n \left[z^0,{r}/{\mathcal{L}(z^0)}\right] \right\}.
	\label{conc2ex}
	\end{gather}
	From \eqref{conc2ex} we obtain inequality \eqref{conc1ex} with
	$p=p_0\prod_{j=1}^n{(\lambda_{2,j}(r))}^{n_0}$.
	The necessity of condition \eqref{conc1ex} is proved.
	
	Now we prove the sufficiency of \eqref{conc3ex}. Suppose that for every $r\in(0,\beta]$ $\exists n_0 \in \mathbb{Z}_+,$ $p>1$ such that $\forall z_0 \in \mathbb{B}^n $ and some
	$ K_j^0\in\mathbb{Z}^n_+$ with $k^0_j\leq n_0$ the inequality \eqref{conc3ex}  holds.

	$$
	\frac{F^{(K_J^0+S)}(z^0)}{S!}=\frac{1}{(2\pi i)^2} \int_{T^n\left(z^0,{R}/{\mathbf{L}(z^0)}\right)} \frac{F^{(K_J^0)}(z)}{(z-z^0)^{S+\mathbf{e}}} dz.
	$$
	In view of \eqref{cauchy-ball} we write Cauchy's formula as following $\forall z^0\in \mathbb{B}^n$  $\forall S \in \mathbb{Z}_+^n $
	$$F^{(K^0_j+S)}(z^0) = \frac{(n+\|S\|-1)!}{(n-1)!} \int_{\mathbb{S}^n(z^0,r/\ell(z^0))} \frac{(\overline{\xi-z^0})^SF^{(K^0_j)}(\xi)}{|\xi-z^0|^{2(n+\|S\|)-1}} d\sigma(\xi)$$
	
	As in \eqref{gamma-cauchy-ball}, this yields
	\begin{gather*}
	|F^{(K^0_j+S)}(z^0)| \leq \frac{(n+\|S\|-1)!}{(n-1)!} \left(\frac{\ell(z^0)}{r}\right)^{2(n+\|S\|)-1} \max\{|F^{(K^0_j)}(z)| \colon 
	z \in \mathbb{B}^n\left[z^0, {r}/{\ell(z^0)}\right]\} \times \\
	\times \int_{\mathbb{S}_n(z^0, {r}/{\ell(z^0)})}|(\overline{\xi-z^0})^S|d\sigma(\xi)\leq 
	\frac{(n+\|S\|-1)!}{(n-1)!} \left(\frac{\ell(z^0)}{r}\right)^{\|S\|} \times \\ \times \max\{|F^{(K^0_j)}(z)| \colon 
	z \in \mathbb{B}^n\left[z^0, {r}/{\ell(z^0)}\right]\}
	\int_{\mathbb{S}^n(z^0,r/\ell(z^0))} 
	\frac{|(\xi-z^0)^S|}{(r/\ell(z^0))^{\|S\|}} d\sigma\left(\frac{\xi-z^0}{r/\ell(z^0)}\right)\leq \\
	\leq \frac{(n+\|S\|-1)!}{(n-1)!} \left(\frac{\ell(z^0)}{r}\right)^{\|S\|} \max\{|F^{(K^0_j)}(z)| \colon 
	z \in \mathbb{B}^n\left[z^0, {r}/{\ell(z^0)}\right]\}
	\int_{\mathbb{S}^n(\mathbf{0},1)} |\xi^{S}| d\sigma(\xi) = \\ 
	= \frac{(n+\|S\|-1)!}{(n-1)!} \left(\frac{\ell(z^0)}{r}\right)^{\|S\|} \max\{|F^{(K^0_j)}(z)| \colon 
	z \in \mathbb{B}^n\left[z^0, {r}/{\ell(z^0)}\right]\}	 
	\frac{\Gamma(n)  \prod_{j=1}^n\Gamma(s_j/2+1)}{\Gamma(n+\|S\|/2)}
	\end{gather*}

	Now we put $r=\beta$ and use \eqref{conc3ex}
	\begin{gather}
	|F^{(K^0_j+S)}(z^0)| \leq 
	\left(\frac{\ell(z^0)}{\beta}\right)^{\|S\|}
	\frac{  (n+\|S\|-1)! \prod_{j=1}^n\Gamma(s_j/2+1)}{\Gamma(n+\|S\|/2)} 
	\times \nonumber \\ \times  \max\{|F^{(K^0_j)}(z)| \colon 
	z \in \mathbb{B}^n\left[z^0, {\beta}/{\ell(z^0)}\right]\}\leq \nonumber \\ \leq 
	p \left(\frac{\ell(z^0)}{\beta}\right)^{\|S\|}
	\frac{(n+\|S\|-1)!  \prod_{j=1}^n\Gamma(s_j/2+1)}{\Gamma(n+\|S\|/2)} 
	|F^{(K_j^0)}(z^0)|. \label{eqa1ex}
	\end{gather}
	Therefore \eqref{eqa1ex} implies for all $j\in\{1,\ldots,n\}$ and $k_j^0\leq n_0$ 
	\begin{gather*}
	\frac{|F^{(K_j^0+S)}(z^0)|}
	{{\mathbf{L}^{K_j^0+S}(z^0)}(K_j^0+S)!} \leq
	p \frac{K^0_j!(n+\|S\|-1)!  \prod_{j=1}^n\Gamma(s_j/2+1)}{\beta^{\|S\|} (K_j^0+S)!\Gamma(n+\|S\|/2)} 
	\frac{|F^{(K_j^0)}(z^0)|}{{\mathbf{L}^{K_j^0}(z^0)}K_j^0!} \leq \\ 
	\leq pn_0! \frac{(n+\|S\|-1)!  \prod_{j=1}^n\Gamma(s_j/2+1)}{\beta^{\|S\|} S!\Gamma(n+\|S\|/2)} 
	\frac{|F^{(K_j^0)}(z^0)|}{{\mathbf{L}^{K_j^0}(z^0)}K_j^0!}
	\end{gather*}
	In view of \eqref{gam1} there exists $n_1$ such that for all $\|S\|\geq n_1$ 
	$$ \frac{(n+\|S\|-1)!  \prod_{j=1}^n\Gamma(s_j/2+1)}{\beta^{\|S\|} S!\Gamma(n+\|S\|/2)} \leq 1.$$
	Obviously that there exists $n_2$ such that for all $\|S\|\geq n_2$ \ $\frac{ pn_0!}{\beta^{\|S\|}}\leq 1.$
	Consequently, we have
	$$ \frac{|F^{(K_j^0+S)}(z^0)|}
	{{\mathbf{L}^{K_j^0+S}(z^0)}(K_j^0+S)!} \leq
	\frac{|F^{(K_j^0)}(z^0)|}{{\mathbf{L}^{K_j^0}(z^0)}K_j^0!} \text{ for all } \|S\|\geq n_1+n_2
	$$
	i. e. $N(F,\mathbf{L},\mathbb{B}^n)\leq n_0+n_1+n_2.$
\end{proof}

\begin{remark}\rm
	Inequality \eqref{conc1ex} is a necessary and sufficient condition of boundedness of $l$-index for functions of one variable \cite{sher,kusher,sherkuz}. But it is unknown whether this condition is sufficient condition of boundedness of $\mathbf{L}$-index in joint variables. Our restrictions \eqref{conc3ex} are corresponding multidimensional sufficient conditions.
\end{remark}

\begin{lemma} \label{indexgreatlex}
		Let $\mathbf{L}_1, $ $\mathbf{L}_2$ be positive continuous functions in $\mathbb{B}^n$ and for every $z\in\mathbb{B}^n$ 
	$\mathbf{L}_1(z)\leq \mathbf{L}_2(z).$
If an analytic in $\mathbb{B}^n$ function $F$ has bounded $\mathbf{L}_1$-index in joint variables then 
	$F$ is of bounded $\mathbf{L}_2$-index in joint variables. 
	If, in addition, for every	 $z\in\mathbb{B}^n$  $\mathcal{L}_1(z)\leq \ell_2(z)$  then $N(F,\mathbf{L}_2,\mathbb{B}^n)\leq N(F,\mathbf{L}_1,\mathbb{B}^n).$
\end{lemma}	
\begin{proof}
	Let $N(F,\mathbf{L}_1,\mathbb{B}^n)=n_0.$
	Using \eqref{ineqoz2} we deduce 
	\begin{gather}
	\frac{|F^{(J)}(z)|}{J!\mathbf{L}^{J}_2(z)} = \frac{\mathbf{L}_1^J(z)}{\mathbf{L}_2^J(z)} \frac{|F^{(J)}(z)|}{J!\mathbf{L}^{J}_1(z)} 
	\leq  \frac{\mathbf{L}_1^J(z)}{\mathbf{L}_2^J(z)}  \max
	\left\{\frac{|F^{(K)}(z)|}{K! \mathbf{L}_1^{K}(z)}:\
	K\in\mathbb{Z}^{n}_{+},\ \|K\|\leq n_0\right\} \leq \nonumber \\
	\leq \frac{\mathbf{L}_1^J(z)}{\mathbf{L}_2^J(z)} 
	\max
	\left\{
	\frac{\mathbf{L}_2^K(z)}{\mathbf{L}_1^K(z)} 
	\frac{|F^{(K)}(z)|}{K! \mathbf{L}_2^{K}(z)}:\
	K\in\mathbb{Z}^{n}_{+},\ \|K\|\leq n_0\right\}\leq \nonumber \\
	\leq \max_{\|K\|\leq n_0} \left(\frac{\mathbf{L}_1(z)}{\mathbf{L}_2(z)}\right)^{J-K} 
	\max
	\left\{\frac{|F^{(K)}(z)|}{K! \mathbf{L}_2^{K}(z)}:\
	K\in\mathbb{Z}^{n}_{+},\ \|K\|\leq n_0\right\}. \label{ineqbigl}
	\end{gather} 
	Since $\mathbf{L}_1(z)\leq \mathbf{L}_2(z)$ it means that for all $\|J\|\geq n n_0$
	$$
	\frac{|F^{(J)}(z)|}{J!\mathbf{L}^{J}_2(z)} \leq 
	\max
	\left\{\frac{|F^{(K)}(z)|}{K! \mathbf{L}_2^{K}(z)}:\
	K\in\mathbb{Z}^{n}_{+},\ \|K\|\leq n_0\right\}.
	$$
	Thus, $F$ has bounded $\mathbf{L}_2$-index in joint variables.
	
	If, in addition, for every	 $z\in\mathbb{B}^n$  $\mathcal{L}_1(z)\leq \ell_2(z)$ 
	then for all $\|J\|\geq n_0$ \eqref{ineqbigl} yields  
	\begin{gather*}
	\frac{|F^{(J)}(z)|}{J!\mathbf{L}^{J}_2(z)} \leq 
	\max_{\|K\|\leq n_0} \left(\frac{\mathcal{L}_1(z)}{\ell_2(z)}\right)^{\|J-K\|} 	
	\max
	\left\{\frac{|F^{(K)}(z)|}{K! \mathbf{L}_2^{K}(z)}:\
	K\in\mathbb{Z}^{n}_{+},\ \|K\|\leq n_0\right\}\leq \\ 
	\leq \max
	\left\{\frac{|F^{(K)}(z)|}{K! \mathbf{L}_2^{K}(z)}:\
	K\in\mathbb{Z}^{n}_{+},\ \|K\|\leq n_0\right\}
	\end{gather*}
	and $N(F,\mathbf{L}_2,\mathbb{B}^n)\leq N(F,\mathbf{L}_1,\mathbb{B}^n).$
\end{proof}

Denote $\widetilde{\mathbf{L}}(z)=(\widetilde{l}_1(z),\ldots,\widetilde{l}_n(z))$. 
The notation $\mathbf{L}\asymp \widetilde{\mathbf{L}}$ means that there exist $\varTheta_1=(\theta_{1,j},\ldots,\theta_{1,n})\in \mathbb{R}_+^n,$ $\varTheta_2=(\theta_{2,j},\ldots,\theta_{2,n})\in \mathbb{R}_+^n$ such that  
$\forall z \in \mathbb{B}^n$
$\theta_{1,j}\widetilde{l}_j(z) \leq l_j(z)\leq \theta_{2,j}\widetilde{l}_j(z)$
for each $j\in\{1,\ldots,n\}.$

\begin{theorem} \label{petr2ex}
	Let $\mathbf{L} \in Q'(\mathbb{B}^n),$ $\mathbf{L}\asymp \widetilde{\mathbf{L}},$ 
	$\sup_{z\in\mathbb{B}^n} \frac{\mathcal{L}(z)}{\ell(z)}=C<
	\infty$ and
	$\min_{1\le j\le n} \theta_{1,j}>\frac{\sqrt{n}}{\beta}.$
	An analytic in $\mathbb{B}^n$  function $F$ has bounded $\widetilde{\mathbf{L}}$-index in joint variables if and only if it has bounded $\mathbf{L}$-index.
\end{theorem}
\begin{proof}
	It is easy to prove that if $\mathbf{L} \in Q'(\mathbb{B}^n)$ and $\mathbf{L}\asymp \widetilde{\mathbf{L}}$ then $\widetilde{\mathbf{L}} \in Q'(\mathbb{B}^n)$ with $\beta'=\beta \min_{1\le j\le n} \theta_{1,j} >\sqrt{n}$  instead of $\beta$ in \eqref{Lbeta-ball}. 
	
	Let $N(F,\widetilde{\mathbf{L}},\mathbb{B}^n)=\widetilde{n}_0<+\infty$. Then by Theorem \ref{petr1ex-ball} for every $\widetilde{r} \in(0,\beta)$ there exists $\widetilde{p}\ge 1$ such that for each $z^0 \in \mathbb{B}^n$ and some $K^0$ with  $\|K^0\|\leq \widetilde{n}_0,$ the inequality \eqref{net1ex} holds with $\widetilde{\mathbf{L}}$ and $\widetilde{r}$ instead of $\mathbf{L}$ and $r$.
Hence,
\begin{gather*}
\frac{\widetilde{p}}{K^0!} \frac{|F^{(K^0)}(z^0)|}{\mathbf{L}^{K^0} (z^0)}=
\frac{\widetilde{p}}{K^0!} \frac{\Theta_{2}^{K^0} |F^{(K^0)}(z^0)|}
{\Theta_{2}^{K^0} \mathbf{L}^{K^0} (z^0)} \geq
\frac{\widetilde{p}}{K^0!} \frac{|F^{(K^0)}(z^0)|}
{\Theta_{2}^{K^0} \widetilde{\mathbf{L}}^{K^0} (z^0)} \geq
\\ \geq
\frac{1}{\Theta_2^{K^0} }\max \left\{ \frac{|F^{(K)}(z)|}
{K!\widetilde{\mathbf{L}}^K(z)}: \|K\| \leq \widetilde{n}_0, z \in \mathbb{B}^n\left[z^0, {\widetilde{r}}/{\widetilde{\mathcal{L}}(z)}\right] \right\} \geq
\\ \geq
\frac{1}{\Theta_2^{K^0} }\max \left\{ \frac{\Theta_{1}^{K} |F^{(K)}(z)|}
{K!\mathbf{L}^{K}(z)}: \|K\| \leq \widetilde{n}_0, z \in \mathbb{B}^n\left[z^0, \min_{1\le j\le n} \Theta_{1,j}{\widetilde{r}}/{\mathcal{L}(z)}\right] \right\} \geq \\
\geq
\frac{\min \limits_{0\! \leq\! \|K\| \!\leq\! n_0
	} \{\Theta_{1}^{K} \}}{\Theta_2^{K^0} }
\max \left\{ \frac{|F^{(K)}(z)|}
{K!\mathbf{L}^{K}(z)}: \|K\| \leq\! \widetilde{n}_0, z \!\in\! \mathbb{B}^n\!\left[z^0, \min_{1\le j\le n} \Theta_{1,j}{\widetilde{r}}/{\mathcal{L}(z)}\right] \right\} \geq\\ 
\geq  
\frac{\min \limits_{0\! \leq\! \|K\| \!\leq\! n_0
	} \{\Theta_{1}^{K} \}}{\Theta_2^{K^0} }
\max \left\{ \frac{|F^{(K)}(z)|}
{K!\mathbf{L}^{K}(z)}: \|K\| \leq\! \widetilde{n}_0, z \!\in\! \mathbb{B}^n\!\left[z^0, \frac{{\widetilde{r}}\min_{1\le j\le n} \Theta_{1,j}}{C\ell(z)}\right] \right\}.
\end{gather*}
	In view of Theorem \ref{petr1ex-ball} we obtain that function $F$ has bounded $\mathbf{L}$-index.
\end{proof}
\begin{theorem}
	Let $\mathbf{L} \in Q'(\mathbb{B}^n),$ a function $F$ be analytic in $\mathbb{B}^n.$  If there exist $r\in(0,\beta],$ $n_0 \in \mathbb{Z}_+,$ $p_0>1$ such that for each $z^0 \in \mathbb{B}^n$ and for some $K^0 \in \mathbb{Z}_+^n$ with  $\|K^0\| \leq n_0$ the inequality \eqref{net1ex-ball} holds 
	then $F$ has bounded $\mathbf{L}$-index in joint variables.
\end{theorem}
\begin{proof}
	The proof of sufficiency in Theorem \ref{petr1ex-ball} with $r=\beta$ implies that $N(F,\mathbf{L},\mathbb{B}^n)<+\infty.$
	
	Let $\mathbf{L}^{*}(z)=\frac{r_0\mathbf{L}(z)}{r},$ $\ell^*(z)=\frac{r_0\ell(z)}{r},$ $r^0=\beta$ and $r$ is radius for which 
	\eqref{net1ex-ball} is true. 
	In a general case from validity of \eqref{net1ex-ball} for $F$ and $\mathbf{L}$ with $r<\beta$  we obtain
	\begin{gather*}
	\max \left\{ \frac{|F^{(K)}(z)|}
	{K!(\mathbf{L}^* (z))^{K}}: \|K\| \leq n_0, z \in \mathbb{B}^n\left[z^0, {r_0}/{\ell^*(z^0)}\right] \right\}  \leq \\
	\leq  \max \left\{ \frac{|F^{(K)}(z)|}
	{K!(r_0 \mathbf{L} (z)/r)^{K}}: \|K\| \leq n_0, z \in \mathbb{B}^n\left[z^0, {r_0}/{(r_0 \ell(z^0)/r)}\right] \right\} \leq \\
	\leq
	\max \left\{ \frac{|F^{(K)}(z)|}
	{K!\mathbf{L}^{K} (z)}: \|K\| \leq n_0, z\! \in\! \mathbb{B}^n\left[z^0, {r}/{\ell(z^0)}\right] \right\} \leq
	\\
	\leq 
	\frac{p_0}{K^0!} \frac{|F^{(K^0)}(z^0)|}{\mathbf{L}^{K^0} (z^0)}
	=
	\frac{\beta^{\|K^0\|}p_0}{r^{\|K^0\|}K^0!} \frac{|F^{(K^0)}(z)|}{(r_0 \mathbf{L} (z)/r)^{K^0}}
	= \frac{p_0\beta^{n_0}}{r^{n_0}} \frac{|F^{(K^0)}(z)|}{K^0!(\mathbf{L}^*(z))^{K^0}}.
	\end{gather*}
	i. e. \eqref{net1ex} holds for $F,$ $\mathbf{L}^*$ and $r_0=\beta.$
	As above now 
	we apply Theorem \ref{petr1ex-ball} to the function $F(z)$ and $\mathbf{L}^{*}(z)=r_0 \mathbf{L} (z)/r$. This implies that $F$ is of bounded $\mathbf{L}^{*}$-index in joint variables. Therefore, by Lemma~\ref{indexgreatlex} the function $F$ has bounded $\mathbf{L}$-index in joint variables.
\end{proof}

\section{Boundedness of $L$-index in joint variables of analytic solutions of systems of partial differential equations} \label{r4.3}
Using  Theorems \ref{boundedomjoin-ball} and \ref{bordte15} we obtain this corollary.
\begin{corollary}[\cite{bandsets}] \label{naslidokt2}
	Let $\mathbf{L}\in{Q}^n,$ $F(z)$ be an analytic in $\mathbb{B}^n$ function, $G$ be a bounded 
domain in $\mathbb{B}^{n}$ such that  $d=\inf_{z\in\overline{G}} (1-|z|)>0.$ The function $F(z)$ is 
of bounded $\mathbf{L}$-index in joint variables and only if there exist
	$p\in\mathbb{Z}_{+}$ and $C>0$ such that for all
	$z\in\mathbb{B}^{n}\setminus G$ the inequality \eqref{bordriv122} holds.
\end{corollary}

Let us to denote $a^+=\max\{a,0\},$  $u_j(t)=u_j(t,R,\Theta)=l_j(\frac{tR}{r^*}e^{i\Theta}),$   where $a\in\mathbb{R},$ 
$t\in[0,r^*],$ $j\in\{1,\ldots,n\},$  $r^*=\max_{1\leq j\leq n}r_j \neq 0$ that is $\frac{t}{r^*}|R|<1.$

Let $\mathbf{L}(Re^{i\Theta})$  be a positive continuously differentiable function in each variable $r_k,$ 
	 $k\in\{1,\ldots,n\},$ $|R|<1,$ $\Theta\in[0,2\pi]^n.$
By $W(\mathbb{B}^n)$ we denote the class of the functions $\mathbf{L}$ such that 
	$r^*(-(u_j(t,R,\Theta))'_t)^+/(r_j l_j^2(\frac{t}{r^*}Re^{i\Theta}))\to 0$ 	uniformly for all 
$\Theta\in[0,2\pi]^n,$ $j\in\{1,\ldots,n\},$ $t\in[0,r^*]$ as $|R|\to 1-0$ 

%

\begin{lemma}
If $\mathbf{L}(z)=(l_1(z),\ldots, l_{n}(z)),$ where every $l_j(z): \mathbb{B}^n\to \mathbb{R}_+$ is a continuous function 
 satisfying \eqref{Lbeta-ball} then 
	$$\max\limits_{\Theta\in[0,2\pi]^n} \int_0^{r^*}  \sum_{j=1}^n \frac{r_j}{r^*}l_j\left(\frac{\tau}{r^*}R e^{i\Theta}\right) d\tau\to +\infty 
\text{ as } |R|\to 1-0.$$ 
\end{lemma}	
\begin{proof}
Using \eqref{Lbeta-ball} we obtain  
\begin{gather*}
\max\limits_{\Theta\in[0,2\pi]^n} \int_0^{r^*}  \sum_{j=1}^n \frac{r_j}{r^*}l_j\left(\frac{\tau}{r^*}R e^{i\Theta}\right) d\tau \geq 
\int_0^{r^*} \sum_{j=1}^n \frac{r_j}{r^*} \frac{\beta}{1-\frac{\tau}{r^*}|R|}d\tau = \\ =
-\sum_{j=1}^n \frac{r_j\beta}{R} \ln (1-|R|)\to +\infty \text{ as } |R|\to 1-0.
\end{gather*}
\end{proof}

\begin{lemma}\label{lema43}
	Let $\mathbf{L}\in W(\mathbb{B}^n),$ $F$ be an analytic in $\mathbb{B}^n$ function.   
	If there exists $R'\in\mathbb{R}^n_+,$ $|R'|<1, $ and $p\in\mathbb{Z}_{+},$ $c>0$ such that for all  
	$z\in\mathbb{B}^n\setminus \mathbb{D}^n(\mathbf{0},R')$ inequality
	\eqref{bordriv122} 
	holds
	then
	\begin{equation} \label{jontgrowth}
	\varlimsup\limits_{|R|\to 1-0} 
	\frac{\ln \max\{|F(z)| \colon z\in\mathbb{T}^n (\mathbf{0},R)\}}{\max\limits_{\Theta\in[0,2\pi]^n} \int_0^{r^*}  \sum_{j=1}^n \frac{r_j}{r^*}l_j\left(\frac{\tau}{r^*}R e^{i\Theta}\right) d\tau}\leq c. 
	\end{equation}
\end{lemma}
\begin{proof}
	Let $R\in \mathbb{R}^n_+$ be such that $1>|R|>|R'|,$ $\Theta\in[0,2\pi]^n.$ Denote $\alpha_j=\frac{r_j}{r^*},$ $j\in\{1,\ldots,n\}$ and $A=(\alpha_1,\ldots,\alpha_n).$  
	We consider a function 
	\begin{equation*} 
	g(t)=\max\left\{\frac{|F^{(S)}(Ate^{i\Theta})|}{\mathbf{L}^S(Ate^{i\Theta})}:\ \|S\|\leq p \right\},
	\end{equation*}
	where  $Ate^{i\Theta}= (\alpha_1te^{i\theta_1},\ldots,\alpha_nte^{i\theta_n})$  and $|At|>|R'|.$
	
	Since the function $\frac{|F^{(S)}(Ate^{i\Theta})|}{\mathbf{L}^{S}(Ate^{i\Theta})}$ is  continuously differentiable by real $t\in[0,r^*],$
	outside the zero set of function $|F^{(S)}(Ate^{i\Theta})|,$
	the function $g(t)$ is a continuously differentiable function on $[0,r^*],$ except, perhaps, for a countable set of points.
	
	Therefore, using    the inequality
	$\frac{d}{dr} |g(r)|\leq |g'(r)|$  which holds except for the points $r=t$ such that $g(t)=0,$  we deduce 
	\begin{gather}
	\frac{d}{dt} \left(\frac{|F^{(S)}(Ate^{i\Theta})|}{\mathbf{L}^S(Ate^{i\Theta})}\right)= \frac{1}{\mathbf{L}^S(Ate^{i\Theta})} 
	\frac{d}{dt}|F^{(S)}(Ate^{i\Theta})|+\nonumber \\ +|F^{(S)}(Ate^{i\Theta})| \frac{d}{dt}\frac{1}{\mathbf{L}^S(Ate^{i\Theta})} 
	\leq  \frac{1}{\mathbf{L}^S(Ate^{i\Theta})} \left|\sum_{j=1}^n F^{(S+\mathbf{e}_j)}(Ate^{i\Theta}) \alpha_j e^{i\theta_j}\right|- \nonumber \\ - 
	\frac{|F^{(S)}(Ate^{i\Theta})|}{\mathbf{L}^S(Ate^{i\Theta})} 
	\sum_{j=1}^n  \frac{s_j u'_j(t)}{l_j(Ate^{i\Theta})} 
	\leq \sum_{j=1}^n  \frac{|F^{(S+\mathbf{e}_j)}(Ate^{i\Theta})|}{\mathbf{L}^{S\!+\!\mathbf{e}_j}(Ate^{i\Theta})} 
	\alpha_jl_j(Ate^{i\Theta})\!+\! \nonumber \\ +
	\frac{|F^{(S)}(Ate^{i\Theta})|}{\mathbf{L}^S(Ate^{i\Theta})} \sum_{j=1}^n  \frac{s_j (-u'_j(t))^+}{l_j(Ate^{i\Theta})} \label{derivindexh}
	\end{gather}
	For absolutely continuous functions
	$h_1,$ $h_2, $ $\ldots,$ $h_k$ and $h(x):=\max\{h_j(z): 1\leq j\leq k\},$
	\  $h'(x)\leq \max\{h'_j(x): 1\leq j\leq k \},$ $x\in[a,b]$  (see \cite[Lemma~4.1, p.~81]{sher}). The function $g$ is absolutely continuous, therefore, from 	\eqref{bordriv122}  and \eqref{derivindexh} it follows that
	\begin{gather*}
	g'(t) \leq \max \left\{\frac{d}{dt} \left(\frac{|F^{(S)}(Ate^{i\Theta})|}{\mathbf{L}^S(Ate^{i\Theta})}\right)\colon 
	\|S\|\leq p \right\} \leq \\ \leq 
	\max_{ \|S\|\leq p} \left\{ 
	\sum_{j=1}^n  \frac{\alpha_jl_j(Ate^{i\Theta})|F^{(S+\mathbf{e}_j)}(Ate^{i\Theta})|}{\mathbf{L}^{S+\mathbf{e}_j}(Ate^{i\Theta})} 
	+ 
	\frac{|F^{(S)}(Ate^{i\Theta})|}{\mathbf{L}^S(Ate^{i\Theta})} \sum_{j=1}^n  \frac{s_j (-u'_j(t))^+}{l_j(Ate^{i\Theta})}    \right\}\leq \\
	\leq g(t) \left( 
	c\sum_{j=1}^n \alpha_jl_j(Ate^{i\Theta})+ 
	\max_{ \|S\|\leq p}\left\{ \sum_{j=1}^n \frac{s_j(-u_j'(t))^+}{l_j(Ate^{i\Theta})}\right\}\right)=\\ =
	g(t)(\beta(t)+\gamma(t)),
	\end{gather*}
	where 
	$$ \beta(t)=c\sum_{j=1}^n \alpha_jl_j(Ate^{i\Theta}), 
	\gamma(t)=\max_{\|S\|\leq p}\left\{ \sum_{j=1}^n \frac{s_j(-u_j'(t))^+}{l_j(Ate^{i\Theta})} \right\}.$$
	Thus, 
	$\frac{d}{dt} \ln g(t)\leq \beta(t)+\gamma(t)$
	and 
	\begin{equation*}
	g(t)\leq g(t_0)\exp \int_{t_0}^t (\beta(\tau)+\gamma(\tau))d\tau, 
	\end{equation*} where $t_0$ is chosen such that $g(t_0)\neq 0.$
	The condition	$\mathbf{L}\in W(\mathbb{B}^n)$ gives 
	\begin{gather*}
	\frac{\gamma(t)}{\beta(t)}= \frac{\sum_{j=1}^n \frac{s_j(-u_j'(t))^+}{l_j(Ate^{i\Theta})}}{c\sum_{j=1}^n \alpha_jl_j(Ate^{i\Theta})}\leq 
	p \sum_{j=1}^n \frac{(-u_j'(t))^+}{\alpha_jl^2_j(Ate^{i\Theta})}\leq 
	p \varepsilon,
	\end{gather*}
	where $\varepsilon=\varepsilon(R) \to 0$ uniformly for all 
	$\Theta\in[0,2\pi]^n,$ $t\in[r_0,r^*]$ as $|R|\to 1-0$ 
	
	But 
	$|F(Ate^{i\Theta})|\leq g(t)\leq g(t_0) \exp \int_{t_0}^t (\beta(\tau)+\gamma(\tau))d\tau$ 
	and $r^*A=R.$ Then we put $t=r^*$ 
	and obtain 
	\begin{gather*}
	\ln \max\{|F(z)\colon \ z\in \mathbb{T}^n(\mathbf{0},R)\}=\ln \max_{\Theta\in[0,2\pi]^n}|F(Re^{i\Theta})| \leq \ln \max_{\Theta\in[0,2\pi]^n} g(r^*) \leq 
	\\ \leq 
	\ln g(t_0)+
	\max_{\Theta\in[0,2\pi]^n} \int_{t_0}^{r^*} (\beta(\tau)+\gamma(\tau))d\tau   \leq\\ 
	\leq \ln g(t_0)+   \max_{\Theta\in[0,2\pi]^n} \int_{t_0}^{r^*} 
	c\sum_{j=1}^n \alpha_jl_j(A\tau e^{i\Theta}) \left(1+ p\varepsilon\right)d\tau = \\ 
	=\ln g(t_0)+c\max_{\Theta\in[0,2\pi]^n} \int_{t_0}^{r^*}  \sum_{j=1}^n \frac{r_j}{r^*}l_j\left(\frac{\tau}{r^*}R e^{i\Theta}\right) \left(1+ p\varepsilon\right) d\tau.
	\end{gather*}
	This implies \eqref{jontgrowth}.
\end{proof}
\begin{lemma} \label{lema53}
	Let $\mathbf{L}\in W(\mathbb{B}^n),$ 
	$F$ be an analytic in $\mathbb{B}^n$ function.   
	If there exists $R'\in\mathbb{R}^n_+,$ $|R'|<1$ and $p\in\mathbb{Z}_{+},$ $c>0$ such that for all  
	$z\in\mathbb{B}^n\setminus \mathbb{D}^n(\mathbf{0},R')$ inequality
	\begin{equation}  \label{lemriv3}
	\max\left\{\frac{|F^{(J)}(z)|}{J!\mathbf{L}^J(z)}: \ \|J\|=p+1 \right\}\leq c\cdot   \max\left\{\frac{|F^{(K)}(z)|}{K!\mathbf{L}^K(z)}: \ \|K\|\leq p \right\}.
	\end{equation}
	holds
	then
	\begin{equation} \label{jontgrowth2}
	\varlimsup\limits_{|R|\to \infty} 
	\frac{\ln \max\{|F(z)| \colon z\in\mathbb{T}^n (\mathbf{0},R)\}}{\max\limits_{\Theta\in[0,2\pi]^n} \int_0^{r^*}  \sum_{j=1}^n \frac{r_j}{r^*}l_j\left(\frac{\tau}{r^*}R e^{i\Theta}\right) d\tau}\leq c(p+1)
	\end{equation}
	%
	%
	%
\end{lemma}
\begin{proof}
	The proof of Lemma \ref{lema53} is similar to the proof of Lemma \ref{lema43}.
	
	Let $R\in \mathbb{R}^n_+$ be such that $1>|R|>|R'|,$ $\Theta\in[0,2\pi]^n.$ Denote $\alpha_j=\frac{r_j}{r^*},$ $j\in\{1,\ldots,n\}$ and $A=(\alpha_1,\ldots,\alpha_n).$  
	We consider a function 
	\begin{equation*} 
	g(t)=\max\left\{\frac{|F^{(S)}(Ate^{i\Theta})|}{S!\mathbf{L}^S(Ate^{i\Theta})}:\ \|S\|\leq p \right\},
	\end{equation*}
	where $At=(\alpha_1t,\ldots,\alpha_nt),$ $Ate^{i\Theta}= (\alpha_1te^{i\theta_1},\ldots,\alpha_nte^{i\theta_n})$  and $|At|>|R'|.$
	
	As above the function $\frac{|F^{(S)}(Ate^{i\Theta})|}{S!\mathbf{L}^{S}(Ate^{i\Theta})}$ is  continuously differentiable by real $t\in[0,r^*],$
	outside the zero set of function $|F^{(S)}(Ate^{i\Theta})|,$
	the function $g(t)$ is a continuously differentiable function on $[0,r^*],$ except, perhaps, for a countable set of points.
	
	Therefore, using    the inequality
	$\frac{d}{dr} |g(r)|\leq |g'(r)|$  which holds except for the points $r=t$ such that $g(t)=0,$  we deduce 
	\begin{gather}
	\frac{d}{dt} \left(\frac{|F^{(S)}(Ate^{i\Theta})|}{S!\mathbf{L}^S(Ate^{i\Theta})}\right)= \frac{1}{S!\mathbf{L}^S(Ate^{i\Theta})} 
	\frac{d}{dt}|F^{(S)}(Ate^{i\Theta})|+\nonumber \\ +|F^{(S)}(Ate^{i\Theta})| \frac{d}{dt}\frac{1}{S!\mathbf{L}^S(Ate^{i\Theta})} 
	\leq  \frac{1}{S!\mathbf{L}^S(Ate^{i\Theta})} \left|\sum_{j=1}^n F^{(S+\mathbf{e}_j)}(Ate^{i\Theta}) \alpha_j e^{i\theta_j}\right|- \nonumber \\ \!-\!
	\frac{|F^{(S)}(Ate^{i\Theta})|}{S!\mathbf{L}^S(Ate^{i\Theta})} 
	\sum_{j=1}^n  \frac{s_j u'_j(t)}{l_j(Ate^{i\Theta})} 
	\leq \sum_{j=1}^n  \frac{|F^{(S+\mathbf{e}_j)}(Ate^{i\Theta})|}{(S\!+\!\mathbf{e}_{j})!\mathbf{L}^{S\!+\!\mathbf{e}_j}(Ate^{i\Theta})} 
	\alpha_j(s_j\!+\!1)l_j(Ate^{i\Theta})\!+\! \nonumber \\ +
	\frac{|F^{(S)}(Ate^{i\Theta})|}{S!\mathbf{L}^S(Ate^{i\Theta})} \sum_{j=1}^n  \frac{s_j (-u'_j(t))^+}{l_j(Ate^{i\Theta})} \label{derivindex2}
	\end{gather}
	For absolutely continuous functions
	$h_1,$ $h_2, $ $\ldots,$ $h_k$ and $h(x):=\max\{h_j(z): 1\leq j\leq k\},$
	\  $h'(x)\leq \max\{h'_j(x): 1\leq j\leq k \},$ $x\in[a,b]$  (see \cite[Lemma~4.1, p.~81]{sher}). The function $g$ is absolutely continuous. Therefore, 	\eqref{bordriv122}  and \eqref{derivindex2} yield 
	\begin{gather*}
	g'(t) \leq \max \left\{\frac{d}{dt} \left(\frac{|F^{(S)}(Ate^{i\Theta})|}{S!\mathbf{L}^S(Ate^{i\Theta})}\right)\colon 
	\|S\|\leq N \right\} \leq \\ \leq 
	\max_{ \|S\|\leq p} \left\{ 
	\sum_{j=1}^n  \frac{\alpha_j(s_j+1)l_j(Ate^{i\Theta})|F^{(S+\mathbf{e}_j)}(Ate^{i\Theta})|}{(S+\mathbf{e}_{j})!\mathbf{L}^{S+
			\mathbf{e}_j}(Ate^{i\Theta})} 
	+ \right. \\ \left.+
	\frac{|F^{(S)}(Ate^{i\Theta})|}{S!\mathbf{L}^S(Ate^{i\Theta})} \sum_{j=1}^n  \frac{s_j (-u'_j(t))^+}{l_j(Ate^{i\Theta})}    \right\}\leq \\
	\leq g(t) \left( c\max_{ \|S\|\leq p}\left\{ 
	\sum_{j=1}^n \alpha_j(s_j+1)l_j(Ate^{i\Theta})\right\}+ 
	\max_{ \|S\|\leq p}\left\{ \sum_{j=1}^n \frac{s_j(-u_j'(t))^+}{l_j(Ate^{i\Theta})}\right\}\right)=\\ =
	g(t)(\beta(t)+\gamma(t)),
	\end{gather*}
	where 
	$$ \beta(t)=c\max_{\|S\|\leq p}\left\{
	\sum_{j=1}^n \alpha_j(s_j+1)l_j(Ate^{i\Theta}) \right\}, 
	\gamma(t)=\max_{\|S\|\leq p}\left\{ \sum_{j=1}^n \frac{s_j(-u_j'(t))^+}{l_j(Ate^{i\Theta})} \right\}.$$
	Thus, 
	$\frac{d}{dt} \ln g(t)\leq \beta(t)+\gamma(t)$
	and 
	\begin{equation*}
	g(t)\leq g(t_0)\exp \int_{t_0}^t (\beta(\tau)+\gamma(\tau))d\tau, 
	\end{equation*} where $t_0$ is chosen such that $g(t_0)\neq 0.$
	Denote $\widetilde{\beta}(t)= \sum_{j=1}^n \alpha_jl_j(Ate^{i\Theta}).$   
	Since $\mathbf{L}\in W(\mathbb{B}^n),$  for some $S^*,$ $\|S^*\|\leq p$ and $\widetilde{S},$ $\|\widetilde{S}\|\leq p,$ we obtain 
	\begin{gather*}
	\frac{\gamma(t)}{\widetilde{\beta}(t)}= \frac{\sum_{j=1}^n \frac{s^*_j(-u_j'(t))^+}{l_j(Ate^{i\Theta})}}{\sum_{j=1}^n \alpha_jl_j(Ate^{i\Theta})}\leq 
	\sum_{j=1}^n s^*_j\frac{(-u_j'(t))^+}{\alpha_jl^2_j(Ate^{i\Theta})}\leq 
	p\sum_{j=1}^n  \frac{(-u_j'(t))^+}{\alpha_jl^2_j(Ate^{i\Theta})} \leq p\varepsilon  \text{ and } \\
	\frac{\beta(t)}{\widetilde{\beta}(t)}=\frac{ c\sum_{j=1}^n \alpha_j(\tilde{s}_j+1)l_j(Ate^{i\Theta})}{\sum_{j=1}^n \alpha_jl_j(Ate^{i\Theta})} =
	c+c\frac{ \sum_{j=1}^n \alpha_j\tilde{s}_jl_j(Ate^{i\Theta})}{\sum_{j=1}^n \alpha_jl_j(Ate^{i\Theta})} \leq \\ 
	\leq c+c
	\sum_{j=1}^n \tilde{s}_j\leq c(1+ p),
	\end{gather*}
	where $\varepsilon=\varepsilon(R) \to 0$ uniformly for all 
	$\Theta\in[0,2\pi]^n,$ $t\in[r_0,r^*]$ as $|R|\to 1-0$ 
	
	But 
	$|F(Ate^{i\Theta})|\leq g(t)\leq g(t_0) \exp \int_{t_0}^t (\beta(\tau)+\gamma(\tau))d\tau$ 
	and $r^*A=R.$ Then we put $t=r^*$ 
	and obtain 
	\begin{gather*}
	\ln \max\{|F(z)\colon \ z\in \mathbb{T}^n(\mathbf{0},R)\}=\ln \max_{\Theta\in[0,2\pi]^n}|F(Re^{i\Theta})| \leq \ln \max_{\Theta\in[0,2\pi]^n} g(r^*) \leq 
	\\ \leq 
	\ln g(t_0)+
	\max_{\Theta\in[0,2\pi]^n} \int_{t_0}^{r^*} (\beta(\tau)+\gamma(\tau))d\tau   \leq\\ 
	\leq \ln g(t_0)+   \max_{\Theta\in[0,2\pi]^n} \int_{t_0}^{r^*} 
	\sum_{j=1}^n \alpha_jl_j(A\tau e^{i\Theta}) \left(c(1+p)+ p\varepsilon\right)d\tau = \\ 
	=\ln g(t_0)+\max_{\Theta\in[0,2\pi]^n} \int_{t_0}^{r^*}  \sum_{j=1}^n \frac{r_j}{r^*}l_j\left(\frac{\tau}{r^*}R e^{i\Theta}\right) \left(c(1+p)+ p\varepsilon\right) d\tau.
	\end{gather*}
	This implies \eqref{jontgrowth2}.
\end{proof}
Using proved lemmas we will formulate and prove propositions that provide  growth estimates of analytic solutions of the following system of partial differential equations:
\begin{equation} \label{dvakhvylka}
G_{p_j\mathbf{e}_j}(z)F^{(p_j\mathbf{e}_j)}(z)+
\sum_{\|S_j\|\le p_j-1} G_{S_j}(z)F^{(S_j)}(z)=H_j(z), \ j\in\{1,\ldots,n\}
\end{equation}
$p_j\in \mathbb{N},$ $S_j\in\mathbb{Z}^n_+,$ $H_j$ and $G_{S_j}$ are analytic in $\mathbb{B}^n$ functions.
\begin{theorem} \label{teorema1}  \renewcommand{\labelenumi}{\arabic{enumi})}
	Let $\mathbf{L}\in W(\mathbb{B}^n)\cap Q(\mathbb{B}^n)$ and for all $z\in\mathbb{B}^n\setminus \mathbb{D}^n(\mathbf{0},R')$ analytic in $\mathbb{B}^n$ functions $H_j$ and $G_{S_j}$ satisfy the following conditions:
	\begin{enumerate}
		\item $ \left|G^{(M)}_{S_j}(z)\right|\leq B_{S_j,M} \mathbf{L}^{p_j\mathbf{e}_j-S_j+M}(z) |G_{p_j\mathbf{e}_j}(z)|$ 
		and $ \left|G^{(M)}_{p_j\mathbf{e}_j}(z)\right|< B_{p_j\mathbf{e}_j,M} \mathbf{L}^{M}(z) |G_{p_j\mathbf{e}_j}(z)|$ 
		for every $j\in\{1,\ldots,n\},$  $\|S_j\|\leq p_j-1,$ $\mathbf{0}\leq M\leq I$ 
		and  $\|I\|=1-p_j+ \sum_{k=1}^n p_k=1+\sum_{\stackrel{k=1}{k\neq j}}^n p_k$ 
		\item $ \left| H_j^{(I)}(z) \right|\leq D_{I,j} \mathbf{L}^{I}(z)|H_j(z)|$ for every $j\in\{1,\ldots,n\}$ and for all $\|I\|=1-p_j+ \sum_{k=1}^n p_k=1+\sum_{\stackrel{k=1}{k\neq j}}^n p_k.$
	\end{enumerate}
	where $B_{S_j,M}$ and $D_{I,j}$ are nonnegative constants,  and $B_{p_j\mathbf{e}_j,M}$ are positive constants.  If an analytic in $\mathbb{B}^n$ function $F(z)$ satisfies \eqref{dvakhvylka}  
	then $F$ has bounded $\mathbf{L}$-index in joint variables and 
	\begin{equation} \label{kashtan}
	\varlimsup\limits_{|R|\to 1-0} 
	\frac{\ln \max\{|F(z)| \colon z\in\mathbb{T}^n (\mathbf{0},R)\}}{\max\limits_{\Theta\in[0,2\pi]^n} \int_0^{r^*}  \sum_{j=1}^n \frac{r_j}{r^*}l_j\left(\frac{\tau}{r^*}R e^{i\Theta}\right) d\tau}\leq c,
	\end{equation}
	where 
	$$
	c\!=\!\max_{\stackrel{\|I\|=1-p_j+ \sum_{k=1}^n p_k,}{j\in\{1,\ldots,n\}}} \!\!\left(D_{I,j} (1+\sum_{\|S_j\|\le p_j-1} B_{S_j,\mathbf{0}})\!+\! \sum_{\stackrel{\mathbf{0}\leq M\leq I}{M\neq \mathbf{0}}}	C^M_{I} B_{p_j\mathbf{e}_j,M}\!+\! 
	\sum_{\mathbf{0}\leq M\leq I} 	C_{I}^M \sum_{\|S_j\|\le p_j-1} B_{S_j,M}\! \right).\!
	$$
\end{theorem}
\begin{proof}
	First, we note that the first condition of the theorem yields $G_{p_j\mathbf{e}_j}(z)\neq 0$ for $z\in\mathbb{B}^n\setminus \mathbb{D}^n(\mathbf{0},R').$
	Taking into account that  the function $F(z)$ satisfies  system  (\ref{dvakhvylka}),  we calculate the partial derivative $I\in\mathbb{Z}^n_+$  in each equation of the system
	\begin{gather}
	\sum_{\mathbf{0}\leq M\leq I}	C^M_{I} G^{(M)}_{p_j\mathbf{e}_j}(z)F^{(p_j\mathbf{e}_j+I-M)}(z)+
	\sum_{\mathbf{0}\leq M\leq I} 	C_{I}^M \sum_{\|S_j\|\le p_j-1} G^{(M)}_{S_j}(z)F^{(S_j+I-M)}(z)=H_j^{(I)}(z),	\label{zirochka}
	\end{gather}
	where $C^M_I=\frac{i_1!\ldots i_n!}{m_1!(i_1-m_1)!\ldots m_n!(i_n-m_n)!}$ 
	and  $\|I\|=1-p_j+ \sum_{k=1}^n p_k=1+\sum_{\stackrel{k=1}{k\neq j}}^n p_k.$
	Using the second condition of the theorem, we obtain
	\begin{gather} 
	\!	|H_j^{(I)}(z)| \!\le\! D_{I,j} \mathbf{L}^{I}(z)|H_j(z)|\le \nonumber \\ \le 
	D_{I,j} \mathbf{L}^{I}(z)\left(|G_{p_j\mathbf{e}_j}(z)||F^{(p_j\mathbf{e}_j)}(z)|+
	\sum_{\|S_j\|\le p_j-1} |G_{S_j}(z)||F^{(S_j)}(z)|\right).\label{trykutnyk} \end{gather}
	Equation (\ref{zirochka}) yields 
	\begin{gather}
	F^{(p_j\mathbf{e}_j+I)}(z)=\frac{1}{G_{p_j\mathbf{e}_j}(z)}	
	\left(
	H_j^{(I)}(z)- 
	\sum_{\stackrel{\mathbf{0}\leq M\leq I}{M\neq \mathbf{0}}}	C^M_I G^{(M)}_{p_j\mathbf{e}_j}(z)F^{(p_j\mathbf{e}_j+I-M)}(z)-
	\right.  \nonumber\\ \left.-
	\sum_{\mathbf{0}\leq M\leq I} 	C_{I}^M \sum_{\|S_j\|\le p_j-1} G^{(M)}_{S_j}(z)F^{(S_j+I-M)}(z)
	\right).
	\label{dvizirochky}
	\end{gather}
	From (\ref{dvizirochky}) and the first condition it follows  
	\begin{gather}
	\left|F^{(p_j\mathbf{e}_j+I)}(z)\right|\!=\!\frac{1}{|G_{p_j\mathbf{e}_j}(z)|}\!	
	\left(
	D_{I,j} \mathbf{L}^{I}(z)\left(|G_{p_j\mathbf{e}_j}(z)||F^{(p_j\mathbf{e}_j)}(z)|\!+\!
	\sum_{\|S_j\|\le p_j-1} |G_{S_j}(z)||F^{(S_j)}(z)|\right)\!+\!\right.\nonumber\\ 
	\left.+
	\sum_{\stackrel{\mathbf{0}\leq M\leq I}{M\neq \mathbf{0}}}	C^M_{I} |G^{(M)}_{p_j\mathbf{e}_j}(z)| |F^{(p_j\mathbf{e}_j+I-M)}(z)|+
	\sum_{\mathbf{0}\leq M\leq I} 	C_{I}^M \sum_{\|S_j\|\le p_j-1} |G^{(M)}_{S_j}(z)| |F^{(S_j+I-M)}(z)|
	\right)\leq \nonumber\\ 
	\leq
	D_{I,j} \mathbf{L}^{I}(z)\left(|F^{(p_j\mathbf{e}_j)}(z)|+
	\sum_{\|S_j\|\le p_j-1} B_{S_j,\mathbf{0}} L^{p_j\mathbf{e}_j-S_j}(z) |F^{(S_j)}(z)|\right)+\nonumber\\ 	
	+ 
	\sum_{\stackrel{\mathbf{0}\leq M\leq I}{M\neq \mathbf{0}}}	C^M_{I} B_{p_j\mathbf{e}_j,M} \mathbf{L}^{M}(z) |F^{(p_j\mathbf{e}_j+I-M)}(z)|+\nonumber\\+
	\sum_{\mathbf{0}\leq M\leq I} 	C_{I}^M \sum_{\|S_j\|\le p_j-1} B_{S_j,M} \mathbf{L}^{p_j\mathbf{e}_j-S_j+M}(z) |F^{(S_j+I-M)}(z)| \label{dopestimates}
	\end{gather}
	
	Dividing this inequality by $L^{p_j\mathbf{e}_j+I}(z),$ we obtain that for every $\|I\|=1+\sum_{\stackrel{k=1}{k\neq j}}^n p_k$ and $j\in\{1,\ldots,n\}$
	\begin{gather*}
	\frac{\left|F^{(p_j\mathbf{e}_j+I)}(z)\right|}{\mathbf{L}^{p_j\mathbf{e}_j+I}(z)}\leq 
	D_{I,j} \left( \frac{|F^{(p_j\mathbf{e}_j)}(z)|}{\mathbf{L}^{p_j\mathbf{e}_j}(z)}+
	\sum_{\|S_j\|\le p_j-1} B_{S_j,\mathbf{0}} \frac{|F^{(S_j)}(z)|}{\mathbf{L}^{S_j}(z)}\right)+\\ 	
	+ 
	\sum_{\stackrel{\mathbf{0}\leq M\leq I}{M\neq \mathbf{0}}}	C^M_{I} B_{p_j\mathbf{e}_j,M} \frac{|F^{(p_j\mathbf{e}_j+I-M)}(z)|}{\mathbf{L}^{p_j\mathbf{e}_j+I-M}(z)}+
	\sum_{\mathbf{0}\leq M\leq I} 	C_{I}^M \sum_{\|S_j\|\le p_j-1} B_{S_j,M} \frac{|F^{(S_j+I-M)}(z)|}{\mathbf{L}^{S_j+I-M}(z)}\leq \\ 
	\leq \left(D_{I,j} (1+\sum_{\|S_j\|\le p_j-1} B_{S_j,\mathbf{0}})+ \sum_{\stackrel{\mathbf{0}\leq M\leq I}{M\neq \mathbf{0}}}	C^M_{I} B_{p_j\mathbf{e}_j,M}+ 
	\sum_{\mathbf{0}\leq M\leq I} 	C_{I}^M \sum_{\|S_j\|\le p_j-1} B_{S_j,M} \right) \times \\ 
	\times \max\left\{ \frac{|F^{(S)}(z)|}{\mathbf{L}^{S}(z)}: \|S\|\le \sum_{j=1}^n p_j\right\}.
	\end{gather*}
	Obviously, $\|p_j\mathbf{e}_j+I\|=1+\sum_{j=1}^n p_j.$
	This implies 
	\begin{gather*}
	\max\left\{\frac{\left|F^{(K)}(z)\right|}{\mathbf{L}^{K}(z)}: 
	\|K\|=1+\sum_{j=1}^n p_j\right\} \leq 
	c\cdot   \max\left\{\frac{|F^{(S)}(z)|}{\mathbf{L}^S(z)}: \  \|S\|\le \sum_{j=1}^n p_j \right\},
	\end{gather*}
	where 
	$$
	c\!=\!\max_{\stackrel{\|I\|=1-p_j+ \sum_{k=1}^n p_k,}{j\in\{1,\ldots,n\}}} \!\left(D_{I,j} (1\!+\!\sum_{\|S_j\|\le p_j-1} B_{S_j,\mathbf{0}})+ \sum_{\stackrel{\mathbf{0}\leq M\leq I}{M\neq \mathbf{0}}}	C^M_{I} B_{p_j\mathbf{e}_j,M}+ 
	\sum_{\mathbf{0}\leq M\leq I} 	C_{I}^M \sum_{\|S_j\|\le p_j-1} B_{S_j,M} \!\right)\! 
	$$
	for all  $z\in\mathbb{C}^n\setminus \mathbb{D}^n(\mathbf{0},R').$
	
	Thus, by Lemma \ref{lema43}  estimate (\ref{kashtan}) holds, and by Corollary \ref{naslidokt2} the analytic in $\mathbb{B}^n$ function $F$ has bounded  $\mathbf{L}$-index in joint variables.
\end{proof}

If system (\ref{dvakhvylka}) is homogeneous $(H_j(z)\equiv 0),$ the previous theorem can be simplified.
\begin{theorem} \label{teorema2}
	Let $\mathbf{L}\in W(\mathbb{B}^n)\cap Q(\mathbb{B}^n)$ and for all $z\in\mathbb{C}^n\setminus \mathbb{D}^n(\mathbf{0},R')$ analytic in $\mathbb{B}^n$ functions $G_{S_j}$
	satisfy the conditions
	$ \left|G^{(M)}_{S_j}(z)\right|\!\leq\! B_{S_j,M} \mathbf{L}^{p_j\mathbf{e}_j-S_j+M}(z) |G_{p_j\mathbf{e}_j}(z)|$ 
	and $ \left|G^{(M)}_{p_j\mathbf{e}_j}(z)\right|\!<\! B_{p_j\mathbf{e}_j,M} \mathbf{L}^{M}(z) |G_{p_j\mathbf{e}_j}(z)|$ 
	for every $j\in\{1,\ldots,n\},$  $\|S_j\|\leq p_j-1,$ $\mathbf{0}\leq M\leq I$ 
	and  $\|I\|=-p_j+ \sum_{k=1}^n p_k=\sum_{\stackrel{k=1}{k\neq j}}^n p_k,$
	where $B_{S_j,M}$ are some nonnegative constants, $B_{p_j\mathbf{e}_j,M}$ are positive constants. If an analytic in $\mathbb{B}^n$ function $F$ is a solution of system
	\eqref{dvakhvylka} with $H_j(z)\equiv 0$  for every $j\in\{1,\ldots,n\}$ 
	then $F$ has bounded $\mathbf{L}$-index in joint variables and 
	\begin{equation} \label{kashtan0}
	\varlimsup\limits_{|R|\to \infty} 
	\frac{\ln \max\{|F(z)| \colon z\in\mathbb{T}^n (\mathbf{0},R)\}}{\max\limits_{\Theta\in[0,2\pi]^n} \int_0^{r^*}  \sum_{j=1}^n \frac{r_j}{r^*}l_j\left(\frac{\tau}{r^*}R e^{i\Theta}\right) d\tau}\leq c,
	\end{equation}
	where 
	$$
	c=\max_{\stackrel{\|I\|=-p_j+ \sum_{k=1}^n p_k,}{j\in\{1,\ldots,n\}}} \left(\sum_{\stackrel{\mathbf{0}\leq M\leq I}{M\neq \mathbf{0}}}	C^M_{I} B_{p_j\mathbf{e}_j,M}+ 
	\sum_{\mathbf{0}\leq M\leq I} 	C_{I}^M \sum_{\|S_j\|\le p_j-1} B_{S_j,M}\! \right).
	$$
\end{theorem}
\begin{proof}
	If $H_j(z)\equiv 0$ then \eqref{dvizirochky} implies
	\begin{gather}
	F^{(p_j\mathbf{e}_j+I)}(z)=\frac{1}{G_{p_j\mathbf{e}_j}(z)}	
	\left(- 
	\sum_{\stackrel{\mathbf{0}\leq M\leq I}{M\neq \mathbf{0}}}	C^M_I G^{(M)}_{p_j\mathbf{e}_j}(z)F^{(p_j\mathbf{e}_j+I-M)}(z)-
	\right.  \nonumber\\ \left.-
	\sum_{\mathbf{0}\leq M\leq I} 	C_{I}^M \sum_{\|S_j\|\le p_j-1} G^{(M)}_{S_j}(z)F^{(S_j+I-M)}(z)
	\right).
	\label{dvizirochky0}
	\end{gather}
	Hence, we obtain 
	\begin{gather*}
	|F^{(p_j\mathbf{e}_j+I)}(z)|\leq \frac{1}{|G_{p_j\mathbf{e}_j}(z)|}	
	\left(\sum_{\stackrel{\mathbf{0}\leq M\leq I}{M\neq \mathbf{0}}}	C^M_I |G^{(M)}_{p_j\mathbf{e}_j}(z)| |F^{(p_j\mathbf{e}_j+I-M)}(z)|+
	\right.  \nonumber\\ \left.+
	\sum_{\mathbf{0}\leq M\leq I} 	C_{I}^M \sum_{\|S_j\|\le p_j-1} |G^{(M)}_{S_j}(z)| |F^{(S_j+I-M)}(z)|
	\right).
	\end{gather*} 	
	
	Dividing the obtained inequality by $\mathbf{L}^{p_j\mathbf{e}_j+I}(z)$ and using assumptions  of the theorem  on the functions $G_{S_j},$ we deduce
	\begin{gather*}
	\frac{|F^{(p_j\mathbf{e}_j+I)}(z)|}{\mathbf{L}^{p_j\mathbf{e}_j+I}(z)}\leq 
	\frac{1}{|G_{p_j\mathbf{e}_j}(z)|\mathbf{L}^{p_j\mathbf{e}_j+I}(z)}	
	\left( \sum_{\stackrel{\mathbf{0}\leq M\leq I}{M\neq \mathbf{0}}}	C^M_I B_{p_j\mathbf{e}_j,M} \mathbf{L}^{M}(z) |G_{p_j\mathbf{e}_j}(z)| |F^{(p_j\mathbf{e}_j+I-M)}(z)|+
	\right.  \nonumber\\ \left.+
	\sum_{\mathbf{0}\leq M\leq I} 	C_{I}^M \sum_{\|S_j\|\le p_j-1} B_{S_j,M} \mathbf{L}^{p_j\mathbf{e}_j-S_j+M}(z) |G_{p_j\mathbf{e}_j}(z)| |F^{(S_j+I-M)}(z)|
	\right)=\\
	= \sum_{\stackrel{\mathbf{0}\leq M\leq I}{M\neq \mathbf{0}}}	C^M_I B_{p_j\mathbf{e}_j,M} \frac{|F^{(p_j\mathbf{e}_j+I-M)}(z)|}{\mathbf{L}^{p_j\mathbf{e}_j+I-M}(z)}+
	\sum_{\mathbf{0}\leq M\leq I} 	C_{I}^M \sum_{\|S_j\|\le p_j-1} B_{S_j,M} \frac{|F^{(S_j+I-M)}(z)|}{\mathbf{L}^{S_j+I-M}(z)}\leq \\
	\!\leq\! \left(\sum_{\stackrel{\mathbf{0}\leq M\leq I}{M\neq \mathbf{0}}}	C^M_{I} B_{p_j\mathbf{e}_j,M}\!+\! 
	\sum_{\mathbf{0}\leq M\leq I} 	C_{I}^M \sum_{\|S_j\|\le p_j-1} B_{S_j,M} \right) \times
	\max\left\{ \frac{|F^{(S)}(z)|}{\mathbf{L}^{S}(z)}: \|S\|\le -1+\sum_{j=1}^n p_j\right\}.\!
	\end{gather*}
	Obviously, $\|p_j\mathbf{e}_j+I\|=\sum_{j=1}^n p_j.$
	Therefore, 
	\begin{gather*}
	\max\left\{\frac{\left|F^{(K)}(z)\right|}{\mathbf{L}^{K}(z)}: 
	\|K\|=\sum_{j=1}^n p_j\right\} \leq 
	c\cdot   \max\left\{\frac{|F^{(S)}(z)|}{\mathbf{L}^S(z)}: \  \|S\|\le -1+\sum_{j=1}^n p_j \right\},
	\end{gather*}
	where 
	$$
	c\!=\!\max_{\stackrel{\|I\|=-p_j+ \sum_{k=1}^n p_k,}{j\in\{1,\ldots,n\}}} \!\left(\sum_{\stackrel{\mathbf{0}\leq M\leq I}{M\neq \mathbf{0}}}	C^M_{I} B_{p_j\mathbf{e}_j,M}+ 
	\sum_{\mathbf{0}\leq M\leq I} 	C_{I}^M \sum_{\|S_j\|\le p_j-1} B_{S_j,M} \!\right)\! 
	$$
	for all  $z\in\mathbb{B}^n\setminus \mathbb{D}^n(\mathbf{0},R').$
	
	Thus, all conditions of Corollary \ref{naslidokt2} are satisfying. Hence, the function $F$ has bounded $\mathbf{L}$-index in joint variables and by Lemma  \ref{lema43} estimate (\ref{kashtan0}) holds.
\end{proof}

Note that estimate \eqref{kashtan} and \eqref{kashtan0} cannot be improved (see examples for $n=1$ in \cite{bordulyakgrowth}). 


Moreover, using Corollary \ref{naslidokt2} and Lemma \ref{lema53} we can supplement two previous Theorems \ref{teorema1} and \ref{teorema2} with propositions, that contain estimates  $\max\{|F(z)| \colon z\in\mathbb{T}^n (\mathbf{0},R)\},$ which can sometimes be better than (\ref{kashtan0}) and (\ref{kashtan}).
Two following theorems have similar proofs to Theorems \ref{teorema1} and \ref{teorema2}.

\begin{theorem} \label{teorema3}
	Let $\mathbf{L}\in W(\mathbb{B}^n)\cap Q(\mathbb{B}^n)$ and for all $z\in\mathbb{C}^n\setminus \mathbb{D}^n(\mathbf{0},R')$ analytic in $\mathbb{B}^n$ functions $H_j$ and $G_{S_j}$ satisfy the following conditions:
	\begin{enumerate}
		\item[1)] $ \left|G^{(M)}_{S_j}(z)\right|\leq B_{S_j,M} \mathbf{L}^{p_j\mathbf{e}_j-S_j+M}(z) |G_{p_j\mathbf{e}_j}(z)|$ 
		and $ \left|G^{(M)}_{p_j\mathbf{e}_j}(z)\right|< B_{p_j\mathbf{e}_j,M} \mathbf{L}^{M}(z) |G_{p_j\mathbf{e}_j}(z)|$ 
		for every $j\in\{1,\ldots,n\},$  $\|S_j\|\leq p_j-1,$ $\mathbf{0}\leq M\leq I$ 
		and  $\|I\|=1-p_j+ \sum_{k=1}^n p_k=1+\sum_{\stackrel{k=1}{k\neq j}}^n p_k$ 
		\item[2)] $ \left| H_j^{(I)}(z) \right|\leq D_{I,j} \mathbf{L}^{I}(z)|H_j(z)|$ for every $j\in\{1,\ldots,n\}$ and for all $\|I\|=1-p_j+ \sum_{k=1}^n p_k=1+\sum_{\stackrel{k=1}{k\neq j}}^n p_k.$
	\end{enumerate}
	where $B_{S_j,M}$ and $D_{I,j}$ are nonnegative constants,  and $B_{p_j\mathbf{e}_j,M}$ are positive constants.  If an analytic in $\mathbb{B}^n$ function $F(z)$ satisfies \eqref{dvakhvylka}  
	then $F$ has bounded $\mathbf{L}$-index in joint variables and 
	\begin{equation} \label{kashtan3}
	\varlimsup\limits_{|R|\to 1-0} 
	\frac{\ln \max\{|F(z)| \colon z\in\mathbb{T}^n (\mathbf{0},R)\}}{\max\limits_{\Theta\in[0,2\pi]^n} \int_0^{r^*}  \sum_{j=1}^n \frac{r_j}{r^*}l_j\left(\frac{\tau}{r^*}R e^{i\Theta}\right) d\tau}\leq c,
	\end{equation}
	where 
	\begin{gather}
	B=\max\{B_{S_j,M}, B_{p_j\mathbf{e}_j,M}\colon j\in\{1,\ldots,n\}, \mathbf{0}\le M\le I, \|I\|=1+\sum_{\stackrel{k=1}{k\neq j}}^n p_k\},\nonumber\\
	c= \!\max_{\stackrel{\|I\|=1-p_j+ \sum_{k=1}^n p_k,}{j\in\{1,\ldots,n\}}} \left(D_{I,j} \left(\frac{p_j!}{(p_j\mathbf{e}_j+I)!}+
	B\sum_{\|S_j\|\le p_j-1} \frac{(p_j\mathbf{e}_j-S_j)!}{(p_j\mathbf{e}_j+I)!}\right)+\right.\nonumber\\ 
	\left.
	+ B\sum_{\stackrel{\mathbf{0}\leq M\leq I}{M\neq \mathbf{0}}}	C^M_{I}  \frac{(p_j\mathbf{e}_j+I-M)!|}{(p_j\mathbf{e}_j+I)!}+B\sum_{\mathbf{0}\leq M\leq I} 	C_{I}^M \sum_{\|S_j\|\le p_j-1}  \frac{(S_j+I-M)!}{(p_j\mathbf{e}_j+I)!} \right) \label{definc}
	\end{gather}
\end{theorem}
\begin{proof}
	As in proof of Theorem \ref{teorema1},
	dividing \eqref{dopestimates} by $(p_j\mathbf{e}_j+I)!L^{p_j\mathbf{e}_j+I}(z),$ we obtain that for every $\|I\|=1+\sum_{\stackrel{k=1}{k\neq j}}^n p_k$ and $j\in\{1,\ldots,n\}$
	\begin{gather*}
	\frac{|F^{(p_j\mathbf{e}_j+I)}(z)|}{(p_j\mathbf{e}_j+I)!\mathbf{L}^{p_j\mathbf{e}_j+I}(z)}	\leq\\ \leq
	D_{I,j} \left(\frac{|F^{(p_j\mathbf{e}_j)}(z)|}{(p_j\mathbf{e}_j+I)!\mathbf{L}^{p_j\mathbf{e}_j}(z)}+
	\sum_{\|S_j\|\le p_j-1} B_{S_j,\mathbf{0}} \frac{|F^{(S_j)}(z)|}{(p_j\mathbf{e}_j+I)!\mathbf{L}^{p_j\mathbf{e}_j-S_j}(z)}\right)+\nonumber\\ 	
	+ 
	\sum_{\stackrel{\mathbf{0}\leq M\leq I}{M\neq \mathbf{0}}}	C^M_{I} B_{p_j\mathbf{e}_j,M} \frac{|F^{(p_j\mathbf{e}_j+I-M)}(z)|}{(p_j\mathbf{e}_j+I)!\mathbf{L}^{p_j\mathbf{e}_j+I-M}(z)}+\nonumber\\+
	\sum_{\mathbf{0}\leq M\leq I} 	C_{I}^M \sum_{\|S_j\|\le p_j-1} B_{S_j,M} \frac{|F^{(S_j+I-M)}(z)|}{(p_j\mathbf{e}_j+I)!\mathbf{L}^{S_j+I-M}(z)}\leq \\
	\leq 			D_{I,j} \left(\frac{|F^{(p_j\mathbf{e}_j)}(z)|}{(p_j\mathbf{e}_j+I)!\mathbf{L}^{p_j\mathbf{e}_j}(z)}+
	B\sum_{\|S_j\|\le p_j-1} \frac{|F^{(S_j)}(z)|}{(p_j\mathbf{e}_j+I)!\mathbf{L}^{p_j\mathbf{e}_j-S_j}(z)}\right)+\nonumber \\
	+ B\sum_{\stackrel{\mathbf{0}\leq M\leq I}{M\neq \mathbf{0}}}	C^M_{I}  \frac{|F^{(p_j\mathbf{e}_j+I-M)}(z)|}{(p_j\mathbf{e}_j+I)!\mathbf{L}^{p_j\mathbf{e}_j+I-M}(z)}+
	B\sum_{\mathbf{0}\leq M\leq I} 	C_{I}^M \sum_{\|S_j\|\le p_j-1}  \frac{|F^{(S_j+I-M)}(z)|}{(p_j\mathbf{e}_j+I)!\mathbf{L}^{S_j+I-M}(z)}\leq \\ 
	\leq \left(D_{I,j} \left(\frac{p_j!}{(p_j\mathbf{e}_j+I)!}+
	B\sum_{\|S_j\|\le p_j-1} \frac{(p_j\mathbf{e}_j-S_j)!}{(p_j\mathbf{e}_j+I)!}\right)
	+ B\sum_{\stackrel{\mathbf{0}\leq M\leq I}{M\neq \mathbf{0}}}	C^M_{I}  \frac{(p_j\mathbf{e}_j+I-M)!|}{(p_j\mathbf{e}_j+I)!}+\right.\nonumber\\ 
	\left.+B\sum_{\mathbf{0}\leq M\leq I} 	C_{I}^M \sum_{\|S_j\|\le p_j-1}  \frac{(S_j+I-M)!}{(p_j\mathbf{e}_j+I)!} \right) 
	\max\left\{ \frac{|F^{(S)}(z)|}{\mathbf{L}^{S}(z)}: \|S\|\le \sum_{j=1}^n p_j\right\}.\!				
	\end{gather*}
	where $B=\max\{B_{S_j,M}, B_{p_j\mathbf{e}_j,M}\colon j\in\{1,\ldots,n\}, \mathbf{0}\le M\le I, \|I\|=1+\sum_{\stackrel{k=1}{k\neq j}}^n p_k.\}$
	
	Obviously, $\|p_j\mathbf{e}_j+I\|=1+\sum_{j=1}^n p_j.$
	For all  $z\in\mathbb{C}^n\setminus \mathbb{D}^n(\mathbf{0},R')$ it implies 
	\begin{gather*}
	\max\left\{\frac{\left|F^{(K)}(z)\right|}{K!\mathbf{L}^{K}(z)}: 
	\|K\|=1+\sum_{j=1}^n p_j\right\} \leq 
	c\cdot   \max\left\{\frac{|F^{(S)}(z)|}{S!\mathbf{L}^S(z)}: \  \|S\|\le \sum_{j=1}^n p_j \right\},
	\end{gather*}
	where $c$ is defined in \eqref{definc}.
	
	In view of Corollary \ref{naslidokt2} the analytic in $\mathbb{B}^n$ function $F$ has bounded  $\mathbf{L}$-index in joint variables.
	And by Lemma \ref{lema53}  estimate \eqref{kashtan3} holds.

\end{proof}
By analogy to the proofs of Theorems \ref{teorema2} and \ref{teorema3} it can be proved the following assertion. 
\begin{theorem} \label{teorema4}
	Let $\mathbf{L}\in W(\mathbb{B}^n)\cap Q(\mathbb{B}^n)$ and for all $z\in\mathbb{C}^n\setminus \mathbb{D}^n(\mathbf{0},R')$ analytic in $\mathbb{B}^n$ functions $G_{S_j}$
	satisfy the conditions
	$ \left|G^{(M)}_{S_j}(z)\right|\!\leq\! B_{S_j,M} \mathbf{L}^{p_j\mathbf{e}_j-S_j+M}(z) |G_{p_j\mathbf{e}_j}(z)|$ 
	and $ \left|G^{(M)}_{p_j\mathbf{e}_j}(z)\right|\!<\! B_{p_j\mathbf{e}_j,M} \mathbf{L}^{M}(z) |G_{p_j\mathbf{e}_j}(z)|$ 
	for every $j\in\{1,\ldots,n\},$  $\|S_j\|\leq p_j-1,$ $\mathbf{0}\leq M\leq I$ 
	and  $\|I\|=-p_j+ \sum_{k=1}^n p_k=\sum_{\stackrel{k=1}{k\neq j}}^n p_k,$
	where $B_{S_j,M}$ are some nonnegative constants, $B_{p_j\mathbf{e}_j,M}$ are positive constants. If an analytic in $\mathbb{B}^n$ function $F$ is a solution of system
	\eqref{dvakhvylka} with $H_j(z)\equiv 0$  for every $j\in\{1,\ldots,n\}$ 
	then $F$ has bounded $\mathbf{L}$-index in joint variables and 
	\begin{equation*} 
	\varlimsup\limits_{|R|\to \infty} 
	\frac{\ln \max\{|F(z)| \colon z\in\mathbb{T}^n (\mathbf{0},R)\}}{\max\limits_{\Theta\in[0,2\pi]^n} \int_0^{r^*}  \sum_{j=1}^n \frac{r_j}{r^*}l_j\left(\frac{\tau}{r^*}R e^{i\Theta}\right) d\tau}\leq c,
	\end{equation*}
	where 
	\begin{gather*}
	B=\max\{B_{S_j,M}, B_{p_j\mathbf{e}_j,M}\colon j\in\{1,\ldots,n\}, \mathbf{0}\le M\le I, \|I\|=\sum_{\stackrel{k=1}{k\neq j}}^n p_k\},\nonumber \\
	c= B\!\max_{\stackrel{\|I\|=-p_j+ \sum_{k=1}^n p_k,}{j\in\{1,\ldots,n\}}} 
	\left(\sum_{\stackrel{\mathbf{0}\leq M\leq I}{M\neq \mathbf{0}}}	C^M_{I}  \frac{(p_j\mathbf{e}_j+I-M)!|}{(p_j\mathbf{e}_j+I)!}+\sum_{\mathbf{0}\leq M\leq I} 	C_{I}^M \sum_{\|S_j\|\le p_j-1}  \frac{(S_j+I-M)!}{(p_j\mathbf{e}_j+I)!} \right). \label{definch}
	\end{gather*}
\end{theorem}
\begin{remark}
	We should like to note that the obtained propositions are new even for analytic in a disc functions.
%
\end{remark}

For example, if $n=1$ then system \eqref{dvakhvylka} reduces to the following differential equation
\begin{equation} \label{dvakhvylkan1}
g_{p}(z)f^{(p)}(z)+
\sum_{j-0}^{p-1} g_{j}(z)f^{(j)}(z)=h(z),
\end{equation}
where  $h$ and $g_j$ are analytic in $\mathbb{D}$ functions.
Then Theorem \ref{teorema1} implies the corollary for $n=1.$
\begin{corollary}   \renewcommand{\labelenumi}{\arabic{enumi})}
	Let $l\in W(\mathbb{D})\cap Q(\mathbb{D})$ and for all $z\in\mathbb{C}$ such that $|z|>r'$ analytic in $\mathbb{D}$ functions $h$ and $g_{j}$ satisfy the following conditions 
	\begin{enumerate}
		\item $ \left|g^{(m)}_{j}(z)\right|\leq B_{j,m}l^{p-j+m}(z) |g_{p}(z)|$ 
		and $ \left|g'_{p}(z)\right|< B_{p,1}l^{m}(z) |g_{p}(z)|$ 
		for every $j\in\{1,\ldots,p-1\},$ $m\in\{0,1\},$ 
		\item $ \left| h'(z) \right|\leq D l(z)|h(z)|,$
	\end{enumerate}
	where $B_{j,m}$ and $D$ are nonnegative constants,  and $B_{p,1}$ is positive constant.  If an analytic in $\mathbb{D}$  function $f$ satisfies \eqref{dvakhvylkan1}  
	then $f$ has bounded $l$-index and 
	\begin{equation*} 
	\varlimsup\limits_{r\to \infty} 
	\frac{\ln \max\{|f(z)| \colon |z|=r \}}{\max\limits_{\theta\in[0,2\pi]} \int_0^{r} l\left(\tau e^{i\theta}\right) d\tau}\leq c,
	\end{equation*}
	where 
	$c\!=D (1+\sum_{j=0}^{p-1} B_{j,0})+ B_{p,1}\!+\! 
	\sum_{m=0}^{1} \sum_{j=0}^{p-1} B_{j,m}.$
\end{corollary}

{\footnotesize

   Department of Advanced Mathematics
   
   Ivano-Frankivs'k National Technical University of Oil and Gas
   
   andriykopanytsia@gmail.com
   
   Department of Function Theory and Theory of Probability
   
   Ivan Franko National University of Lviv
   
   olskask@gmail.com
}


\end{document}